\documentclass[11pt]{elsarticle}
\usepackage[margin=.7in]{geometry}
\usepackage{graphicx} 
\usepackage{amsfonts}
\usepackage{amsmath}
\usepackage{amssymb}
\usepackage{amsthm}
\usepackage{mathrsfs}
\usepackage{dsfont}
 \newcommand{\R}{\mathbb {R}}
 \newcommand{\E}{\mathbb {E}}
 
 \renewcommand{\P}{\mathbb{P}}
 
\usepackage{autobreak}
\usepackage[hidelinks]{hyperref}
\usepackage[noabbrev,capitalize]{cleveref}
\crefname{equation}{equation}{equations}
\numberwithin{equation}{section}

\newtheorem{theorem}{Theorem}[section]
\newtheorem{lemma}[theorem]{Lemma}

\newtheorem{remark}[theorem]{Remark}
\newtheorem{definition}[theorem]{Definition}
\newtheorem{proposition}[theorem]{Proposition}

\newcommand{\dd}{\mathrm{d}}
\begin{document}
\begin{frontmatter}

\title{Stochastic Tamed Navier--Stokes Equations with Wiener and Jump Noise on $\mathbb R^3$. \\ II.
Global $L^p$ Well-Posedness}

\author{Bikram Podder} \ead{bikrampoddar2@gmail.com} \author{Surendra Kumar} \ead{surendraiitr8@gmail.com}

\affiliation{organization={Department of Mathematics}, addressline={University of Delhi}, city={ Delhi}, postcode={110007}, country={India}}
\begin{abstract}
We establish intrinsic continuation criteria and finite-energy global solvability for the stochastic tamed Navier--Stokes equations on $\mathbb R^3$ driven simultaneously by multiplicative cylindrical Wiener and compensated Poisson noise.
Under local coefficient hypotheses, the maximal local $L^p$ solution, $p>3$, satisfies a blow-up alternative independent of auxiliary cutoffs and a Serrin criterion with time exponent $2p/(p-3)$.
Additional coercivity of the taming term and compatible $L^2$ and gradient noise bounds yield the global well-posedness theory for divergence-free initial data $u_0\in L^p(\Omega;L^p)\cap L^2(\Omega;L^2)$.
No smallness or initial $H^1$ regularity is required.
The solution has c\`adl\`ag $L^p\cap L^2$ paths and gains $H^1$ regularity at positive times, with time-weighted $H^1$ and $H^2$ estimates.
The proof combines finite-energy persistence with endpoint completion that retains terminal Poisson jumps and permits restart in the original uniqueness class.
\end{abstract}
\end{frontmatter}
\tableofcontents
\section{Introduction}
\label{sec:introduction}

We study continuation and global solvability for the stochastic tamed Navier--Stokes equations on $\mathbb R^3$, driven simultaneously by multiplicative cylindrical Wiener noise and a compensated Poisson random measure.
For a viscosity $\nu>0$, the system is
\begin{equation}
\label{eq:intro-main-system}
\begin{cases}
\begin{aligned}
\dd u(t)=& \Bigl[\nu\Delta u(t) -\mathcal P\bigl((u(t)\cdot\nabla)u(t)\bigr) -\mathcal P\bigl(g_N(|u(t)|^2)u(t)\bigr)\Bigr]\,\dd t \\
&+\mathcal P\sigma(t,u(t-))\,\dd W_t +\displaystyle\int_Z \mathcal PG(t,u(t-),z)\,\widetilde N(\dd t,\dd z),
\end{aligned}
\\
\nabla\cdot u(t)=0,\qquad (t,x)\in(0,\infty)\times\R^3, \\
u(0)=u_0.
\end{cases}
\end{equation}
Here $\mathcal P$ denotes the Helmholtz--Leray projection, $W$ is a cylindrical Wiener process on a separable Hilbert space $\mathcal U$, and $N$ is a Poisson random measure on a $\sigma$-finite measure space $(Z,\mathcal Z,\mu)$, with compensator $\dd t\,\mu(\dd z)$.
We write $\widetilde N=N-\dd t\,\mu$ for its compensated version.
The taming parameter $N$ is fixed, and the map $v\mapsto g_N(|v|^2)v$ has at most cubic growth.
The stochastic basis, coefficient measurability, and joint compatibility of the driving noises are specified in \cref{Section 2}.
Dependence of the coefficients on $\omega$ is suppressed, and the stochastic integrands are represented through predictable left limits.

For $p>3$, let $E=L^p_\sigma(\R^3;\R^3)$, where the subscript $\sigma$ denotes the solenoidal subspace.
The companion Paper I \cite{PodderKumar2026} provides a pathwise unique maximal local strong solution $\bigl(u,(\tau_n)_{n\ge1},\tau_{\max}\bigr)$ for $u_0\in L^p(\Omega,\mathcal F_0;E)$, with compatible stopped c\`adl\`ag representatives in the local $L^p$ energy class.
The question addressed here is how to continue this same solution through conditions expressed in its original $L^p$ topology, and whether an additional finite-energy assumption on the initial datum ensures globality under suitable taming and noise hypotheses.
To extend the solution, we need a continuation criterion based purely on the solution's intrinsic properties, rather than the auxiliary cutoffs used in its construction.
Crucially, this approach must preserve both the prescribed stochastic basis and the original uniqueness class.

This passage from local $L^p$ solvability to global energy control has distinct analytic and probabilistic difficulties.
On $\R^3$, membership in $L^p$ implies neither finite $L^2$ energy nor $H^1$ regularity, while $H^1$ controls the $L^p$ norm only for $2\le p\le6$.
A global Hilbert-space estimate must therefore be connected to the continuation requirements of the existing $L^p$ solution, including the range $p>6$.
Discontinuous forcing also changes the endpoint argument: an exit can occur by a jump whose post-jump value exceeds the localizing level, although the noise coefficient is evaluated at the controlled pre-jump state.
Continuation requires a closed solution branch with an admissible terminal datum, so estimates before the endpoint must be reconciled with the stopped equation through that endpoint.

We address this problem by combining intrinsic $L^p$ continuation criteria with finite-energy persistence and positive-time regularization.
The resulting global theorem applies to $L^p\cap L^2$ initial velocities for every $p>3$, under the stated gradient bounds on the noise, and continues the canonical solution of Paper I \cite{PodderKumar2026}.
The point of comparison with existing Hilbert-space and local $L^p$ theories is thus the preservation of this solution and its native continuation class throughout the global argument.

\subsection{Related literature}

The global analysis of the tamed Navier--Stokes equations rests on the additional dissipation supplied by the taming term.
R\"ockner and Xicheng Zhang \cite{Rockner2009} developed the deterministic theory on $\R^3$, including existence, uniqueness and regularity, and related the tamed approximations to suitable weak solutions of the classical equations.
Their stochastic theory \cite{MR2520127} established global strong solvability from $H^1$ initial data under multiplicative Wiener forcing.
Further developments include the treatment of Dirichlet boundary conditions by R\"ockner and Tusheng Zhang \cite{MR2852224} and the whole-space construction of Brze\'zniak and Dhariwal \cite{MR4085355}.
For discontinuous forcing, Dong and Zhang \cite{MR4127295} proved existence and uniqueness for the three-dimensional tamed system with multiplicative L\'evy noise and periodic boundary conditions, with c\`adl\`ag paths in the solenoidal $H^1$ space.
More recently, Lu, Yang and Li \cite{LuYangLi2025} studied well-posedness and averaging for Wiener-driven tamed equations under locally weak monotonicity assumptions.
These results establish the role of taming in global stochastic solvability and provide the principal $H^1$ precedents for the present work.

A related global theory concerns Navier--Stokes equations with nonlinear damping of the form $|u|^{\beta-1}u$.
The deterministic results of Cai and Jiu \cite{MR2401535} and Zhang, Wu and Lu \cite{MR2754840} describe how such damping affects existence and uniqueness in three dimensions.
For stochastic convective Brinkman--Forchheimer equations, Mohan \cite{MR4439993} obtained global pathwise unique strong solutions from finite-energy data under pure jump forcing on bounded and periodic domains, together with higher regularity in the periodic setting under additional assumptions.
Kinra and Mohan \cite{KinraMohanUnbounded2025} developed a global theory with multiplicative Gaussian noise on general unbounded domains under the corresponding damping and coefficient conditions.
In a complementary whole-space $L^p$ setting, Mohan \cite{MR4314117} constructed local mild solutions for deterministic and L\'evy-driven convective Brinkman--Forchheimer equations.
Thus global finite-energy solvability for stochastic fluid equations with damping, and local $L^p$ solvability with jump forcing, both have substantial precedents.

Abstract variational theories offer another approach to these equations.
Agresti and Veraar \cite{MR4716343} developed a critical variational framework based on local Lipschitz estimates and coercivity.
Their applications include global well-posedness of the tamed Navier--Stokes equations on $\R^3$ for $H^1$ initial data with gradient-dependent Wiener forcing.
Bechtel, Germ and Veraar \cite{MR5120871} extended this framework to L\'evy-driven equations, establishing maximal local solutions, blow-up criteria and global well-posedness under suitable coercivity hypotheses.
Their applications include tamed Navier--Stokes equations, and the framework does not require a compact embedding of the energy space into the pivot space.
It therefore also applies on unbounded domains.
These results are particularly relevant when comparing the present $L^p$ continuation theory with existing global theories in Hilbert spaces.

Continuation and regularization have also been developed through stochastic maximal regularity.
Agresti and Veraar \cite{MR4437443} proved blow-up criteria, including criteria of Serrin type, and instantaneous regularization for nonlinear parabolic equations with Gaussian noise.
Their approach uses temporal weights to accommodate rough initial data.
For stochastic Navier--Stokes equations, Du and Zhang \cite{MR4002154} established local well-posedness in critical homogeneous Besov spaces and probabilistic global results for sufficiently small data.
Agresti and Veraar \cite{MR4703457} treated transport-type Wiener noise in critical Besov spaces and obtained stochastic Serrin criteria.
Accordingly, continuation through integrability conditions and regularization from rough initial states are established mechanisms in stochastic fluid theory.
Their implementation depends on the solution spaces and the structure of the noise.

The $L^p$ approach complements the Hilbert-space theories of stochastic Navier--Stokes equations; see \cite{MR348841,MR1922695,MR2459085,MR3049076} for foundational developments.
In the whole-space Sobolev setting, Kim \cite{MR2865433} obtained local strong solutions for subcritical initial data and, under the stated non-degeneracy and smallness assumptions, global solutions with large probability.
For initial velocities in Lebesgue spaces, the work of Kukavica and collaborators \cite{MR4385406,MR4552356} developed local strong theories on $\mathbb T^3$, reaching the range $p>3$.
Kukavica, Wang and Xu \cite{MR4908978} subsequently established a whole-space local theory in this range with multiplicative cylindrical Wiener noise, using smooth convolution Fourier multipliers.
At the critical exponent, Ayd{\i}n, Kukavica and Xu \cite{MR5107357} constructed local probabilistically strong solutions for general $L^3$ data on $\mathbb T^3$ with multiplicative Wiener forcing, with uniqueness in their specified solution class.
This critical result concerns a different domain and noise setting from the present mixed-noise theory on $\R^3$.

Several works address jump forcing directly in Lebesgue or related spaces.
Fernando, R\"udiger and Sritharan \cite{MR3411979} studied local mild solutions of stochastic Navier--Stokes equations driven by compensated Poisson noise.
Mohan and Sritharan \cite{MR3669657} treated additive Wiener and compensated Poisson forcing with initial data in $L^m(\R^m)$, while Zhu, Brze\'zniak and Liu \cite{MR3991628} developed an $L^p$ theory for the two-dimensional equations with jump noise and negative-order Sobolev initial data.
At the energy level, Motyl \cite{MR3034603} constructed martingale solutions on possibly unbounded domains with Wiener and Poisson forcing, and Chen and collaborators \cite{MR3912800} treated L\'evy-driven nonhomogeneous incompressible flow on bounded domains.
Related strong-solution constructions for discontinuous stochastic fluid equations include the Sobolev theory for Euler equations of Kukavica and Xu \cite{MR3950967} and the hereditary-viscosity model studied by Mohan and Sritharan \cite{MR3942498}.

The stochastic estimates used in these theories are supported by the Banach-space inequalities of Hausenblas \cite{MR2832576}, the $L^p$ valued Poisson integral estimates of Dirksen \cite{MR3265175}, and the maximal regularity and convolution estimates of Brze\'zniak and Hausenblas \cite{MR2529441} and Zhu, Brze\'zniak and Hausenblas \cite{MR3634281}.
The $L^p$ It\^o formulas with jumps of Gy\"ongy and Wu \cite{MR4165650} provide a corresponding basis for energy identities.
In particular, compensated Poisson estimates require both quadratic and $p$th-order integrability.
The treatment of terminal jumps must also respect the predictable pre-jump state.

The companion Paper I \cite{PodderKumar2026} constructs the maximal local $L^p$ solution of the present tamed system with simultaneous Wiener and compensated Poisson forcing for $p>3$.
It establishes pathwise uniqueness, localized dependence on the initial datum, restart at bounded stopping times, and stochastic pasting.
The present paper takes this maximal solution as its starting point.
Its purpose is to identify continuation conditions in the original $L^p$ class and to prove that additional finite-energy assumptions yield global continuation and positive-time Sobolev regularity.

The preceding literature establishes global solvability in Hilbert energy spaces and local solvability in $L^p$.
To our knowledge, the connection between these theories has not been established in the following form: intrinsic continuation of the maximal local $L^p$ solution of the tamed equations on $\R^3$ with simultaneous multiplicative Wiener and compensated Poisson noise, together with global persistence in the same uniqueness class for $L^p\cap L^2$ initial data throughout the range $p>3$.
The present paper establishes this connection under the stated coefficient and coercivity assumptions.
We prove continuation criteria independent of the auxiliary cutoffs and combine them with finite-energy estimates and positive-time Sobolev regularization to continue the canonical solution globally, without an initial $H^1$ assumption or a smallness condition.
The continuation argument includes the terminal Poisson jump in the stopped equation, thereby reconciling the estimates before the endpoint with the admissible state required for restart.
\subsection{Main results and methods}

Throughout the paper we retain the local coefficient assumptions of \cite{PodderKumar2026}.
The continuation criteria use only this local framework; the additional coercivity and gradient-noise hypotheses enter the finite-energy global theory.
All statements concern the maximal solution of \eqref{eq:intro-main-system} on the prescribed stochastic basis.

Our first result characterizes finite-time breakdown through the intrinsic quantity
\begin{equation*}
\mathcal Q_u(t) :=\sup_{0\le s\le t}\|u(s)\|_p^p +\int_0^t\|u(s)\|_{3p}^p\,\dd s, \qquad t<\tau_{\max}.
\end{equation*}
Theorem \ref{thm:3.6-intrinsic-blowup} establishes
\begin{equation*}
\P\left( \tau_{\max}<\infty,\quad \lim_{t\uparrow\tau_{\max}}\mathcal Q_u(t)<\infty \right)=0.
\end{equation*}
Theorem \ref{thm:native-serrin-continuation} strengthens this alternative to a Serrin criterion involving the original spatial norm alone:
\begin{equation*}
\P\left( \tau_{\max}<\infty,\quad \int_0^{\tau_{\max}}\|u(s)\|_p^{r_p}\,\dd s<\infty \right)=0, \qquad r_p:=\frac{2p}{p-3},\qquad \frac{2}{r_p}+\frac{3}{p}=1.
\end{equation*}
In particular,
\begin{equation*}
\limsup_{t\uparrow\tau_{\max}}\|u(t)\|_p=\infty \quad\text{almost surely on }\{\tau_{\max}<\infty\}.
\end{equation*}
Neither criterion requires an additional $L^2$ or $H^1$ assumption on the initial datum.
Both are formulated directly in terms of the maximal solution and are independent of its auxiliary localizing sequence and approximation parameters.

The intrinsic criterion rests on recovery of the local-energy dissipation from a bound on $\mathcal Q_u$, followed by control of the deterministic and stochastic tails at a localized finite endpoint.
This yields a terminal state and a closed stopped formulation in the original uniqueness class.
The endpoint Poisson contribution is retained, with its coefficient evaluated at the predictable left limit.
Restart begins from the terminal value and uses the subsequent stochastic interval, so a jump at the switching time is counted exactly once.
The restart and pasting results of \cite{PodderKumar2026} then give an admissible extension.
The Serrin refinement follows from interpolation between $L^p$ and $L^{3p}$, combined with estimates for the full projected drift and both noise terms.

Our second main result supplies a global continuation mechanism for finite-energy initial data.
In addition to the local assumptions, we impose the taming conditions
\begin{equation}
\label{eq:intro-taming-conditions}
\begin{aligned}
0\le g_N(r)&\le C_\nu r,\qquad g_N'(r)\ge0,\qquad r\ge0, \\
c_{N,\nu}&:=\sup_{r\ge0} \left(\frac r\nu-2g_N(r)\right)<\infty,
\end{aligned}
\end{equation}
and the $L^2$ coefficient and gradient-noise bounds \eqref{eq:2-raw-L2-growth} and \eqref{eq:2-H1}.
The coercivity in \eqref{eq:intro-taming-conditions} controls the convective contribution to the $H^1$ energy balance through the taming dissipation and a lower-order energy term.
Under these hypotheses, Theorem \ref{thm:6.2-finite-energy-global} proves that
\begin{equation*}
u_0\in L^p\bigl(\Omega,\mathcal F_0;L^p_\sigma(\R^3)\bigr) \cap L^2\bigl(\Omega,\mathcal F_0;L^2_\sigma(\R^3)\bigr), \qquad p>3,
\end{equation*}
implies $\P(\tau_{\max}=\infty)=1$.
No smallness condition on $u_0$ is required.
The resulting global strong solution is pathwise unique in the local-energy class of Paper I.
For every finite $T>0$, almost surely,
\begin{equation*}
u\in \mathbb D\bigl([0,T];L^p_\sigma\cap L^2_\sigma\bigr) \cap L^2(0,T;H^1_\sigma), \qquad \mathcal Q_u(T)<\infty.
\end{equation*}
Here and below, the spatial domain is $\R^3$, and intersections are equipped with their sum norms.

The transition from finite energy to the Sobolev regularity needed for global continuation is quantified by a time-weighted estimate.
Proposition \ref{prop:4.4-time-weighted-H1} and Theorem \ref{thm:6.2-finite-energy-global} give
\begin{equation*}
\E\left[ \sup_{0<t\le T}t\|u(t)\|_{H^1}^2 +\nu\int_0^T t\|u(t)\|_{H^2}^2\,\dd t \right] \le C_T\left(1+\E\|u_0\|_2^2\right).
\end{equation*}
Thus, for every $0<\delta<T$, almost surely,
\begin{equation*}
u\in \mathbb D\bigl([\delta,T];L^p_\sigma\cap H^1_\sigma\bigr) \cap L^2(\delta,T;H^2_\sigma).
\end{equation*}
The time weight accommodates the possible singularity of the gradient at the initial time.
It also supplies admissible $H^1$ restart data at positive stopping times, after localization of the corresponding moments.
If $u_0$ already belongs to $L^2(\Omega;H^1_\sigma)$, the $H^1$ persistence and global energy estimate hold from time zero; see Theorems \ref{thm:4.3-H1-persistence} and \ref{thm:6.1-global-continuation}.

Transferring these Hilbert-space bounds to the $L^p$ continuation theory requires estimates for the projected nonlinearities.
Although convection cancels in the $L^2$ balance, the vector $J_p(u)=|u|^{p-2}u$ is generally not solenoidal.
The $L^p$ energy argument therefore retains both contributions to the drift pressure:
\begin{equation*}
-\Delta\pi =\sum_{i,j=1}^3\partial_i\partial_j(u_i u_j) +\nabla\cdot\bigl(g_N(|u|^2)u\bigr).
\end{equation*}
Lemma \ref{lem:5.1-pressure} and Proposition \ref{prop:5.3-stopped-Lp} control these terms and yield a native $L^p$ estimate weighted by an integrating factor associated with
\begin{equation*}
\Gamma(t) :=\int_0^t \left(1+\|u(s)\|_{H^1}^{2p}+\|u(s)\|_\infty\right)\,\dd s
\end{equation*}
in the $H^1$ regime.
Localizing this integral gives estimates independent of the intrinsic exit level.
The Serrin criterion also exhibits the connection directly: Sobolev embedding controls its spatial norm for $3<p\le6$, and interpolation between $H^1$ and $H^2$ supplies the required time integrability for $p>6$.

The energy estimates are obtained through symmetric spatial regularization, which preserves the $L^2$ cancellation and the coercive $H^1$ balance.
Convergence in the native path norm and weak lower semicontinuity transfer these bounds to the maximal solution.
The jump estimates use the quadratic and $p$th moments stipulated in the local theory.
The Hilbert-space estimates require only the corresponding second moments.
Controlling the nonnegative jump remainders through the Poisson integral and its compensator avoids an additional fourth-moment assumption.

Global pathwise regularity is distinguished from unweighted moment estimates.
Under the basic global hypotheses, $\mathcal Q_u(T)<\infty$ almost surely for every finite $T$.
For $3<p\le6$, if $u_0\in L^p(\Omega;H^1_\sigma)$ and the projected jump coefficient satisfies the additional $H^1$ $p$th-moment growth assumption, Proposition \ref{prop:6.3-unweighted-intrinsic-moments} gives
\begin{equation*}
\E\left[ \mathcal Q_u(T) +\int_0^T\int_{\R^3} |u(s,x)|^{p-2}|\nabla u(s,x)|^2\,\dd x\,\dd s \right] \le C_T\left(1+\E\|u_0\|_{H^1}^p\right).
\end{equation*}
No exponential moment of $\Gamma$ is required for this conclusion.

Finally, Theorem~\ref{thm:6.4-global-continuous-dependence} establishes continuous dependence of the global solution on the initial datum in the native $L^p$ topology.
More precisely, for initial data in the global existence class $L^p(\Omega;L^p_\sigma)\cap L^2(\Omega;L^2_\sigma)$, convergence of the initial data in probability in $L^p_\sigma$ implies convergence of the corresponding solutions in probability, uniformly in $L^p_\sigma$ on every finite time interval and in $L^p(0,T;L^{3p}_\sigma)$.
No convergence of the initial data in $L^2_\sigma$ is required.
\subsection{Organization of the paper}
The paper is organized as follows.
\cref{Section 2} specifies the stochastic framework and coefficient assumptions and records the local results used from Paper I \cite{PodderKumar2026}.
\cref{Sec:3} develops endpoint completion and proves the intrinsic and Serrin continuation criteria.
\cref{Section 4} establishes $L^2$ and $H^1$ persistence and time-weighted regularization.
\cref{Section 5} treats the pressure and noise contributions to the stopped $L^p$ estimate.
\cref{Section 6} proves global continuation for $H^1$ and then $L^p\cap L^2$ initial data, derives the additional unweighted intrinsic moment bounds under stronger moment assumptions, and establishes global continuous dependence on the initial datum in the native $L^p$ topology.

\section{Preliminaries, assumptions, and the local theory}
\label{Section 2}

Throughout the paper, $p\in(3,\infty)$, $\nu>0$, and $N>0$ are fixed.
All spatial function spaces are defined on $\R^3$.
We first specify the analytic and stochastic framework, then record the local theory of Paper~I \cite{PodderKumar2026} in the form required for continuation.
The assumptions used for the Hilbert-space estimates are distinguished from those inherited from the local construction.

\subsection{Spatial notation and regularization}

We write $\|\cdot\|_q$ for the norm of $L^q(\R^3;\R^3)$, with the same notation for scalar and tensor fields when the target space is clear.
Let $C^\infty_{c,\sigma}$ denote the smooth, compactly supported, divergence-free vector fields, and set
\begin{equation}
\label{eq:2-solenoidal-spaces}
\begin{aligned}
L^q_\sigma &:=\overline{C^\infty_{c,\sigma}}^{L^q} =\{v\in L^q:\nabla\cdot v=0\text{ in }\mathcal D'\}, \qquad 1<q<\infty, \\
E&:=L^p_\sigma,\qquad \mathcal H:=L^2_\sigma.
\end{aligned}
\end{equation}
We use $W^{k,q}$ for the usual Sobolev spaces and $H^s=W^{s,2}$ for the Bessel potential spaces, including the negative orders used in distributional formulations.
With the Fourier convention $$ \widehat f(\xi)=\int_{\R^3}e^{-2\pi i x\cdot\xi}f(x)\,\dd x, $$ the $H^s$ norm is $\|(1+4\pi^2|\xi|^2)^{s/2}\widehat f\|_2$.
The corresponding solenoidal subspace is denoted by $H^s_\sigma$.
In particular, $\|v\|_{H^1}^2=\|v\|_2^2+\|\nabla v\|_2^2$.
Intersections carry their sum norms; $\langle\cdot,\cdot\rangle$ denotes a distributional or dual pairing, and $(\cdot,\cdot)_2$ the $L^2$ inner product.
Spatial derivatives are understood weakly.

The Helmholtz--Leray projection is $$ \mathcal P=I+\nabla(-\Delta)^{-1}\nabla\cdot, \qquad (\mathcal Pv)_i=v_i+\sum_{j=1}^3R_iR_jv_j, \qquad R_j=\partial_j(-\Delta)^{-1/2}.
$$ It is bounded on $W^{k,q}$, $1<q<\infty$, and on the Hilbert Sobolev scale, commutes with spatial derivatives, and is the orthogonal projection onto $H$ in $L^2$.
Our tensor convention is $$ (a\otimes b)_{ij}:=b_i a_j, \qquad (\nabla\cdot F)_i:=\sum_{j=1}^3\partial_jF_{ij}, \qquad (\nabla\psi)_{ij}:=\partial_j\psi_i.
$$ Thus $\nabla\cdot(a\otimes b)=(a\cdot\nabla)b$ when $\nabla\cdot a=0$.
Writing $$ (\mathcal P^{(1)}F)_{ij} :=F_{ij}+\sum_{\ell=1}^3R_iR_\ell F_{\ell j}, $$ we have $\nabla\cdot\mathcal P^{(1)}F=\mathcal P\nabla\cdot F$ in distributions.

Let $S_\nu(t)=e^{\nu t\Delta}$.
We retain the Gaussian regularization of \cite{MR4908978} :
\begin{equation}
\label{eq:2-smoothing}
\begin{aligned}
S_m f:={\mathrm P}_{\le m}f &:=\mathcal F^{-1}\bigl(e^{-|\xi|^2/m^2}\widehat f(\xi)\bigr) =K_m*f, \\
K_m(x)&:=m^3K(mx),\qquad K(x):=\pi^{3/2}e^{-\pi^2|x|^2},\qquad m\in\mathbb N.
\end{aligned}
\end{equation}
The operators $S_m$ are self-adjoint on $L^2$ and commute with $\mathcal P$, spatial derivatives, and $S_\nu(t)$.
They are smoothing operators, not projections.
For every Banach space $X$,
\begin{equation}
\label{eq:2-smoothing-properties}
\begin{aligned}
\|S_m f\|_{L^q(X)}&\le\|f\|_{L^q(X)}, &&1\le q\le\infty, \\
S_m f&\longrightarrow f\quad\text{in }L^q(X), &&1\le q<\infty, \\
\|S_m f-S_n f\|_{L^q(X)} &\le C_K|m^{-1}-n^{-1}|\,\|\nabla f\|_{L^q(X^3)}, &&1\le q<\infty,
\end{aligned}
\end{equation}
where the last assertion holds for $f\in W^{1,q}(\R^3;X)$ and $C_K=\int_{\R^3}|x|K(x)\,\dd x$.
We also use $$ \|\nabla^j S_m f\|_{L^q(X)} \le C_{j,r,q}m^{j+3(1/r-1/q)}\|f\|_{L^r(X)}, \qquad 1\le r\le q\le\infty,\quad j\in\mathbb N_0.
$$ These bounds follow from the scalar convolution kernel and apply equally to vector fields, tensors, and Wiener coefficients; see \cite{PodderKumar2026}, Section~2.

We shall repeatedly use interpolation and the Sobolev inequalities $$ \|v\|_6\le C\|\nabla v\|_2, \qquad \|v\|_\infty\le C\|v\|_{H^1}^{1/2}\|v\|_{H^2}^{1/2} \le C\|v\|_{H^2}, $$ on their respective Sobolev domains.
In particular,
\begin{equation}
\label{eq:2-native-Sobolev}
\|v\|_{3p}^p =\bigl\||v|^{p/2}\bigr\|_6^2 \le C\left\|\nabla\bigl(|v|^{p/2}\bigr)\right\|_2^2, \qquad |v|^{p/2}\in H^1.
\end{equation}
Since the spatial domain is unbounded, an $L^p$ assumption does not imply finite $L^2$ energy.
Every additional condition on the initial datum is therefore stated explicitly as an intersection of spaces.

\subsection{Stochastic basis, integrals, and path spaces}
\label{subsec:2-stochastic}

Let $(\Omega,\mathcal F,(\mathcal F_t)_{t\ge0},\P)$ be a complete filtered probability space satisfying the usual conditions, and let $\mathscr P$ denote its predictable $\sigma$-algebra.
Let $\mathcal U$ be a separable real Hilbert space, $W$ a cylindrical $\mathcal U$-Wiener process, and $N$ a Poisson random measure on $(0,\infty)\times Z$, where $(Z,\mathcal Z,\mu)$ is $\sigma$-finite.
The compensator of $N$ relative to $(\mathcal F_t)$ is $\dd t\,\mu(\dd z)$, and $$ \widetilde N(\dd t,\dd z):=N(\dd t,\dd z)-\dd t\,\mu(\dd z).
$$ The noises are independent and jointly compatible with the filtration: for every deterministic $t\ge0$, their joint increments after $t$ are independent of $\mathcal F_t$.
This is the hypothesis used for restart in \cref{thm:2-Paper-I}(iii).
No additional Euclidean L\'evy-measure condition is imposed on $\mu$; the required integrability is specified through the jump coefficient.

For $1<q<\infty$, set
\begin{equation}
\label{eq:2-gamma-notation}
\begin{aligned}
\gamma_q&:=\gamma\bigl(\mathcal U;L^q(\R^3;\R^3)\bigr), \\
\mathbb L^q&:=L^q\bigl(\R^3;\gamma(\mathcal U;\R^3)\bigr) \simeq\gamma_q.
\end{aligned}
\end{equation}
The identification is the $\gamma$-Fubini isomorphism, with constants depending only on $q$.
For spatial derivatives, the finite-dimensional target is replaced by the appropriate tensor space.
The ideal property gives the bounded action of $\mathcal P$ and $S_m$ on these coefficient spaces.
The spaces $L^p$ and $E$ are UMD and have martingale type~$2$.
Stochastic integrals are defined by completion of elementary predictable integrands and, when necessary, localization.
For predictable $B$ and $r\ge2$,
\begin{equation}
\label{eq:2-Wiener-maximal}
\E\sup_{t\le T} \left\|\int_0^t B(s)\,\dd W_s\right\|_p^r \le C_{p,r}\E \left(\int_0^T\|B(s)\|_{\gamma_p}^2\,\dd s\right)^{r/2}.
\end{equation}
For a $\mathscr P\otimes\mathcal Z$-measurable $L^p$-valued integrand $J$,
\begin{equation}
\label{eq:2-Poisson-maximal}
\begin{aligned}
\E\sup_{t\le T} \left\|\int_{(0,t]\times Z}J(s,z)\, \widetilde N(\dd s,\dd z)\right\|_p^p \le&C_p\E \left(\int_0^T\int_Z\|J(s,z)\|_p^2 \,\mu(\dd z)\,\dd s\right)^{p/2} \\
&+C_p\E\int_0^T\int_Z\|J(s,z)\|_p^p \,\mu(\dd z)\,\dd s.
\end{aligned}
\end{equation}
We refer to \cite{MR2330977,MR3617205} for Wiener integration and to \cite{MR3265175} for the Poisson estimate.
The latter is used as an upper bound in the two displayed integrand norms.
Both inequalities apply to stopped integrands.

For a real local martingale $M$ starting at zero, the Davis--BDG inequality gives $$ \E\sup_{t\le T}|M(t)|\le C\E[M]_T^{1/2}, $$ with preliminary localization when required.
We also use the following consequence of the compensator identity \cite{MR1464694}: if $K\ge0$ is predictable and $\E\int_0^T\int_ZK\,\mu(\dd z)\,\dd s<\infty$, then $$ \E\sup_{t\le T} \left|\int_{(0,t]\times Z}K(s,z)\, \widetilde N(\dd s,\dd z)\right| \le2\E\int_0^T\int_ZK(s,z)\,\mu(\dd z)\,\dd s.
$$ Here the compensated integral is the difference between the integrable Poisson integral and its compensator.
For $K$ equal to a squared jump norm, this argument requires only the corresponding second moment of the jump coefficient.

For a separable Banach space $X$, $\mathbb D([0,T];X)$ denotes the space of c\`adl\`ag paths, equipped with the Skorokhod $J_1$ topology whenever a path-space topology is invoked.
Following Paper~I, $L^r(\Omega;\mathbb D([0,T];X))$ denotes adapted c\`adl\`ag processes, identified up to indistinguishability, for which $$ \E\sup_{t\le T}\|v(t)\|_X^r<\infty.
$$ Convergence in this notation means convergence in the expected supremum norm.
This is a convention for a space of processes, rather than a Bochner-space identification with the $J_1$ path space.
It implies uniform convergence in probability and hence convergence in probability in the $J_1$ topology.
The notation on a random closed interval has the same meaning, with the supremum taken through its terminal time.

For an adapted c\`adl\`ag process, $v_-$ is predictable; we set $v(0-)=v(0)$.
Moreover, $v=v_-$ for $\dd\P\otimes\dd t$-almost every point.
For a finite stopping time $\rho$, $v(\rho)$ is $\mathcal F_\rho$-measurable and $v(\rho-)$ is $\mathcal F_{\rho-}$-measurable.
If $\rho\le\theta$ are stopping times, the processes
\begin{equation}
\label{eq:2-predictable-intervals}
\mathbf1_{(0,\theta]},\qquad \mathbf1_{(\rho,\theta]},\qquad \mathbf1_A\mathbf1_{(\rho,\theta]},\quad A\in\mathcal F_\rho,
\end{equation}
are predictable.
Integrands defined on such intervals are extended by zero elsewhere.
Accordingly, stopping retains a possible Poisson jump at the terminal time, whereas a restart at $\rho$ uses $(\rho,\theta]$ and excludes the jump already contained in the restart datum.

The Poisson random measure has no mark at a positive predictable stopping time.
Indeed, for bounded predictable $\theta>0$ and $B\in\mathcal Z$ with $\mu(B)<\infty$, $$ \E N(\{\theta\}\times B) =\mu(B)\E\int_0^\infty\mathbf1_{\{s=\theta\}}\,\dd s=0.
$$ A countable exhaustion of $Z$ and localization in time give the general assertion.
This fact will be applied only after predictability of the relevant lifetime has been established.

\subsection{The equation and coefficient hypotheses}
\label{subsec:2-hypotheses}

The coefficients are defined on the ambient velocity space $L^p$.
We assume that
\begin{equation}
\label{eq:2-coefficient-measurability}
\begin{aligned}
\sigma&:\Omega\times\R_+\times L^p\longrightarrow\gamma_p, \\
G&:\Omega\times\R_+\times L^p\times Z\longrightarrow L^p
\end{aligned}
\end{equation}
are respectively $\mathscr P\otimes\mathcal B(L^p)$- and $\mathscr P\otimes\mathcal B(L^p)\otimes\mathcal Z$-measurable.
Suppressing the dependence on $\omega$, write $$ \Sigma(t,v):=\mathcal P\sigma(t,v), \qquad \mathcal G(t,v,z):=\mathcal PG(t,v,z).
$$ The equation is
\begin{equation}
\label{eq:1_Main}
\begin{cases}
\begin{aligned}
\dd u=&\bigl[\nu\Delta u-\mathcal P((u\cdot\nabla)u) -\mathcal P(g_N(|u|^2)u)\bigr]\,\dd t \\
&+\Sigma(t,u(t-))\,\dd W_t +\displaystyle\int_Z\mathcal G(t,u(t-),z)\, \widetilde N(\dd t,\dd z),
\end{aligned}
\\
\nabla\cdot u=0,\qquad u(0)=u_0.
\end{cases}
\end{equation}
Using $u(t-)$ in the Wiener coefficient merely selects a predictable representative.
In the Poisson coefficient, the left limit is part of the equation.

All coefficient estimates below hold outside a common $\dd\P\otimes\dd t$-null set, independent of the state variables, with constants uniform in $(\omega,t)$.
An estimate involving an additional spatial norm is required on the intersection where its right-hand side is finite.
The coefficient realizations in different spatial spaces agree as distributions on their common domains.

\paragraph{Local taming assumptions.}
We assume $g_N\in C^1([0,\infty))$ and retain the bounds used in Paper~I:
\begin{equation*}
\label{eq:2-taming-local}
\begin{aligned}
0\le g_N(r)&\le C_\nu r,\qquad r\ge0, \\
|g_N(|a|^2)a-g_N(|b|^2)b| &\le C_{N,\nu}(|a|^2+|b|^2)|a-b|, \qquad a,b\in\R^3.
\end{aligned}
\tag{T0}
\end{equation*}
In particular, $|g_N(|a|^2)a|\le C_\nu|a|^3$.
The difference bound holds for the standard taming functions with bounded derivative.

\paragraph{Local Wiener assumptions.}
Fix $q_\sigma\in(p,3p/2)$ , with $q_\sigma$ sufficiently close to $3p/2$ for the interpolation arguments used there.
We assume
\begin{equation*}
\label{eq:2-W-local}
\begin{aligned}
\|\sigma(t,v)\|_{\mathbb L^p} &\le C\bigl(1+\|v\|_{q_\sigma}^2\bigr), \\
\|\sigma(t,v)-\sigma(t,w)\|_{\mathbb L^p} &\le C\left\|(|v|+|w|)^{1/2}|v-w|\right\|_p, \\
\|\nabla\sigma(t,v)\|_{\mathbb L^p} &\le C\bigl(1+\|v\|_{3p/2}^2\bigr).
\end{aligned}
\tag{W}
\end{equation*}
Thus $(3p/2)^-$ in \cite{PodderKumar2026} denotes the fixed exponent $q_\sigma$ used here.

\paragraph{Local jump assumptions.}
For a fixed $\alpha\in[0,2/3)$ and each $r\in\{2,p\}$, we assume
\begin{equation*}
\label{eq:2-J-local}
\begin{aligned}
\int_Z\|G(t,v,z)\|_p^r\,\mu(\dd z) &\le C_r(1+\|v\|_p^r), \\
\int_Z\|G(t,v,z)-G(t,w,z)\|_p^r\,\mu(\dd z) &\le C_r\left\|(|v|+|w|)^\alpha|v-w|\right\|_p^r, \\
\int_Z\|\nabla G(t,v,z)\|_p^r\,\mu(\dd z) &\le C_r(1+\|v\|_{3p/2}^r).
\end{aligned}
\tag{J}
\end{equation*}
When $\alpha=0$, the weight in the difference estimate is understood as~$1$.

\paragraph{$L^2$ coefficient bounds.}
The final local theorem of Paper~I also uses the unprojected bound
\begin{equation}
\label{eq:2-raw-L2-growth}
\|\sigma(t,v)\|_{\gamma_2}^2 +\int_Z\|G(t,v,z)\|_2^2\,\mu(\dd z) \le C_0(1+\|v\|_2^2),\qquad v\in L^p\cap L^2.
\end{equation}
We retain this coefficient hypothesis throughout.
It does not require the initial datum to belong to $L^2$ in the intrinsic continuation theory.
Orthogonality of $\mathcal P$ yields
\begin{equation*}
\label{eq:2-H0}
\|\Sigma(t,v)\|_{\gamma_2}^2 +\int_Z\|\mathcal G(t,v,z)\|_2^2\,\mu(\dd z) \le C_0(1+\|v\|_2^2),\qquad v\in E\cap H.
\tag{H0}
\end{equation*}

\paragraph{Additional assumptions for the Hilbert-space estimates.}
For the energy and global continuation arguments, we supplement \eqref{eq:2-taming-local} with
\begin{equation*}
\label{eq:2-taming-coercive}
g_N'(r)\ge0,\qquad c_{N,\nu}:=\sup_{r\ge0}\left(\frac r\nu-2g_N(r)\right)<\infty.
\tag{T}
\end{equation*}
The $H^1$ persistence and positive-time regularization arguments also use
\begin{equation*}
\label{eq:2-H1}
\begin{aligned}
&\|\nabla\Sigma(t,v)\|_{\gamma(\mathcal U;L^2)}^2 +\int_Z\|\nabla\mathcal G(t,v,z)\|_2^2\,\mu(\dd z) \\
&\hspace{3em}\le C_1(1+\|v\|_{H^1}^2), \qquad v\in E\cap H^1_\sigma.
\end{aligned}
\tag{H1}
\end{equation*}
The tensor target in the first norm is understood as above.
Together with \eqref{eq:2-H0}, this gives the corresponding $H^1$ realizations of the noise at an $H^1$ state.
Only a second jump moment in $H^1$ is included in these standing assumptions.

\paragraph{Consequences in $L^p$.}
The value of $\sigma(t,0)$ is uniformly bounded by \eqref{eq:2-W-local}.
Applying its difference estimate with $w=0$ and then projecting gives
\begin{equation*}
\label{eq:2-P1}
\|\Sigma(t,v)\|_{\gamma_p} \le C_p\bigl(1+\||v|^{3/2}\|_p\bigr) =C_p\bigl(1+\|v\|_{3p/2}^{3/2}\bigr), \quad v\in E\cap L^{3p/2}.
\tag{P1}
\end{equation*}
Likewise, \eqref{eq:2-J-local} implies
\begin{equation*}
\label{eq:2-P2}
\int_Z\|\mathcal G(t,v,z)\|_p^r\,\mu(\dd z) \le C_r(1+\|v\|_p^r), \qquad v\in E,\quad r\in\{2,p\}.
\tag{P2}
\end{equation*}
Thus \eqref{eq:2-P1}--\eqref{eq:2-P2}, used in \cref{Sec:3}, \cref{Section 5}, impose no additional assumptions.
The difference and spatial derivative bounds remain part of the hypotheses inherited from Paper~I.

For the pressure estimates, set $J_p(v):=|v|^{p-2}v$.
This field need not be divergence-free, even when $v$ is, so the Leray projection must be retained in an $L^p$ energy pairing.
For $v\in C^\infty_{c,\sigma}$, define the deterministic pressure by the Newtonian potential of the right-hand side below:
\begin{equation}
\label{eq:2-pressure-convention}
\begin{aligned}
-\Delta\pi(v) &=\sum_{i,j=1}^3\partial_i\partial_j(v_i v_j) +\nabla\cdot\bigl(g_N(|v|^2)v\bigr), \\
\nu\Delta v-\mathcal P((v\cdot\nabla)v) -\mathcal P(g_N(|v|^2)v) &=\nu\Delta v-(v\cdot\nabla)v-g_N(|v|^2)v-\nabla\pi(v).
\end{aligned}
\end{equation}
Lemma~\ref{lem:5.1-pressure} establishes the pressure bounds and extends this convention to the required energy class.

\subsection{Closed local representatives and the input from Paper I}

Continuation requires a terminal state from which the equation can be restarted.
We therefore make explicit the closed-interval convention already present in the stopped formulation of Paper~I.
It concerns the terminal time of each local representative and does not prescribe a value at the maximal lifetime.

For finite stopping times $\alpha\le\beta$ and a process $v$ on $[\alpha,\beta]$, put
\begin{equation}
\label{eq:2-local-energy}
\mathcal E_p(v;\alpha,\beta) :=\E\left[ \sup_{\alpha\le s\le\beta}\|v(s)\|_p^p +\int_\alpha^\beta \left\|\nabla\bigl(|v(s)|^{p/2}\bigr)\right\|_2^2\,\dd s \right].
\end{equation}
For $\psi\in C_c^\infty(\R^3;\R^3)$ and $v\in E$, define $$ \mathcal B_\psi(v) :=\nu\langle v,\Delta\psi\rangle +\langle v\otimes v,\nabla\mathcal P\psi\rangle -\langle g_N(|v|^2)v,\mathcal P\psi\rangle.
$$ The pairings are well defined for $p>3$ by the Leray bounds and \eqref{eq:2-taming-local}; in particular, $$ |\mathcal B_\psi(v)| \le C_\psi\bigl(\|v\|_p+\|v\|_p^2+\|v\|_p^3\bigr).
$$

\begin{definition}
[Admissible closed local solution] \label{def:2-local-solution} Let $u_0$ be an $\mathcal F_0$-measurable $E$-valued random variable.
An admissible closed local solution of \eqref{eq:1_Main} is a pair $(v,\theta)$ such that:
\begin{enumerate}
\item[(i)] $\theta$ is an almost surely positive finite stopping time; \item[(ii)] $v$ is an adapted $E$-valued c\`adl\`ag process, stopped at $\theta$, with $v(0)=u_0$ and $\mathcal E_p(v;0,\theta)<\infty$; \item[(iii)] for every $\psi\in C_c^\infty(\R^3;\R^3)$, the identity
\begin{equation}
\label{eq:2-stopped-equation}
\begin{aligned}
\langle v(t),\psi\rangle =&\langle u_0,\psi\rangle +\int_0^{t\wedge\theta}\mathcal B_\psi(v(s))\,\dd s \\
&+\int_0^{t\wedge\theta} \left\langle\Sigma(s,v(s-))^*\psi,\dd W_s\right\rangle_{\mathcal U}+\int_{(0,t\wedge\theta]\times Z} \langle\mathcal G(s,v(s-),z),\psi\rangle\, \widetilde N(\dd s,\dd z)
\end{aligned}
\end{equation}
holds indistinguishably in $t\ge0$.
\end{enumerate}
The stochastic integrals are interpreted by localization in the integrand norms of \eqref{eq:2-Wiener-maximal}--\eqref{eq:2-Poisson-maximal}.
\end{definition}

The word \emph{strong} refers to the prescribed stochastic basis and noises; the equation is interpreted weakly in space.
The closed formulation \eqref{eq:2-stopped-equation} includes any jump at $\theta$ and gives an $\mathcal F_\theta$-measurable restart datum $v(\theta)$ with finite $p$th moment.
It does not require $\theta$ to be predictable or the post-jump value $v(\theta)$ to satisfy a deterministic exit-level bound.
Auxiliary estimates may also be stopped at times equal to zero, in which case the time integrals are empty.
Such zero lifetimes are not included in the family of admissible positive lifetimes.

By \eqref{eq:2-native-Sobolev}, an admissible solution also satisfies
\begin{equation}
\label{eq:2-local-intrinsic-energy}
\E\int_0^\theta\|v(s)\|_{3p}^p\,\dd s \le C\mathcal E_p(v;0,\theta).
\end{equation}
A solution on a stochastic interval $[0,\zeta)$ belongs to the local-energy uniqueness class if it admits increasing finite positive stopping times $\theta_n\uparrow\zeta$ and admissible closed representatives on $[0,\theta_n]$ that agree on every closed overlap and represent the solution before $\zeta$.
Strict inequality $\theta_n<\zeta$ is not part of this definition.
When $\zeta=\infty$ almost surely, this is the global local-energy class used below.

Localization on an event $A\in\mathcal F_0$ means that the weak formulation and the energy bounds are restricted to $A$.
The indicator $\mathbf1_A$ is predictable and may be brought inside the stochastic integrals.
For the extension to initial data without moment assumptions at the end of \cref{Section 6}, the uniqueness class is understood after localization on increasing $\mathcal F_0$-measurable events exhausting $\Omega$; the family of attainable lifetimes below continues to use Definition~\ref{def:2-local-solution} with its stated moment requirement.

We now collect the local existence, maximality, restart, and pasting results of \cite{PodderKumar2026}.
The maximality statement is Proposition~2.4 of that paper; its proof is given in Section~6.
The stronger taming and gradient-noise assumptions used for the Hilbert-space estimates are not required for this input.

\begin{theorem}
[Local existence, maximality, and restart from Paper I] \label{thm:2-Paper-I} Assume the stochastic hypotheses of Subsection~\ref{subsec:2-stochastic}, the measurability in \eqref{eq:2-coefficient-measurability}, and \eqref{eq:2-taming-local}, \eqref{eq:2-W-local}, \eqref{eq:2-J-local}, and \eqref{eq:2-raw-L2-growth}.
Let $u_0\in L^p(\Omega,\mathcal F_0;E)$.
\begin{enumerate}
\item[(i)] \emph{Local existence and uniqueness} (Paper~I, Theorem~2.3 \cite{PodderKumar2026}).
There exist an $(\mathcal F_t)$-stopping time $\tau$ with $\P(\tau>0)=1$ and an adapted $E$-valued c\`adl\`ag local strong solution $u$ on $[0,\tau]$ such that $$u\in L^p\bigl(\Omega;\mathbb D([0,\tau];E)\bigr) \cap L^p\bigl(\Omega;L^p(0,\tau;L^{3p})\bigr),$$ and $$\E\left[ \sup_{0\le s\le\tau}\|u(s)\|_p^p + \int_0^\tau \left\|\nabla\bigl(|u(s)|^{p/2}\bigr)\right\|_2^2\,\dd s \right] \le C\bigl(1+\E\|u_0\|_p^p\bigr).$$ The solution is pathwise unique among local strong solutions satisfying this local-energy regularity.

In particular, let $ \theta:=\tau\wedge1 $ and $v:=u(\cdot\wedge\theta)$.
$(v,\theta)$ is an admissible closed local solution in the sense of Definition~\ref{def:2-local-solution}, with $ 0<\theta\le1 \ \P\text{-a.s.} $ There exists an admissible solution $(v,\theta)$ with $0<\theta\le1$ almost surely and
\begin{equation}
\label{eq:2-imported-local-bound}
\mathcal E_p(v;0,\theta) \le C\bigl(1+\E\|u_0\|_p^p\bigr).
\end{equation}
Any two admissible solutions with the same initial datum and driving noises agree indistinguishably on their closed common lifetime.

\item[(ii)] \emph{Maximality} (Paper~I, Proposition~2.4 \cite{PodderKumar2026}).
Fix the initial datum, stochastic basis, coefficient maps, and noises, and let $\mathscr T$ be the family of finite lifetimes attainable in Definition~\ref{def:2-local-solution}.
There is a maximal construction $(u,(\tau_n),\tau_{\max})$ with $$ \tau_{\max}=\operatorname*{ess\,sup}_{\theta\in\mathscr T}\theta, \qquad 0<\tau_n\le n, \qquad \tau_n\uparrow\tau_{\max}\quad\P\text{-a.s.} $$ Each $\tau_n\in\mathscr T$ has an admissible closed stopped representative $U_n$.
After a common null modification, these representatives agree on all closed overlaps and determine an adapted, locally $E$-valued c\`adl\`ag process $u$ on $[0,\tau_{\max})$.
Every other solution in the local-energy uniqueness class has lifetime at most $\tau_{\max}$ and agrees with $u$ before that lifetime.
For a closed local solution, agreement includes its endpoint whenever that endpoint lies strictly below $\tau_{\max}$.
The maximal lifetime is unique up to almost sure equality, and the maximal process is unique up to indistinguishability on $[0,\tau_{\max})$.

\item[(iii)] \emph{Restart} (Paper~I, Lemma~6.2 \cite{PodderKumar2026}).
Let $\rho$ be a bounded stopping time, and choose an $\mathcal F_\rho$-measurable $\xi\in L^p(\Omega;E)$.
Define $$
\begin{aligned}
\mathcal F_s^\rho&:=\mathcal F_{\rho+s}, \\
W_s^\rho(h)&:=W_{\rho+s}(h)-W_\rho(h), \qquad h\in\mathcal U, \\
N^\rho((0,s]\times B)&:=N((\rho,\rho+s]\times B), \qquad B\in\mathcal Z,\quad\mu(B)<\infty.
\end{aligned}
$$ The shifted filtration and noises satisfy the same stochastic hypotheses.
The translated coefficients $\sigma^\rho(\omega,s,a)=\sigma(\omega,\rho(\omega)+s,a)$ and $G^\rho(\omega,s,a,z)=G(\omega,\rho(\omega)+s,a,z)$ retain the required measurability and structural constants.
There is a pathwise unique restarted branch $\widehat v$ on $[\rho,\widehat\theta]$ satisfying
\begin{equation}
\label{eq:2-imported-restart}
\begin{aligned}
\rho<\widehat\theta&\le\rho+1\quad\P\text{-a.s.}, \qquad \widehat v(\rho)=\xi, \\
\mathcal E_p(\widehat v;\rho,\widehat\theta) &\le C\bigl(1+\E\|\xi\|_p^p\bigr).
\end{aligned}
\end{equation}
Its extension by zero before $\rho$ and by its terminal value after $\widehat\theta$ is adapted and c\`adl\`ag.
For every shifted stopping time $S$, $\rho+S$ is an original stopping time and $\mathcal F_S^\rho=\mathcal F_{\rho+S}$.

\item[(iv)] \emph{Pasting} (Paper~I, Lemma~6.3 \cite{PodderKumar2026}).
The family $\mathscr T$ is upward directed.
If $(v,\rho)$ is admissible, $\rho$ is bounded, and the branch in \textup{(iii)} starts from $\xi=v(\rho)$, then $$ w(t):=
\begin{cases}
v(t),&0\le t\le\rho, \\
\widehat v(t\wedge\widehat\theta),&t>\rho
\end{cases}
$$ is its unique admissible extension through $\widehat\theta$, and $$ \mathcal E_p(w;0,\widehat\theta) \le\mathcal E_p(v;0,\rho) +\mathcal E_p(\widehat v;\rho,\widehat\theta)<\infty.
$$ The original and restarted Poisson integrals are supported on $(0,\rho]$ and $(\rho,\widehat\theta]$, respectively.
\end{enumerate}
The constants in \eqref{eq:2-imported-local-bound} and \eqref{eq:2-imported-restart} depend only on the structural parameters and are independent of auxiliary cutoffs and the bounded restart time.
\end{theorem}

Restricting the attainable family in Theorem~\ref{thm:2-Paper-I} to finite, or to bounded, lifetimes leaves its essential supremum unchanged: an admissible solution can be stopped at each positive integer.
The theorem supplies closed local representatives without imposing a terminal value at $\tau_{\max}$.
Lemma~\ref{lem:3.1-strict-localization} will prove that their localizing times satisfy $\tau_n<\tau_{\max}$; endpoint completion and the continuation criteria are established subsequently in \cref{Sec:3}.

\subsection{The linear estimate and lower semicontinuity}

The continuation argument uses a linear estimate to recover the nonlinear dissipation and to complete finite endpoints.
We state the three-dimensional, vector-valued specialization of Paper~I, Theorem~3.1 \cite{PodderKumar2026}, with viscosity $\nu$.
This specialization leaves $E=L^p_\sigma$ unchanged; the linear equation itself is posed in the ambient space $L^p(\R^3;\R^3)$.

\begin{theorem}
[Stochastic heat equation with Wiener and jump forcing] \label{Thm:heat-exis} Let $T\in(0,\infty)$ and $$ \frac{3p}{p+1}\le q_f\le p, \qquad \frac{3p}{2p+1}\le q_h\le p.
$$ Suppose that $v_0\in L^p(\Omega,\mathcal F_0;L^p)$, that $f$ and $h$ are progressively measurable with $$
\begin{aligned}
f&\in L^p\bigl(\Omega\times(0,T); L^{q_f}(\R^3;\R^{3\times3})\bigr), \\
h&\in L^p\bigl(\Omega\times(0,T); L^{q_h}(\R^3;\R^3)\bigr),
\end{aligned}
$$ and that $g$ is predictable with $g\in L^p(\Omega\times(0,T);\mathbb L^p)$.
Let $G_0$ be $\mathscr P\otimes\mathcal Z$-measurable, with values in $L^p$, and assume $$ \E\int_0^T\int_Z\|G_0(s,z)\|_p^p\,\mu(\dd z)\,\dd s +\E\left(\int_0^T\int_Z\|G_0(s,z)\|_p^2 \,\mu(\dd z)\,\dd s\right)^{p/2}<\infty.
$$ Then $$
\begin{aligned}
\dd v(t) &=\bigl[\nu\Delta v(t)+\nabla\cdot f(t)+h(t)\bigr]\,\dd t +g(t)\,\dd W_t +\int_ZG_0(t,z)\,\widetilde N(\dd t,\dd z), \\
v(0)&=v_0,
\end{aligned}
$$ has a unique adapted solution in the distributional formulation such that $$ v\in L^p\bigl(\Omega;\mathbb D([0,T];L^p)\bigr), \qquad |v|^{p/2}\in L^2\bigl(\Omega\times(0,T);H^1\bigr).
$$ Moreover,
\begin{equation}
\label{ineq:3_eng-est}
\begin{aligned}
\E\bigg[ &\sup_{0\le t\le T}\|v(t)\|_p^p +\int_0^T\left\|\nabla\bigl(|v(s)|^{p/2}\bigr)\right\|_2^2\,\dd s \bigg] \\
\le&C\E\bigg[ \|v_0\|_p^p +\int_0^T\bigl(\|f(s)\|_{q_f}^p +\|h(s)\|_{q_h}^p +\|g(s)\|_{\mathbb L^p}^p\bigr)\,\dd s \\
&\hspace{4em} +\int_0^T\int_Z\|G_0(s,z)\|_p^p\,\mu(\dd z)\,\dd s +\left(\int_0^T\int_Z\|G_0(s,z)\|_p^2 \,\mu(\dd z)\,\dd s\right)^{p/2} \bigg],
\end{aligned}
\end{equation}
where $C=C(p,q_f,q_h,\nu,T)$ is independent of the initial datum and the forcing terms.
\end{theorem}

The viscosity-dependent form follows from the spatial change of variables $x=\sqrt\nu\,y$ in the result of Paper~I.
Estimate~\eqref{ineq:3_eng-est} also applies to differences, with the differences of the initial data and forcing on its right-hand side.
In particular, its dissipation term then involves $\nabla(|v_1-v_2|^{p/2})$.

For a stopping time $0\le\theta\le T$, restriction of the forcing to $(0,\theta]$ retains its possible terminal jump.
The resulting linear solution agrees with the original solution through $\theta$ and thereafter evolves as $S_\nu(t-\theta)v(\theta)$.
Applying \eqref{ineq:3_eng-est} to this heat extension gives the estimates used in Lemmas~\ref{lem:3.3-energy-recovery} and~\ref{lem:3.4-tails-endpoint}.
A process made constant after $\theta$ is instead governed by the stopped equation, in which the Laplacian is stopped as well.
The fixed-cutoff Picard construction of Paper~I, Section~4 \cite{PodderKumar2026}, is also used in Lemma~\ref{lem:4.1-regularization}; the estimates uniform in the cutoffs are proved separately in Section~4 below.

We record the lower-semicontinuity principle of Paper~I, Lemma~3.2 \cite{PodderKumar2026}.
If $v_m,v\in L^p(\Omega;\mathbb D([0,T];L^p))$ satisfy $$ \E\sup_{t\le T}\|v_m(t)-v(t)\|_p^p\longrightarrow0, \qquad \sup_m\E\int_0^T \left\|\nabla\bigl(|v_m(s)|^{p/2}\bigr)\right\|_2^2\,\dd s<\infty, $$ then
\begin{equation}
\label{eq:2-lsc-conclusion}
\E\int_0^T \left\|\nabla\bigl(|v(s)|^{p/2}\bigr)\right\|_2^2\,\dd s \le\liminf_{m\to\infty}\E\int_0^T \left\|\nabla\bigl(|v_m(s)|^{p/2}\bigr)\right\|_2^2\,\dd s.
\end{equation}
Indeed, the powers converge strongly in $L^2(\Omega\times(0,T);L^2)$.
Along a subsequence realizing the lower limit, their gradients converge weakly in $L^2(\Omega\times(0,T);L^2(\R^3;\R^3))$.
The distributional derivative identity identifies the limit, and weak lower semicontinuity proves \eqref{eq:2-lsc-conclusion}.
The same derivative argument applies pathwise on a fixed sample's time interval whenever the powers converge distributionally and the lower limit of their gradient energies is finite; no pathwise uniform bound is inferred from an expectation bound alone.

For the nonlinear forcing estimates, we choose the exponents as in Paper~I, Section~5 \cite{PodderKumar2026}:
\begin{equation}
\label{eq:2_q_fq_h_choice}
\max\left\{\frac p2,\frac{3p}{p+1}\right\} <q_f<\frac{3p}{4}, \qquad \max\left\{\frac{3p}{2p+1},\frac p3\right\} <q_h<\frac{3p}{7}.
\end{equation}
These intervals are nonempty for every $p>3$ and lie within the ranges of Theorem~\ref{Thm:heat-exis}.
Their strict upper bounds provide the positive time exponents in the intrinsic coefficient estimates.

\subsection{Initial data and the scope of the continuation results}

Section~\ref{Sec:3} assumes only $$ u_0\in L^p(\Omega,\mathcal F_0;E) $$ and the hypotheses of Theorem~\ref{thm:2-Paper-I}.
For its maximal solution, define $$ \mathcal Q_u(t) :=\sup_{0\le s\le t}\|u(s)\|_p^p +\int_0^t\|u(s)\|_{3p}^p\,\dd s, \qquad 0\le t<\tau_{\max}.
$$ By \eqref{eq:2-local-intrinsic-energy} and the consistent closed localizations, this quantity is finite before $\tau_{\max}$, outside a single null set.
Its continuation properties and associated exit times are established in Theorem~\ref{thm:3.6-intrinsic-blowup}.

Under the additional assumptions \eqref{eq:2-taming-coercive} and \eqref{eq:2-H1}, the $H^1$ persistence result and Theorem~\ref{thm:6.1-global-continuation} use
\begin{equation}
\label{eq:2-H1-initial-data}
u_0\in L^p(\Omega,\mathcal F_0;E) \cap L^2(\Omega,\mathcal F_0;H^1_\sigma).
\end{equation}
Theorem~\ref{thm:6.2-finite-energy-global} lowers this requirement to
\begin{equation}
\label{eq:2-L2-initial-data}
u_0\in L^p(\Omega,\mathcal F_0;E) \cap L^2(\Omega,\mathcal F_0;H),
\end{equation}
with the same coefficient and taming hypotheses, by using the positive-time regularization of Proposition~\ref{prop:4.4-time-weighted-H1} and bounded restart.
Assumption~\eqref{eq:2-H1-initial-data} gives $H^1$ regularity through time zero; under \eqref{eq:2-L2-initial-data}, the $H^1$ conclusions are stated at positive times and in time-weighted form.

For data satisfying \eqref{eq:2-H1-initial-data}, the integrating factor in Section~5 is based on
\begin{equation}
\label{eq:2-integral-convention}
\Gamma(t):=\int_0^t \bigl(1+\|u(s)\|_{H^1}^{2p}+\|u(s)\|_\infty\bigr)\,\dd s.
\end{equation}
Theorem~\ref{thm:4.3-H1-persistence} gives $\Gamma((T\wedge\tau_{\max})-)<\infty$ almost surely for every deterministic $T<\infty$.
Here evaluation from the left denotes the integral over $[0,T\wedge\tau_{\max})$ and requires no terminal value of $u$.
For the finite-energy data in \eqref{eq:2-L2-initial-data}, the argument is applied after a positive restart; finiteness of \eqref{eq:2-integral-convention} from time zero is not assumed.

Global solutions are understood in the local-energy class of Definition~\ref{def:2-local-solution}.
Theorem~\ref{thm:6.4-global-continuous-dependence} establishes continuous dependence, for fixed coefficients and driving noises, in probability in the $L^p$ path-supremum and $L^p(0,T;L^{3p})$ norms on every finite interval.
Unweighted $p$th moments on deterministic intervals require a separate estimate; the additional assumptions of Proposition~\ref{prop:6.3-unweighted-intrinsic-moments} are imposed only there.

Constants may depend on $p,\nu,N$, the structural coefficient bounds, the chosen interpolation exponents, and the deterministic time horizon.
Dependence on fixed approximation parameters is permitted in the corresponding construction.
Whenever uniformity is asserted, independence of the relevant approximation parameters and stopping levels is specified in the estimate.

\section{Intrinsic and Serrin continuation criteria}
\label{Sec:3} The purpose of this section is to derive continuation criteria directly from the maximal $L^p$ theory of Paper I.
The main difficulty is that the maximal solution is a priori defined only on $[0,\tau_{\max})$, while a restart argument requires a well-defined terminal state at a stopping time.
We first refine the localizing sequence so that it approaches $\tau_{\max}$ strictly from below.
\begin{lemma}
[Strict localization of the maximal lifetime] \label{lem:3.1-strict-localization} Let $p>3$ and $E=L^p_\sigma(\R^3;\R^3)$.
Assume the hypotheses of \cref{thm:2-Paper-I}.
Let $(u,(\tau_n)_{n\ge1},\tau_{\max})$ be the maximal construction of \cref{thm:2-Paper-I}(ii), with
\begin{equation}
\label{eq:3.1-localizing-sequence}
0<\tau_n\le n, \qquad \tau_n\uparrow\tau_{\max} \quad\P\text{-a.s.}
\end{equation}
Then
\begin{equation}
\label{eq:3.1-strict-localization}
\P\left( \tau_n<\tau_{\max}\text{ for every }n\ge1 \right)=1.
\end{equation}
In particular, the sequence $(\tau_n)$ announces $\tau_{\max}$, and $\tau_{\max}$ is predictable.
\end{lemma}

\begin{proof}
Let $\mathcal E_p(v;\alpha,\beta)$ be defined in \eqref{eq:2-local-energy}.
Fix $n\ge1$, and let $U_n$ be the stopped c\`adl\`ag local representative on $[0,\tau_n]$ supplied by \cref{thm:2-Paper-I}(ii).
This representative is defined on the closed interval even before strict localization has been established.
Since $U_n$ is adapted and c\`adl\`ag, its stopped value $\xi_n:=U_n(\tau_n)$ is $\mathcal F_{\tau_n}$-measurable.
Moreover,
\begin{equation*}
\E\|\xi_n\|_p^p \le \mathcal E_p(U_n;0,\tau_n)<\infty.
\end{equation*}

Because $\tau_n\le n$, the bounded-restart result of \cref{thm:2-Paper-I}(iii), applies at $\tau_n$ with initial datum $\xi_n$.
It supplies an $(\mathcal F_t)$-stopping time $\widehat\tau_n$ and a restarted branch $\widehat v_n$ such that
\begin{equation*}
\begin{aligned}
\tau_n&<\widehat\tau_n\le n+1 \quad\P\text{-a.s.}, \\
\widehat v_n(\tau_n)&=\xi_n, \\
\mathcal E_p(\widehat v_n;\tau_n,\widehat\tau_n) &\le C\bigl(1+\E\|\xi_n\|_p^p\bigr).
\end{aligned}
\end{equation*}
Define the pasted process directly by
\begin{equation*}
w_n(t) :=
\begin{cases}
U_n(t),&0\le t\le\tau_n, \\
\widehat v_n(t\wedge\widehat\tau_n),&t>\tau_n.
\end{cases}
\end{equation*}
For each deterministic $t$, the two branches are $\mathcal F_t$-measurable on $\{t\le\tau_n\}$ and $\{\tau_n<t\}$, respectively; for the second branch this follows from the adapted stopped extension in the restart lemma.
Thus $w_n$ is adapted.
The identity $\widehat v_n(\tau_n)=U_n(\tau_n)$ gives c\`adl\`ag paths across the pasting time, and hence predictable left limits.

The stochastic intervals used in the two stopped formulations satisfy
\begin{equation*}
\mathds 1_{(0,\tau_n]} +\mathds 1_{(\tau_n,\widehat\tau_n]} =\mathds 1_{(0,\widehat\tau_n]}.
\end{equation*}
Both indicators on the left are predictable.
By , [Lemma 6.3(ii) \cite{PodderKumar2026}], the pasted process therefore satisfies the original equation through $\widehat\tau_n$, with the prescribed initial datum and driving noises.
A Poisson jump at $\tau_n$ is retained by the first interval alone, and a possible jump at $\widehat\tau_n$ is included.
Furthermore,
\begin{equation*}
\begin{aligned}
\mathcal E_p(w_n;0,\widehat\tau_n) &\le \mathcal E_p(U_n;0,\tau_n) +\mathcal E_p(\widehat v_n;\tau_n,\widehat\tau_n) \\
&\le \mathcal E_p(U_n;0,\tau_n) +C\bigl(1+\E\|\xi_n\|_p^p\bigr) <\infty.
\end{aligned}
\end{equation*}
Consequently, $\widehat\tau_n$ belongs to the family $\mathscr T$ of attainable lifetimes defining $\tau_{\max}$.
Maximality yields
\begin{equation*}
\tau_n<\widehat\tau_n\le\tau_{\max} \quad\P\text{-a.s.}
\end{equation*}
For this fix $n$ define $$A_n:=\{\omega \in \Omega | \tau_n(\omega) <\tau_{\max}(\omega)\}$$ which has probability $1$.
Taking the countable intersection over $n$ we get $\P \bigl( \bigcap_{n=1}^\infty A_n\bigr)=1$ which proves \eqref{eq:3.1-strict-localization}, in particular on $\{\tau_{\max}<\infty\}$.
Together with \eqref{eq:3.1-localizing-sequence}, this shows that, after a common null modification, $(\tau_n)$ is an announcing sequence for $\tau_{\max}$.
Hence $\tau_{\max}$ is predictable.
\end{proof}
We next record estimates for the nonlinear and stochastic coefficients on an arbitrary stochastic interval strictly contained in the maximal lifetime.
These estimates express all forcing terms solely in terms of the intrinsic $L^p$ -$L^{3p}$ control of the solution and will be used both to recover the local-energy dissipation and to control tails approaching the endpoint.
\begin{lemma}
[Intrinsic coefficient estimates] \label{lem:3.2-intrinsic-coefficients} Let $p>3$, and let $(u,(\tau_n),\tau_{\max})$ be the maximal solution of Lemma \ref{lem:3.1-strict-localization}.
For $0<t<\tau_{\max}$,
\begin{equation*}
\begin{aligned}
f(t)&=-\mathcal P^{(1)}(u(t)\otimes u(t)), & h(t)&=-\mathcal P\bigl(g_N(|u(t)|^2)u(t)\bigr), \\
g(t)&=\mathcal P\sigma(t,u(t-)), & H(t,z)&=\mathcal PG(t,u(t-),z).
\end{aligned}
\end{equation*}

Choose $q_f,\ q_h$ as in \eqref{eq:2_q_fq_h_choice} and define
\begin{equation}
\label{eq:3.2-time-exponents}
\frac1\ell=\frac1{q_f}-\frac1p, \qquad \alpha_f=\frac{3p/\ell-1}{2}, \qquad \theta_h=\frac{p/q_h-1}{2}, \qquad r_h=3(1-\theta_h).
\end{equation}

Let $0\le\rho\le\theta\le T$ be stopping times, where $T<\infty$ and $\theta<\tau_{\max}$ a.s., and set
\begin{equation*}
\delta:=\theta-\rho,\qquad M:=\sup_{\rho\le s<\theta}\|u(s)\|_p,\qquad A:=\int_\rho^\theta\|u(s)\|_{3p}^p\,\dd s,
\end{equation*}
with $M=0$ on $\{\rho=\theta\}$.
Then $M,A<\infty$ a.s.
and almost surely,
\begin{equation}
\label{eq:3.2-drift-integrability}
\int_\rho^\theta\|f(s)\|_{q_f}^p\,\dd s \le C M^{p(1+\alpha_f)} \delta^{\alpha_f}A^{1-\alpha_f}, \quad \int_\rho^\theta\|h(s)\|_{q_h}^p\,\dd s \le C M^{3p\theta_h} \delta^{1-r_h}A^{r_h},
\end{equation}
Moreover, $\mathds 1_{(\rho,\theta]}g$ and $\mathds 1_{(\rho,\theta]}H$ are respectively predictable and $\mathscr P\otimes\mathcal Z$-measurable, and
\begin{equation}
\label{eq:3.2-wiener-integrability}
\int_\rho^\theta\|g(s)\|_{\mathbb L^p}^p\,\dd s \le C\left( \delta+M^{3p/4}\delta^{1/4}A^{3/4} \right),
\end{equation}
\begin{equation}
\label{eq:3.2-poisson-integrability}
\int_\rho^\theta\int_Z\|H(s,z)\|_p^p\,\mu(\dd z)\,\dd s+ \left( \int_\rho^\theta\int_Z\|H(s,z)\|_p^2\,\mu(\dd z)\,\dd s \right)^{p/2} \le C(\delta+\delta^{p/2})(1+M^p).
\end{equation}
If, in addition, $M^p+A\le L$ a.s.
for some deterministic $L\ge1$, then every left-hand side in \eqref{eq:3.2-drift-integrability}--\eqref{eq:3.2-poisson-integrability} has expectation bounded by
\begin{equation*}
C_T(1+L^3).
\end{equation*}
Here $C$ and $C_T$ depend only on the structural parameters and the coefficient bounds, and are independent of the localizing sequence and of the auxiliary regularizations.
\end{lemma}

\begin{proof}
For $p>3$, the intervals \eqref{eq:2_q_fq_h_choice} are non empty and lie within the admissible ranges of the heat estimate in \cref{Thm:heat-exis} .
Also,
\begin{equation*}
p<\ell<3p,\qquad 0<\alpha_f<1,\qquad \frac23<\theta_h<1,\qquad 0<r_h<1,
\end{equation*}
as follows directly from \eqref{eq:2_q_fq_h_choice},\eqref{eq:3.2-time-exponents} [see Theorem 5.1 \cite{PodderKumar2026}].

The local-energy estimate \cref{thm:2-Paper-I} and the Sobolev inequality imply, simultaneously for all $n$, almost surely,
\begin{equation}
\label{eq:3.2-local-spatial-integrability}
\int_0^{\tau_n}\|u(s)\|_{3p}^p\,\dd s \le C\int_0^{\tau_n} \left\|\nabla\bigl(|u(s)|^{p/2}\bigr)\right\|_2^2\,\dd s <\infty.
\end{equation}
Since $\theta<\tau_{\max}$, each path outside a fixed null set satisfies $\theta\le\tau_n$ for some $n$.
Thus \eqref{eq:3.2-local-spatial-integrability} and the c\`adl\`ag regularity on $[0,\tau_n]$ give $M,A<\infty$.

Let $U=u(\cdot\wedge\theta)$, using the progressively measurable representative of \cref{thm:2-Paper-I}(ii).
Stopping preserves adaptedness, and the preceding pathwise localization shows that $U$ is $L^p$ valued c\`adl\`ag.
Hence $U_-$ is predictable.
The coefficient measurability assumptions and the predictability of $\mathds 1_{(\rho,\theta]}$ \eqref{eq:2-predictable-intervals},therefore give the asserted versions of $g$ and $H$.
Furthermore, on $\{\rho<\theta\}$,
\begin{equation}
\label{eq:3.2-preterminal-control}
\|u(s-)\|_p\le M,\qquad \rho<s\le\theta,
\end{equation}
and $u(s-)=u(s)$ for Lebesgue-almost every $s\in(\rho,\theta)$.
In particular, \eqref{eq:3.2-preterminal-control} includes the left limit at $\theta$, without requiring a bound on the post-jump value $u(\theta)$.

The taming hypothesis gives $|g_N(|v|^2)v|\le C_{N,\nu}|v|^3$.
Since all projection exponents used below belong to $(1,\infty)$, the Leray multiplier bounds and H\"older's inequality yield
\begin{equation}
\label{eq:3.2-basic-spatial-bounds}
\begin{aligned}
\|f(s)\|_{q_f} &\le C\|u(s)\|_p\|u(s)\|_\ell, & \|f(s)\|_{3p/4} &\le C\|u(s)\|_p\|u(s)\|_{3p}, \\
\|h(s)\|_{q_h} &\le C\|u(s)\|_{3q_h}^3, & \|h(s)\|_p &\le C\|u(s)\|_{3p}^3,
\end{aligned}
\end{equation}
for almost every $s\in(\rho,\theta)$.
The interpolation identities $\frac1\ell=\frac{\alpha_f}{p} +\frac{1-\alpha_f}{3p}, \ \frac1{3q_h}=\frac{\theta_h}{p} +\frac{1-\theta_h}{3p}$ therefore imply
\begin{equation}
\|f(s)\|_{q_f} \le C\|u(s)\|_p^{1+\alpha_f} \|u(s)\|_{3p}^{1-\alpha_f},\quad \|h(s)\|_{q_h} \le C\|u(s)\|_p^{3\theta_h} \|u(s)\|_{3p}^{r_h}.
\end{equation}
On $\{\rho<\theta\}$, H\"older's inequality in time gives
\begin{equation}
\label{eq:3.2-time-holder}
\int_\rho^\theta\|u(s)\|_{3p}^{r}\,\dd s \le\delta^{1-r/p}A^{r/p}, \qquad 0<r\le p.
\end{equation}
Applying this with $r=p(1-\alpha_f)$ and $r=pr_h$ proves \eqref{eq:3.2-drift-integrability}.
The remaining bounds in \eqref{eq:3.2-basic-spatial-bounds} give these end point estimates aditionally as, $\int_\rho^\theta\|f(s)\|_{3p/4}^p\,\dd s \le C M^pA, \int_\rho^\theta\|h(s)\|_p^{p/3}\,\dd s \le CA,$.
The estimate $\left(\int_\rho^\theta\|h(s)\|_p\,\dd s\right)^p \le C\delta^{p-3}A^3.$ follows from \eqref{eq:3.2-time-holder} with $r=3<p$.

For the Wiener coefficient, the growth assumption \eqref{eq:2-W-local} at zero and its weighted Lipschitz bound give, for $v\in L^p\cap L^{3p}$,
\begin{equation*}
\begin{aligned}
\|\mathcal P\sigma(t,v)\|_{\mathbb L^p} &\le C\left( \|\sigma(t,0)\|_{\mathbb L^p} +\|\sigma(t,v)-\sigma(t,0)\|_{\mathbb L^p} \right) \\
&\le C\left(1+\bigl\||v|^{3/2}\bigr\|_p\right) =C\left(1+\|v\|_{3p/2}^{3/2}\right).
\end{aligned}
\end{equation*}
Here the first inequality uses the ideal property of $\gamma$-radonifying operators and the $\gamma$-Fubini identification.
Since $\|v\|_{3p/2} \le\|v\|_p^{1/2}\|v\|_{3p}^{1/2},$ we obtain
\begin{equation*}
\|g(s)\|_{\mathbb L^p}^p \le C\left( 1+\|u(s)\|_p^{3p/4}\|u(s)\|_{3p}^{3p/4} \right)
\end{equation*}
for almost every $s\in(\rho,\theta)$.
Equation \eqref{eq:3.2-time-holder} with $r=3p/4$ proves the estimate in \eqref{eq:3.2-wiener-integrability}.
Finally, the jump growth assumption \eqref{eq:2-J-local}, the $L^p$ boundedness of $\mathcal P$, and \eqref{eq:3.2-preterminal-control} imply
\begin{equation*}
\int_\rho^\theta\int_Z\|H(s,z)\|_p^r \,\mu(\dd z)\,\dd s \le C\delta(1+M^r), \qquad r\in\{2,p\}.
\end{equation*}
Raising the estimate for $r=2$ to the power $p/2$ and adding the estimate for $r=p$ proves \eqref{eq:3.2-poisson-integrability}.
All integral estimates vanish on $\{\rho=\theta\}$.

If $M^p+A\le L$, substitute $M\le L^{1/p}$, $A\le L$, and $\delta\le T$ into the preceding bounds.
The largest resulting power of $L$ is $3$, since $3\theta_h+r_h=3$, and all powers of $\delta$ are positive.
The expectation bound follows.
Only the stated coefficient bounds, spatial multiplier bounds, interpolation, and H\"older's inequality enter the estimates, so their constants have the asserted independence.
\end{proof}
The intrinsic quantities controlled in Lemma \ref{lem:3.2-intrinsic-coefficients} do not, by themselves, contain the dissipation appearing in the local-energy class of \cref{thm:2-Paper-I}.
The next lemma shows that this missing dissipation can be recovered uniformly from the intrinsic bound by freezing the nonlinear forcing at a stopping time and applying the stochastic heat estimate.
\begin{lemma}
[Recovery of the local-energy dissipation] \label{lem:3.3-energy-recovery} Assume the hypotheses of Lemma \ref{lem:3.2-intrinsic-coefficients}, and retain its notation.
Fix $T>0$ and $L\ge1$.
Let $\theta$ be a stopping time satisfying $0\le\theta\le T$ and $\theta<\tau_{\max}$ almost surely.
If
\begin{equation}
\label{eq:3.3-intrinsic-bound}
\sup_{0\le s<\theta}\|u(s)\|_p^p +\int_0^\theta\|u(s)\|_{3p}^p\,\dd s \le L \qquad\P\text{-a.s.} \qquad \sup \varnothing:=0,
\end{equation}
then
\begin{equation}
\label{eq:3.3-stopped-energy}
\E\left[ \sup_{0\le s\le\theta}\|u(s)\|_p^p +\int_0^\theta \left\|\nabla\bigl(|u(s)|^{p/2}\bigr)\right\|_2^2\,\dd s \right] \le C_T\left(1+\E\|u_0\|_p^p+L^3\right).
\end{equation}
The constant has the structural dependencies of Lemma \ref{lem:3.2-intrinsic-coefficients}, together with $T$, and is independent of $\theta$, $L$, the localizing sequence.
The supremum in \eqref{eq:3.3-stopped-energy} includes a possible jump at $\theta$.

For $t<\tau_{\max}$, define
\begin{equation*}
\mathcal Q_u(t) :=\sup_{0\le s\le t}\|u(s)\|_p^p +\int_0^t\|u(s)\|_{3p}^p\,\dd s.
\end{equation*}
Set
\begin{equation}
\label{eq:3.3-terminal-localization}
\zeta_T:=T\wedge\tau_{\max},\quad \mathcal Q_u(\zeta_T-) :=\sup_{0\le t<\zeta_T}\mathcal Q_u(t),\quad B_{L,T}:=\{\mathcal Q_u(\zeta_T-)\le L\}.
\end{equation}
Then
\begin{equation}
\label{eq:3.3-localized-energy}
\E\left[ \mathds 1_{B_{L,T}} \int_0^{\zeta_T} \left\|\nabla\bigl(|u(s)|^{p/2}\bigr)\right\|_2^2\,\dd s \right] \le C_T\left(1+\E\|u_0\|_p^p+L^3\right).
\end{equation}
The integral ending at $\zeta_T$ is understood as an increasing limit from below.
In particular,
\begin{equation}
\label{eq:3.3-finite-energy}
\P\left( \mathcal Q_u(\zeta_T-)<\infty,\quad \int_0^{\zeta_T} \left\|\nabla\bigl(|u(s)|^{p/2}\bigr)\right\|_2^2\,\dd s =\infty \right)=0.
\end{equation}
\end{lemma}

\begin{proof}
Let $\theta_n:=\theta\wedge\tau_n$ and $\chi_n:=\mathds 1_{(0,\theta_n]}$.
By Lemma \ref{lem:3.1-strict-localization}, $\theta_n<\tau_{\max}$ almost surely, and \eqref{eq:3.3-intrinsic-bound} holds with $\theta_n$ in place of $\theta$.
The process $\chi_n$ is predictable.
The restrictions $\chi_n f$ and $\chi_n h$, extended by zero, have progressively measurable versions in the spatial forcing spaces of Lemma \ref{lem:3.2-intrinsic-coefficients}.
That lemma also supplies the required predictable versions of $\chi_n g$ and $\chi_n H$.
Its coefficient estimates give
\begin{equation}
\label{eq:3.3-frozen-coefficient}
\begin{aligned}
\E\Bigg[ &\int_0^{\theta_n} \left( \|f(s)\|_{q_f}^p+\|h(s)\|_{q_h}^p +\|g(s)\|_{\mathbb L^p}^p \right)\,\dd s \\
&+\int_0^{\theta_n}\int_Z \|H(s,z)\|_p^p\,\mu(\dd z)\,\dd s+\left( \int_0^{\theta_n}\int_Z \|H(s,z)\|_p^2\,\mu(\dd z)\,\dd s \right)^{p/2} \Bigg] \le C_T(1+L^3),
\end{aligned}
\end{equation}
uniformly in $n$.
Thus all forcing hypotheses of the stochastic heat theorem \cref{Thm:heat-exis} are satisfied.

Write $S_\nu(t):=e^{\nu t\Delta}$, let $U_n=u(\cdot\wedge\tau_n)$ be the stopped representative , and set $\xi_n:=U_n(\theta_n)$.
Define,
\begin{equation*}
V_n(t) :=
\begin{cases}
U_n(t),&0\le t\le\theta_n, \\
S_\nu(t-\theta_n)\xi_n,&\theta_n<t\le T.
\end{cases}
\end{equation*}
The random variable $\xi_n$ is $\mathcal F_{\theta_n}$-measurable.
For each deterministic $t$, the second branch is $\mathcal F_t$-measurable on $\{\theta_n<t\}$.
Hence $V_n$ is adapted.
Strong continuity of $S_\nu$ at zero gives c\`adl\`ag paths across $\theta_n$, with $V_n(\theta_n)=U_n(\theta_n)$.

For every $\varphi\in C_c^\infty(\R^3;\R^3)$, the semigroup identity gives, pathwise on $\{\theta_n<t\}$,
\begin{equation*}
\begin{aligned}
\left\langle S_\nu(t-\theta_n)\xi_n-\xi_n,\varphi\right\rangle =\nu\int_{\theta_n}^t \left\langle S_\nu(s-\theta_n)\xi_n,\Delta\varphi \right\rangle\,\dd s.
\end{aligned}
\end{equation*}
Combining this identity with the stopped weak formulation for $U_n$ shows that $V_n$ solves
\begin{equation}
\label{eq:3.3-frozen-heat-equation}
\begin{cases}
\dd V_n =\left[ \nu\Delta V_n+\nabla\cdot(\chi_n f)+\chi_n h \right]\dd t +\chi_n g\,\dd W_t +\displaystyle\int_Z\chi_n H(t,z)\, \widetilde N(\dd t,\dd z), \\
V_n(0)=u_0,
\end{cases}
\end{equation}
on $[0,T]$.
The Poisson integral is supported on $(0,\theta_n]$, so the possible jump at $\theta_n$ is retained exactly once.
The Leray projection remains part of the coefficients $f$ and $h$ in this equation.

Contractivity and the deterministic heat energy inequality give, for $a\in L^p(\R^3;\R^3)$,
\begin{equation}
\label{eq:3.3-heat-energy}
\sup_{0\le r\le T}\|S_\nu(r)a\|_p^p \le\|a\|_p^p, \qquad \int_0^T \left\|\nabla\bigl(|S_\nu(r)a|^{p/2}\bigr)\right\|_2^2\,\dd r \le\frac{p}{4\nu(p-1)}\|a\|_p^p.
\end{equation}
The second inequality follows by the $L^p$ heat energy identity for smooth data and approximation, using the lower-semicontinuity result.
Together with the finite local-energy moment at $\tau_n$, \eqref{eq:3.3-heat-energy} places $V_n$ in the uniqueness class of the heat theorem.
Indeed,
\begin{equation*}
\E\sup_{t\le T}\|V_n(t)\|_p^p \le\E\sup_{t\le\tau_n}\|U_n(t)\|_p^p<\infty,
\end{equation*}
and its dissipation is integrable before $\theta_n$ by \cref{thm:2-Paper-I} and after $\theta_n$ by \eqref{eq:3.3-heat-energy}.

Apply the heat estimate \eqref{ineq:3_eng-est} to \eqref{eq:3.3-frozen-heat-equation}.
The spatial change $x=\sqrt{\nu}\,y$ gives its version with viscosity $\nu$.
Using \eqref{eq:3.3-frozen-coefficient}, we obtain
\begin{equation*}
\E\left[ \sup_{t\le T}\|V_n(t)\|_p^p +\int_0^T \left\|\nabla\bigl(|V_n(s)|^{p/2}\bigr)\right\|_2^2\,\dd s \right] \le C_T\left(1+\E\|u_0\|_p^p+L^3\right).
\end{equation*}

Since $V_n=u$ on $[0,\theta_n]$, this bounds the left-hand side of \eqref{eq:3.3-stopped-energy} with $\theta_n$ in place of $\theta$.
Moreover, $\theta<\tau_{\max}$ implies that $\theta_n=\theta$ for all sufficiently large $n$ on almost every path.
Fatou's lemma therefore proves \eqref{eq:3.3-stopped-energy}, including the terminal value $u(\theta)$.

To pass to the possibly maximal endpoint, define the finite, adapted, nondecreasing c\`adl\`ag processes
\begin{equation}
\label{eq:3.3-stopped-functional}
\mathcal Q_n(t) :=\sup_{0\le s\le t\wedge\tau_n}\|u(s)\|_p^p +\int_0^{t\wedge\tau_n}\|u(s)\|_{3p}^p\,\dd s, \qquad t\ge0,
\end{equation}
and set
\begin{equation}
\label{eq:3.3-stops}
\sigma_{n,L} :=T\wedge\tau_n\wedge \inf\{t\ge0:\mathcal Q_n(t)>L\}, \qquad \inf\varnothing:=\infty.
\end{equation}
The entrance time in this definition is a stopping time by adaptedness and right continuity of $\mathcal Q_n$ and the usual conditions on the filtration.
Thus $\sigma_{n,L}$ is a stopping time, and $\sigma_{n,L}<\tau_{\max}$ almost surely.

On $\{\sigma_{n,L}>0\}$, continuity of the time integral and the definition of the entrance time imply
\begin{equation*}
\sup_{0\le s<\sigma_{n,L}}\|u(s)\|_p^p +\int_0^{\sigma_{n,L}}\|u(s)\|_{3p}^p\,\dd s =\mathcal Q_n(\sigma_{n,L}-) \le L.
\end{equation*}
The same bound holds under the empty-interval convention when $\sigma_{n,L}=0$.
Consequently, \eqref{eq:3.3-stopped-energy} applies to $\sigma_{n,L}$, uniformly in $n$.
On $B_{L,T}$, no entrance into $(L,\infty)$ occurs strictly before $T\wedge\tau_n$.
Therefore,
\begin{equation}
\label{eq:3.3-localized-event}
\sigma_{n,L}=T\wedge\tau_n \qquad\text{on }B_{L,T}.
\end{equation}
Since $\mathcal{Q}_u$ is adapted and c\`adl\`ag on every strict localization and $\zeta_T$ is predictable, $\mathcal{Q}_u(\zeta_T-)$ is $\mathcal{F}_{\zeta_T-}$-measurable, hence $B_{L,T} \in \mathcal{F}_{\zeta_T-}$.
A possible jump at $T\wedge\tau_n$ does not affect this equality.
Using \eqref{eq:3.3-stopped-energy} and \eqref{eq:3.3-localized-event}, then applying Fatou's lemma to nonnegative random variables, yields
\begin{equation*}
\begin{aligned}
\E\left[ \mathds 1_{B_{L,T}} \int_0^{\zeta_T} \left\|\nabla\bigl(|u(s)|^{p/2}\bigr)\right\|_2^2\,\dd s \right]&\le \liminf_{n\to\infty} \E\int_0^{\sigma_{n,L}} \left\|\nabla\bigl(|u(s)|^{p/2}\bigr)\right\|_2^2\,\dd s \\
&\quad\le C_T\left(1+\E\|u_0\|_p^p+L^3\right).
\end{aligned}
\end{equation*}
This proves \eqref{eq:3.3-localized-energy}.
Finally,
\begin{equation*}
\{\mathcal Q_u(\zeta_T-)<\infty\} =\bigcup_{L\in\mathbb N,\ L\ge1}B_{L,T}.
\end{equation*}
The dissipation is finite almost surely on every $B_{L,T}$ by \eqref{eq:3.3-localized-energy}.
Taking the countable union proves \eqref{eq:3.3-finite-energy}.
\end{proof}
The preceding estimate yields finite local energy up to controlled stopping times, but it does not yet provide a state at a possibly maximal endpoint.
We now use the tail estimates of Lemma \ref{lem:3.2-intrinsic-coefficients} to construct a strong $L^p$ endpoint and a c\`adl\`ag completion of the solution, including the contribution of a possible terminal Poisson jump.
\begin{lemma}
[Tail estimates and endpoint completion] \label{lem:3.4-tails-endpoint} Assume the hypotheses of Lemma \ref{lem:3.2-intrinsic-coefficients} and \ref{lem:3.3-energy-recovery}.
Fix $T>0$ and $L\ge1$.
Let $\theta$ be a stopping time such that
\begin{equation}
\label{eq:3.4-endpoint}
0\le\theta\le T\wedge\tau_{\max}, \qquad \mathcal Q_u(\theta-) :=\sup_{0\le t<\theta}\mathcal Q_u(t)\le L \quad\P\text{-a.s.},
\end{equation}
with $\mathcal Q_u(0-):=0$.
Then $u$ admits a unique adapted, stopped, $L^p_\sigma$-valued c\`adl\`ag completion $\overline u$ on $[0,\theta]$ satisfying the original equation through $\theta$ and
\begin{equation}
\label{eq:3.4-completed-energy}
\E\left[ \sup_{t\le\theta}\|\overline u(t)\|_p^p +\int_0^\theta \left\|\nabla\bigl(|\overline u(s)|^{p/2}\bigr)\right\|_2^2 \,\dd s \right] \le C_T\left(1+\E\|u_0\|_p^p+L^3\right).
\end{equation}

On $\{\theta>0\}$, the strong spatial limit $\xi_-:=\lim_{t\uparrow\theta}u(t) \ \text{ exists in }L^p(\R^3;\R^3 ) \ \P\text{-a.s.}.$ With $\xi_-:=u_0$ on $\{\theta=0\}$ and $\xi:=\overline u(\theta)$, one has
\begin{equation*}
\xi_-\in L^p(\Omega,\mathcal F_{\theta-};L^p_\sigma), \quad \xi\in L^p(\Omega,\mathcal F_\theta;L^p_\sigma).
\end{equation*}
The terminal jump is
\begin{equation}
\label{eq:3.4-terminal-jump}
\xi-\xi_- =\int_Z\mathcal PG(\theta,\xi_-,z)\, N(\{\theta\},\dd z) \quad\text{on }\{\theta>0\}.
\end{equation}
In particular, $\xi=\xi_-$ on $\{\theta=\tau_{\max}\}$.
Finally, let $\zeta_T$ and $B_{L,T}$ be as in \eqref{eq:3.3-terminal-localization}, and set $B_T:=\{\mathcal Q_u(\zeta_T-)<\infty\}$.

There is an $\mathcal F_{\zeta_T-}$-measurable $L^p_\sigma$-valued random variable $\xi_T$, set equal to zero outside $B_T$, such that
\begin{equation}
\label{eq:3.4-eventwise-endpoint}
\begin{aligned}
u(t)&\longrightarrow\xi_T \quad\text{in }L^p \quad\text{as }t\uparrow\zeta_T, \qquad\P\text{-a.s. on }B_T, \\
\E\left[\mathds 1_{B_{L,T}}\|\xi_T\|_p^p\right] &\le C_T\left(1+\E\|u_0\|_p^p+L^3\right).
\end{aligned}
\end{equation}
The constant depends only on $T$ and the structural parameters and is independent of the localizing sequence.
Moreover, on each $B_{L,T}$, $\xi_T$ coincides with the terminal value of the corresponding stopped completion at $\zeta_T$.
\end{lemma}
\begin{proof}
Put $\chi_n:=\mathds 1_{(0,\theta_n]}$ and $\theta_n:=\theta\wedge\tau_n,$.
By Lemma \ref{lem:3.1-strict-localization}, $\theta_n<\tau_{\max}$ almost surely.
The bound \eqref{eq:3.4-endpoint} implies
\begin{equation*}
\sup_{0\le s<\theta_n}\|u(s)\|_p^p +\int_0^{\theta_n}\|u(s)\|_{3p}^p\,\dd s \le L,
\end{equation*}
with the usual empty-interval convention.
Lemma \ref{lem:3.3-energy-recovery} therefore gives finite local-energy moments through each $\theta_n$.
Define the heat extensions
\begin{equation*}
V_n(t) :=
\begin{cases}
u(t),&0\le t\le\theta_n, \\
S_\nu(t-\theta_n)u(\theta_n),&\theta_n<t\le T.
\end{cases}
\end{equation*}

The adaptedness argument and the homogeneous heat continuation show that $V_n$ is an adapted c\`adl\`ag solution on $[0,T]$ of
\begin{equation*}
\begin{cases}
\dd V_n =\left[\nu\Delta V_n+\nabla\cdot(\chi_n f)+\chi_n h\right]\dd t +\chi_n g\,\dd W_t +\displaystyle\int_Z\chi_n H(t,z)\,\widetilde N(\dd t,\dd z), \\
V_n(0)=u_0.
\end{cases}
\end{equation*}
In particular, its stochastic coefficients are predictable, and a possible jump at $\theta_n$ is included.
For $m\ge n$, write $\delta_{n,m}:=\theta_m-\theta_n$ and define
\begin{equation*}
\begin{aligned}
\mathcal R_{n,m}:=\E\Bigg[ &\int_{\theta_n}^{\theta_m} \left( \|f(s)\|_{q_f}^p+\|h(s)\|_{q_h}^p +\|g(s)\|_{\mathbb L^p}^p \right)\,\dd s \\
&+\int_{\theta_n}^{\theta_m}\int_Z \|H(s,z)\|_p^p\,\mu(\dd z)\,\dd s+\left( \int_{\theta_n}^{\theta_m}\int_Z \|H(s,z)\|_p^2\,\mu(\dd z)\,\dd s \right)^{p/2} \Bigg].
\end{aligned}
\end{equation*}
Apply Lemma \ref{lem:3.2-intrinsic-coefficients} on $(\theta_n,\theta_m]$.
The intrinsic quantities on this interval satisfy $M^p+A\le L$, so
\begin{equation*}
\mathcal R_{n,m} \le C\,\E\Big[ L^2\delta_{n,m}^{\alpha_f} +L^3\delta_{n,m}^{1-r_h} +L^{3/2}\delta_{n,m}^{1/4} +(1+L) \left(\delta_{n,m}+\delta_{n,m}^{p/2}\right) \Big] \le C_T(1+L^3)\, \E\delta_{n,m}^{\beta}.
\end{equation*}
where $\beta:=\min\left\{\alpha_f,1-r_h,\frac14\right\}>0$.
Since $\delta_{n,m}\le\theta-\theta_n\le T$ and $\theta_n\uparrow\theta$, dominated convergence gives
\begin{equation*}
\sup_{m\ge n}\mathcal R_{n,m}\longrightarrow0.
\end{equation*}

Consequently, $(\chi_n f,\chi_n h,\chi_n g,\chi_n H)$ is Cauchy in the forcing spaces of the stochastic heat \cref{Thm:heat-exis}, including both Poisson spaces $L^p\bigl(\Omega;L^p((0,T)\times Z;L^p)\bigr), \ L^p\bigl(\Omega;L^2((0,T)\times Z;L^p)\bigr),$ where the time-mark measure is $\dd t\otimes\mu$.
Denote the limits by $(f^\theta,h^\theta,g^\theta,H^\theta)$.
Progressive and predictable versions are preserved under these norm limits as subsequences converge almost everywhere in the respective product measures, and all spatial target spaces are separable.

Let $V$ be the solution supplied by the heat theorem with these limiting coefficients and initial datum $u_0$.
The uniform coefficient bounds of Lemma \ref{lem:3.2-intrinsic-coefficients} give
\begin{equation}
\label{eq:3.4-limit-heat-energy}
\E\left[ \sup_{t\le T}\|V(t)\|_p^p +\int_0^T \left\|\nabla\bigl(|V(s)|^{p/2}\bigr)\right\|_2^2\,\dd s \right] \le C_T\left(1+\E\|u_0\|_p^p+L^3\right).
\end{equation}

Here and below the heat theorem is used with viscosity $\nu$, by the spatial rescaling already used in Lemma \ref{lem:3.3-energy-recovery}.
The difference $V_m-V_n$ solves the same linear equation with zero initial datum and forcing supported on the predictable interval $(\theta_n,\theta_m]$.
Its heat estimate yields
\begin{equation*}
\E\sup_{t\le T}\|V_m(t)-V_n(t)\|_p^p \le C_T\mathcal R_{n,m}.
\end{equation*}
Applying stochastic heat estimates to $V-V_n$ and passing the limit of the forcing coefficient gives,
\begin{equation*}
\E\sup_{t\le T}\|V(t)-V_n(t)\|_p^p \le C_T(1+L^3)\, \E(\theta-\theta_n)^\beta
\end{equation*}
Then by Dominated convergence theorem one gets, $\E\sup_{t\le T}\|V(t)-V_n(t)\|_p^p \longrightarrow 0$ In particular, a subsequence converges to $V$ uniformly on $[0,T]$ almost surely.
For each fixed $n$, every $V_m$ with $m\ge n$ agrees with $u$ on $[0,\theta_n]$.
Taking a countable intersection therefore gives
\begin{equation}
\label{eq:3.4-agreement}
V(t)=u(t)\quad\text{for all }0\le t\le\theta_n,\ n\ge1, \qquad V(t)=u(t)\quad\text{for all }0\le t<\theta,
\end{equation}
outside one null set.
The process $V$ is divergence-free by this uniform convergence, since each $V_n$ is divergence-free.
Its c\`adl\`ag regularity is supplied by the heat theorem and the convergence above is in the expected supremum norm.

Set $\chi_\theta:=\mathds 1_{(0,\theta]}$.
For $s<\theta$, the coefficients $\chi_n f,\chi_n h, \chi_n g,\chi_n H$ eventually agree with the original coefficients, whereas they vanish for $s>\theta$.
Using \eqref{eq:3.4-agreement}, their norm limits are therefore identified, in their respective product measures, as
\begin{equation}
\label{eq:3.4-limit-coefficients}
\begin{aligned}
f^\theta(s) &=-\chi_\theta(s)\mathcal P^{(1)}(V(s)\otimes V(s)), & h^\theta(s) &=-\chi_\theta(s)\mathcal P(g_N(|V(s)|^2)V(s)), \\
g^\theta(s) &=\chi_\theta(s)\mathcal P\sigma(s,V(s-)), & H^\theta(s,z) &=\chi_\theta(s)\mathcal PG(s,V(s-),z).
\end{aligned}
\end{equation}
The graph of $\theta$ has zero Lebesgue-time measure.
The last two expressions are predictable, with the corresponding product measurability for $H^\theta$.
Thus their equality in the stochastic integrand norms identifies the stochastic integrals indistinguishably, including their terminal values.

Define $\overline u(t):=V(t\wedge\theta)$ for $t\ge0$.
It is adapted, stopped, and c\`adl\`ag.
Stopping the limiting heat equation and using \eqref{eq:3.4-limit-coefficients} gives, for every $\varphi\in C_c^\infty(\R^3;\R^3)$, indistinguishably in $t$,
\begin{equation}
\label{eq:3.4-completed-weak-equation}
\begin{aligned}
\langle\overline u(t),\varphi\rangle =&\langle u_0,\varphi\rangle+\int_0^{t\wedge\theta} \Big[ \nu\langle\overline u(s),\Delta\varphi\rangle +\langle\mathcal P^{(1)} (\overline u(s)\otimes\overline u(s)),\nabla\varphi\rangle -\langle\mathcal P(g_N(|\overline u(s)|^2)\overline u(s)), \varphi\rangle \Big]\,\dd s \\
&+\int_0^{t\wedge\theta} \left\langle (\mathcal P\sigma(s,\overline u(s-)))^*\varphi,\dd W_s \right\rangle_{\mathcal U}+\int_0^{t\wedge\theta}\int_Z \left\langle\mathcal PG(s,\overline u(s-),z),\varphi\right\rangle \,\widetilde N(\dd s,\dd z).
\end{aligned}
\end{equation}
The stochastic integrals in this identity are stopped integrals with the predictable indicator $\chi_\theta$.
Their upper endpoints are included.
Equation \eqref{eq:3.4-limit-heat-energy} implies \eqref{eq:3.4-completed-energy}.

By \eqref{eq:3.4-agreement}, $V(t)=u(t)$ for every $t<\theta$ outside a fixed null set.
Since $V$ has $L^p$ valued c\`adl\`ag paths, its strong left limit $V(\theta-)$ exists on $\{\theta>0\}$, and therefore
\begin{equation*}
\xi_-:=\lim_{t\uparrow\theta}u(t)=V(\theta-) \qquad\text{in }L^p_\sigma.
\end{equation*}
The left-limit process $V_-$ is predictable, so evaluation at the stopping time $\theta$ gives $V(\theta-)\in\mathcal F_{\theta-}$, whereas adaptedness and c\`adl\`ag regularity give $V(\theta)\in\mathcal F_\theta$.
Together with \eqref{eq:3.4-limit-heat-energy}, this yields
\begin{equation*}
\xi_-\in L^p(\Omega,\mathcal F_{\theta-};L^p_\sigma), \qquad \xi:=V(\theta)\in L^p(\Omega,\mathcal F_\theta;L^p_\sigma).
\end{equation*}

The finite-variation terms in \eqref{eq:3.4-completed-weak-equation} are continuous in time, and the Wiener integral has continuous paths.
Hence the only possible jump at $\theta$ is contributed by the compensated Poisson integral.
For a measurable $Z_j\subset Z$ with $\mu(Z_j)<\infty$, its jump identity gives
\begin{equation*}
\Delta_\theta \int_0^\cdot\int_{Z_j} \mathcal PG(s,\overline u(s-),z)\, \widetilde N(\dd s,\dd z) = \int_{Z_j} \mathcal PG(\theta,\xi_-,z)\, N(\{\theta\},\dd z),
\end{equation*}
because the compensator $\dd s\,\mu(\dd z)$ has no atom at the singleton $\{\theta\}$.
Letting $Z_j\uparrow Z$ and using the convergence in the $L^p$ and $L^2$ Poisson norms, together with the Poisson maximal estimate , extends this identity to the full mark space.
A countable separating family of test functions then yields the $L^p_\sigma$-valued identity
\begin{equation*}
\xi-\xi_- = \int_Z\mathcal PG(\theta,\xi_-,z)\, N(\{\theta\},\dd z).
\end{equation*}
Thus the preterminal path uniquely determines $\xi_-$, and the preceding jump identity uniquely determines the terminal value $\xi$, hence the stopped completion is unique.

It remains to identify the endpoint on $\{\theta=\tau_{\max}\}$.
Choose measurable sets $Z_j\uparrow Z$ with $\mu(Z_j)<\infty$.
By Lemma \ref{lem:3.1-strict-localization}, $\tau_{\max}$ is predictable, so its graph is a predictable set.
The compensation formula therefore yields
\begin{equation*}
\begin{aligned}
\E\int_0^T\int_{Z_j} \mathds 1_{\{s=\tau_{\max}\}}\,N(\dd s,\dd z) &= \mu(Z_j)\, \E\int_0^T \mathds 1_{\{s=\tau_{\max}\}}\,\dd s =0.
\end{aligned}
\end{equation*}
Since the random variable on the left is nonnegative,
\begin{equation*}
N(\{\tau_{\max}\}\times Z_j)=0 \qquad\text{a.s. on }\{\tau_{\max}\le T\}.
\end{equation*}
Letting $j\uparrow\infty$ gives $N(\{\tau_{\max}\}\times Z)=0$ on the same event.
In particular, on $\{\theta=\tau_{\max}\}$, \eqref{eq:3.4-terminal-jump} implies $\xi=\xi_-$.

Finally, on $\{\theta<\tau_{\max}\}$ one has $\theta_n=\theta$ for all sufficiently large $n$, whereas on $\{\theta=\tau_{\max}\}$ strict localization gives $\theta_n=\tau_n<\theta$ and $\theta_n\uparrow\theta$.
Hence, using \eqref{eq:3.4-agreement},
\begin{equation*}
u(\theta_n)=V(\theta_n)\longrightarrow
\begin{cases}
V(\theta),&\theta<\tau_{\max}, \\
V(\theta-)=V(\theta),&\theta=\tau_{\max},
\end{cases}
=\xi \qquad\text{a.s. in }L^p.
\end{equation*}
Moreover,
\begin{equation*}
\|u(\theta_n)-\xi\|_p^p \le 2^p\sup_{t\le T}\|V(t)\|_p^p,
\end{equation*}
and the right-hand side is integrable by \eqref{eq:3.4-limit-heat-energy}.
Dominated convergence therefore yields
\begin{equation*}
\E\|u(\theta_n)-\xi\|_p^p\longrightarrow0.
\end{equation*}

It remains to obtain the eventwise assertion without a deterministic bound on $\mathcal Q_u(\zeta_T-)$.
For each $L\ge1$, use the stopping times $\sigma_{n,L}$ of \eqref{eq:3.3-stops}.
Their stopped intrinsic functionals agree up to the smaller localizing lifetime, hence $\sigma_{n,L}\le\sigma_{n+1,L}$.
Set
\begin{equation}
\label{eq:3.4-budget-endpoint}
\theta^{(L)}:=\lim_{n\to\infty}\sigma_{n,L} \le\zeta_T.
\end{equation}
For every $t<\theta^{(L)}$, some $\sigma_{n,L}>t$.
The pre-exit bounds in Lemma \ref{lem:3.3-energy-recovery} therefore imply $\mathcal Q_u(\theta^{(L)}-)\le L \ \P\text{-a.s.}$ Apply the result already proved with $\theta=\theta^{(L)}$, and denote the resulting heat process by $V^{(L)}$.
By \eqref{eq:3.3-localized-event}, $\theta^{(L)}=\zeta_T \ \text{on }B_{L,T}.$ Thus $V^{(L)}(\zeta_T-)$ is the required left limit on $B_{L,T}$, and the limits agree on overlaps.

For completeness, $B_{L,T}\in\mathcal F_{\zeta_T-}$.
By right continuity of $\mathcal Q_u$ before its lifetime, $$B_{L,T} =\bigcap_{r\in\mathbb Q\cap[0,T)} \left( \{\zeta_T\le r\} \cup \{\zeta_T>r,\ \mathcal Q_u(r)\le L\} \right),$$ where $\mathcal Q_u(r)$ is used only on $\{r<\zeta_T\}$.
Each set belongs to $\mathcal F_{\zeta_T-}$ by the definition of that sigma-field.
Since $V^{(L)}(\zeta_T-)$ is also $\mathcal F_{\zeta_T-}$-measurable, the compatible variables at integer levels define $\xi_T$ on $B_T=\bigcup_{k\in\mathbb N}B_{k,T}$, with value zero elsewhere.
We define it as
\begin{equation*}
B_{0,T} := \varnothing, \qquad \xi_T := \sum_{k=1}^{\infty} \mathbf{1}_{B_{k,T} \setminus B_{k-1,T}} V^{(k)}(\zeta_{T}-) \quad \text{on } B_T,
\end{equation*}
and $\xi_T = 0$ on $B_T^c$.
Compatibility on overlaps makes this independent of the chosen integer level.
If $\omega\in B_T$, then $\omega\in B_{k,T}$ for some $k\in\mathbb N$.
On $B_{k,T}$ one has $\theta^{(k)}=\zeta_T$, and hence the endpoint completion constructed above gives $$ u(t)\longrightarrow V^{(k)}(\zeta_T-)=\xi_T \qquad\text{in }L^p \quad\text{as }t\uparrow\zeta_T.
$$ Thus the first assertion in \eqref{eq:3.4-eventwise-endpoint} holds almost surely on $B_T$.

Moreover, on $B_{L,T}$ compatibility of the levelwise completions gives $\xi_T=V^{(L)}(\zeta_T-)$.
Therefore, $$ \E \left[ \mathds 1_{B_{L,T}}\|\xi_T\|_p^p \right] \le \E\sup_{t\le T}\|V^{(L)}(t)\|_p^p\le C_T\left(1+\E\|u_0\|_p^p+L^3\right), $$ by \eqref{eq:3.4-limit-heat-energy} applied to the level-$L$ completion.
This proves \eqref{eq:3.4-eventwise-endpoint}.
Finally, $\zeta_T=T\wedge\tau_{\max}$ is predictable.
Choose measurable sets $Z_j\uparrow Z$ with $\mu(Z_j)<\infty$.
Since the graph of $\zeta_T$ is predictable, the compensation formula yields
\begin{equation*}
\E\int_0^T\int_{Z_j} \mathds 1_{\{s=\zeta_T\}}\,N(\dd s,\dd z) = \mu(Z_j)\, \E\int_0^T \mathds 1_{\{s=\zeta_T\}}\,\dd s =0.
\end{equation*}
Hence $N(\{\zeta_T\}\times Z_j)=0 \ \P\text{-a.s.}$ for every $j$, and therefore, letting $j\uparrow\infty$, $N(\{\zeta_T\}\times Z)=0 \ \P\text{-a.s.}$ On $B_{L,T}$ one has $\theta^{(L)}=\zeta_T$, so the terminal-jump identity for the corresponding completion gives
\begin{equation*}
V^{(L)}(\zeta_T)-V^{(L)}(\zeta_T-) = \int_Z \mathcal PG\bigl(\zeta_T,V^{(L)}(\zeta_T-),z\bigr) \,N(\{\zeta_T\},\dd z) =0.
\end{equation*}
Consequently,
\begin{equation*}
V^{(L)}(\zeta_T)=V^{(L)}(\zeta_T-)=\xi_T \qquad\text{on }B_{L,T},
\end{equation*}
so the left limit agrees with the completed terminal value, as asserted.

\end{proof}
Once such an endpoint value is available, the bounded-restart theorem \cref{thm:2-Paper-I}(iii) can be applied there.
The following proposition carries out this restart and pastes the new branch to the original solution, thereby showing that every endpoint with finite intrinsic control is strictly extendible.
\begin{proposition}
[Localized continuation] \label{prop:3.5-localized-continuation} Assume the hypotheses of Lemmas \ref{lem:3.1-strict-localization}--\ref{lem:3.4-tails-endpoint}, including the joint compatibility hypothesis of the bounded-restart theorem \cref{thm:2-Paper-I}(iii).
Let $(u,(\tau_n),\tau_{\max})$ be the prescribed maximal local strong solution, and fix $T>0$ and $L\ge1$.
Suppose that $\theta$ is an $(\mathcal F_t)$-stopping time satisfying
\begin{equation*}
0\le\theta\le T\wedge\tau_{\max},\quad \mathcal Q_u(\theta-) :=\sup_{0\le t<\theta}\mathcal Q_u(t)\le L \qquad\P\text{-a.s.},
\end{equation*}
where $\mathcal Q_u(0-):=0$.
Denote by $\overline u$ the stopped completion through $\theta$ supplied by Lemma \ref{lem:3.4-tails-endpoint}.

There exist an $(\mathcal F_t)$-stopping time $\widehat\theta$ and a local strong solution $(w,\widehat\theta)$ on the original stochastic basis, with the prescribed initial datum and driving noises, such that
\begin{equation}
\label{eq:3.5-extended-lifetime}
\theta<\widehat\theta \le\theta+1\le T+1, \ \widehat\theta \le\tau_{\max} \qquad\P\text{-a.s.}
\end{equation}
The stopped process $w$ is adapted, divergence-free, and $L^p$ valued c\`adl\`ag, and
\begin{equation}
\label{eq:3.5-extension}
\P\left(
\begin{aligned}
w(t)&=\overline u(t) &&\text{for every }0\le t\le\theta, \\
w(t)&=u(t) &&\text{for every }0\le t<\widehat\theta
\end{aligned}
\right)=1.
\end{equation}
Moreover,
\begin{equation*}
\E\left[ \sup_{0\le t\le\widehat\theta}\|w(t)\|_p^p +\int_0^{\widehat\theta} \left\|\nabla\bigl(|w(s)|^{p/2}\bigr)\right\|_2^2\,\dd s \right] \le C_T\left(1+\E\|u_0\|_p^p+L^3\right).
\end{equation*}
Here $C_T$ depends only on $T$ and the structural constants of the preceding lemmas and the local theorem of Paper I; it is independent of $L$, $\theta$, the localizing sequence.
Any two such extensions agree almost surely on their common closed lifetime.

In particular, for $\zeta_T,\ B_{L,T}$ and $\theta^{(L)}$ being the stopping time defined in \eqref{eq:3.4-budget-endpoint}, there is an extension $(w^{(L)},\widehat\theta^{(L)})$ with the preceding properties for $\theta=\theta^{(L)}$, and
\begin{equation}
\label{eq:3.5-continuation}
\widehat\theta^{(L)}>\zeta_T \qquad\P\text{-a.s. on }B_{L,T}.
\end{equation}
\end{proposition}

\begin{proof}
\emph{Step 1} By Lemma \ref{lem:3.4-tails-endpoint}, $\overline u$ is an adapted, stopped c\`adl\`ag process satisfying the original equation through $\theta$.
Evaluation at the stopping time gives
\begin{equation}
\label{eq:3.5-restart-datum}
\xi:=\overline u(\theta) \in L^p(\Omega,\mathcal F_\theta;L^p_\sigma),\quad \E\|\xi\|_p^p \le\E\sup_{t\le\theta}\|\overline u(t)\|_p^p \le C_T\left(1+\E\|u_0\|_p^p+L^3\right).
\end{equation}
Since $\theta\le T$, \cref{thm:2-Paper-I}(iii), applies at $\theta$ with initial datum $\xi$.
It supplies a stopping duration $0<\vartheta\le1$ in the shifted filtration and an original-filtration stopping time $\widehat\theta:=\theta+\vartheta$, together with a restarted branch $\widehat v$ on $[\theta,\widehat\theta]$ satisfying $\widehat v(\theta)=\xi,$ and
\begin{equation}
\label{eq:3.5-restarted-branch}
\E\left[ \sup_{\theta\le t\le\widehat\theta}\|\widehat v(t)\|_p^p +\int_\theta^{\widehat\theta} \left\|\nabla\bigl(|\widehat v(s)|^{p/2}\bigr)\right\|_2^2 \,\dd s \right] \le C\left(1+\E\|\xi\|_p^p\right).
\end{equation}
The translated equation uses the original coefficients and noises on $(\theta,\widehat\theta]$.
The restart datum is the completed value $\xi$, including the possible terminal jump in \eqref{eq:3.4-terminal-jump}.
The restart theorem also applies on $\{\theta=0\}$, where $\xi=u_0$.

\emph{Step 2:} Define,
\begin{equation*}
w(t):=
\begin{cases}
\overline u(t),&0\le t\le\theta, \\
\widehat v(t\wedge\widehat\theta),&t>\theta.
\end{cases}
\end{equation*}
For every deterministic $t$, the first branch is $\mathcal F_t$-measurable on $\{t\le\theta\}$.
On $\{\theta<t\}$, measurability of the second branch follows from the adapted stopped extension supplied by the restart theorem.
Thus $w$ is adapted.
The identity $\widehat v(\theta)=\overline u(\theta)$ gives c\`adl\`ag paths at the junction.
The process is divergence-free and stopped at $\widehat\theta$.
Its left-limit process is predictable.
For $\varphi\in C_c^\infty(\R^3;\R^3)$, write
\begin{equation*}
\mathcal B(a) :=\nu\langle a,\Delta\varphi\rangle +\langle\mathcal P^{(1)}(a\otimes a),\nabla\varphi\rangle -\langle\mathcal P(g_N(|a|^2)a),\varphi\rangle, \qquad a\in L^p_\sigma.
\end{equation*}
These pairings are well-defined: the tensor and taming terms belong to $L^{p/2}$ and $L^{p/3}$, respectively, and $p/3>1$.
As in Lemma \ref{lem:3.2-intrinsic-coefficients}, $\mathcal P^{(1)}$ acts on the first tensor index.
Set
\begin{equation*}
\chi_0:=\mathds 1_{(0,\theta]},\quad \chi_1:=\mathds 1_{(\theta,\widehat\theta]} =\mathds 1_{(0,\widehat\theta]}-\mathds 1_{(0,\theta]},\quad \chi_0+\chi_1=\mathds 1_{(0,\widehat\theta]}=:\chi.
\end{equation*}
These indicators are predictable, being adapted and left-continuous on $(0,\infty)$.
The states and their left limits agree with those of $\overline u$ on the first interval and those of $\widehat v$ on the second.
Adding the completed stopped equation \eqref{eq:3.4-completed-weak-equation} to the translated restart increment therefore yields, indistinguishably for each $\varphi$,
\begin{equation}
\label{eq:3.5-weak}
\begin{aligned}
\langle w(t),\varphi\rangle =&\langle u_0,\varphi\rangle +\int_0^t\chi(s)\mathcal B(w(s))\,\dd s \\
&+\int_0^t\chi(s) \left\langle \bigl(\mathcal P\sigma(s,w(s-))\bigr)^*\varphi, \dd W_s \right\rangle_{\mathcal U}+\int_0^t\int_Z\chi(s) \left\langle\mathcal PG(s,w(s-),z),\varphi\right\rangle \,\widetilde N(\dd s,\dd z), \qquad t\ge0.
\end{aligned}
\end{equation}
Only the possible $\omega$-dependence of the coefficients has been suppressed.
The stochastic integrands are predictable, with the additional $\mathcal Z$-measurability for the jump term.
Their integrability is inherited by restriction and addition from the two constituent equations.
The first Poisson interval contains $\theta$, whereas the second excludes $\theta$ and contains $\widehat\theta$.
Hence a jump at the pasting time is counted exactly once, and a terminal jump at $\widehat\theta$ is retained.
When $\theta=0$, the first interval is empty and the same identity is the restarted equation with initial value $u_0$.

\emph{Step 3:} The piecewise definition gives
\begin{equation}
\label{eq:3.5-energy-bound}
\begin{aligned}
\E\left[ \sup_{t\le\widehat\theta}\|w(t)\|_p^p +\int_0^{\widehat\theta} \left\|\nabla\bigl(|w(s)|^{p/2}\bigr)\right\|_2^2\,\dd s \right]&\le \E\left[ \sup_{t\le\theta}\|\overline u(t)\|_p^p +\int_0^\theta \left\|\nabla\bigl(|\overline u(s)|^{p/2}\bigr)\right\|_2^2 \,\dd s \right] \\
&\qquad+ \E\left[ \sup_{\theta\le t\le\widehat\theta}\|\widehat v(t)\|_p^p +\int_\theta^{\widehat\theta} \left\|\nabla\bigl(|\widehat v(s)|^{p/2}\bigr)\right\|_2^2 \,\dd s \right] \\
&\quad\le C_T\left(1+\E\|u_0\|_p^p+L^3\right),
\end{aligned}
\end{equation}
by \eqref{eq:3.4-completed-energy}, \eqref{eq:3.5-restart-datum}, and \eqref{eq:3.5-restarted-branch}.
The Sobolev inequality on $\R^3$, applied to $|w|^{p/2}$, also yields
\begin{equation*}
\E\int_0^{\widehat\theta}\|w(s)\|_{3p}^p\,\dd s \le C\E\int_0^{\widehat\theta} \left\|\nabla\bigl(|w(s)|^{p/2}\bigr)\right\|_2^2\,\dd s <\infty.
\end{equation*}
Consequently, \eqref{eq:3.5-weak} and \eqref{eq:3.5-energy-bound} place $(w,\widehat\theta)$ in the local-energy uniqueness class of \cref{thm:2-Paper-I}(i).
Its lifetime is strictly positive even if $\theta$ can vanish.
Let $\mathscr T$ be the family of attainable lifetimes in \cref{thm:2-Paper-I}(ii).
The preceding construction shows that $\widehat\theta\in\mathscr T$; hence
\begin{equation*}
\theta<\widehat\theta \le\operatorname*{ess\,sup}_{\sigma\in\mathscr T}\sigma =\tau_{\max} \qquad\P\text{-a.s.}
\end{equation*}
Together with $\vartheta\le1$, this proves \eqref{eq:3.5-extended-lifetime}.
Pathwise uniqueness gives agreement of $w$ and the stopped representative of $u$ on $[0,\widehat\theta\wedge\tau_n]$ for each $n$.
Taking a common full-probability event and using $\tau_n\uparrow\tau_{\max}$ proves \eqref{eq:3.5-extension}.
The same uniqueness statement gives agreement of any two extensions on their common closed lifetime.
Finally, Lemma \ref{lem:3.4-tails-endpoint} gives $0\le\theta^{(L)}\le\zeta_T, \ \mathcal Q_u(\theta^{(L)}-)\le L \quad\P\text{-a.s.},\ \theta^{(L)}=\zeta_T \ \P\text{-a.s.
on }B_{L,T}.$ Apply the construction on the whole probability space with $\theta=\theta^{(L)}$ and its adapted completion.
It yields $(w^{(L)},\widehat\theta^{(L)})$ with $\widehat\theta^{(L)}>\theta^{(L)}$ almost surely.
Restricting this inequality to $B_{L,T}$ proves \eqref{eq:3.5-continuation}.
\end{proof}
The continuation mechanism above immediately yields an intrinsic blow-up criterion for the maximal solution.
We also introduce the associated level exit times, which provide an increasing sequence of stopping times exhausting $\tau_{\max}$ and will serve as the canonical localizations in the sequel.
\begin{theorem}
[Intrinsic blow-up alternative] \label{thm:3.6-intrinsic-blowup} Let $p>3$, assume the hypotheses of Proposition \ref{prop:3.5-localized-continuation}, and let $(u,(\tau_n),\tau_{\max})$ be the prescribed maximal $L^p_\sigma(\R^3;\R^3)$-solution.
For $0\le t<\tau_{\max}$, set
\begin{equation*}
\mathcal Q_u(t) :=\sup_{0\le s\le t}\|u(s)\|_p^p +\int_0^t\|u(s)\|_{3p}^p\,\dd s.
\end{equation*}
Then
\begin{equation}
\label{eq:3.6-blowup-alternative}
\P\left( \tau_{\max}<\infty,\quad \lim_{t\uparrow\tau_{\max}}\mathcal Q_u(t)<\infty \right)=0.
\end{equation}
In particular,
\begin{equation}
\label{eq:3.6-intrinsic-divergence}
\lim_{t\uparrow\tau_{\max}} \left( \sup_{0\le s\le t}\|u(s)\|_p^p +\int_0^t\|u(s)\|_{3p}^p\,\dd s \right)=\infty \quad\P\text{-a.s. on }\{\tau_{\max}<\infty\}.
\end{equation}

For each deterministic $L\ge1$, define the intrinsic exit time
\begin{equation}
\label{eq:3.6-intrinsic-exits}
\lambda_L :=\inf\{t\in[0,\tau_{\max}):\mathcal Q_u(t)\ge L\} \wedge\tau_{\max}, \qquad \inf\varnothing:=\infty.
\end{equation}
These are $(\mathcal F_t)$-stopping times, nondecreasing in $L$, and
\begin{equation}
\label{eq:3.6-exit-lifetime}
\lambda_L\uparrow\tau_{\max}\ \text{as }L\to\infty,\ \P\text{-a.s.}, \quad \lambda_L<\tau_{\max}\ \P\text{-a.s. on }\{\tau_{\max}<\infty\}.
\end{equation}
With $\mathcal Q_u(\lambda_L-) :=\sup_{0\le t<\lambda_L}\mathcal Q_u(t)$ and the convention $\mathcal Q_u(0-)=0$, one also has
\begin{equation}
\label{eq:3.6-exit-levels}
\begin{aligned}
\mathcal Q_u(\lambda_L-)\le L \ \P\text{-a.s.}, \qquad \mathcal Q_u(\lambda_L) \ge L \ \P\text{-a.s. on }\{\lambda_L<\tau_{\max}\}.
\end{aligned}
\end{equation}
The definition allows $\lambda_L=0$ when $\|u_0\|_p^p\ge L$.
A jump may produce strict inequality in \eqref{eq:3.6-exit-levels}.
\end{theorem}

\begin{proof}
The local regularity from \cref{thm:2-Paper-I}(i) implies, on a common full-probability event, that $u$ is locally c\`adl\`ag in $L^p$ and belongs locally to $L^p(0,\tau_{\max};L^{3p})$.
Consequently, $\mathcal Q_u$ is finite, nondecreasing, and c\`adl\`ag on $[0,\tau_{\max})$.
In particular, its limit at the maximal lifetime exists in $[0,\infty]$ and equals
\begin{equation*}
\mathcal Q_u(\tau_{\max}-) :=\sup_{0\le t<\tau_{\max}}\mathcal Q_u(t).
\end{equation*}

\emph{Step 1:} Fix $T>0$ and $L\ge1$, and retain the notation $\zeta_T, B_{L,T}$ and define $A_{L,T}:=\{\tau_{\max}\le T\}\cap B_{L,T}$.
The endpoint functional and these events are measurable by Lemma \ref{lem:3.4-tails-endpoint}.
Proposition \ref{prop:3.5-localized-continuation}, applied to the stopping time $\theta^{(L)}$ of \eqref{eq:3.4-budget-endpoint}, supplies an attainable lifetime $\widehat\theta^{(L)}$ such that
\begin{equation*}
\widehat\theta^{(L)}\le\tau_{\max} \ \P\text{-a.s.},\qquad \widehat\theta^{(L)}>\zeta_T \ \P\text{-a.s. on }B_{L,T}.
\end{equation*}
On $A_{L,T}$, these relations give $\tau_{\max}=\zeta_T<\widehat\theta^{(L)}\le\tau_{\max}$.
Therefore
\begin{equation}
\label{eq:3.6-null-event}
\P(A_{L,T})=0.
\end{equation}
The continuation is constructed on the whole probability space at $\theta^{(L)}$, the event $A_{L,T}$ is used only to compare the resulting lifetimes.
Since
\begin{equation*}
\{\tau_{\max}<\infty,\ \mathcal Q_u(\tau_{\max}-)<\infty\} =\bigcup_{m=1}^{\infty}\bigcup_{k=1}^{\infty}A_{k,m},
\end{equation*}
the countable union of \eqref{eq:3.6-null-event} proves \eqref{eq:3.6-blowup-alternative}.
Monotonicity of $\mathcal Q_u$ then gives \eqref{eq:3.6-intrinsic-divergence}.

\emph{Step 2:} Fix $L\ge1$.
On $\{\lambda_L<\tau_{\max}\}$, the defining set in \eqref{eq:3.6-intrinsic-exits} is nonempty.
Taking times in this set converging to $\lambda_L$ from the right and using right continuity gives $\mathcal Q_u(\lambda_L)\ge L$.
For every $t<\lambda_L$, one has $\mathcal Q_u(t)<L$; hence $\mathcal Q_u(\lambda_L-)\le L$, including the case $\lambda_L=0$ by convention.
This proves \eqref{eq:3.6-exit-levels}.

For deterministic $t\ge0$, monotonicity and the preceding property yield
\begin{equation}
\label{eq:3.6-exit-measurability}
\{\lambda_L\le t\} =\{\tau_{\max}\le t\} \cup\{t<\tau_{\max},\ \mathcal Q_u(t)\ge L\}.
\end{equation}
To verify the measurability of the second event, let $\mathcal Q_n$ be the adapted stopped functional in \eqref{eq:3.3-stopped-functional}.
Compatibility of the stopped representatives and $\tau_n\uparrow\tau_{\max}$ give
\begin{equation*}
\{t<\tau_{\max},\ \mathcal Q_u(t)\ge L\} =\bigcup_{n=1}^{\infty} \bigl(\{t<\tau_n\}\cap\{\mathcal Q_n(t)\ge L\}\bigr) \in\mathcal F_t.
\end{equation*}
Together with \eqref{eq:3.6-exit-measurability} and the completeness of the filtration, this proves that $\lambda_L$ is a stopping time.

\emph{Step 3:} The defining level sets show that $\lambda_L$ is nondecreasing in $L$ and bounded above by $\tau_{\max}$.
Set $\lambda_\infty:=\lim_{k\to\infty}\lambda_k$, where $k$ ranges over the positive integers.
Fix a path on which the preceding regularity holds and let $t<\tau_{\max}$.
Since $\mathcal Q_u(t)<\infty$, choose an integer $k>\max\{1,\mathcal Q_u(t)\}$.
Equation \eqref{eq:3.6-exit-measurability} gives $\lambda_k>t$.
Thus $\lambda_\infty\ge t$ for every $t<\tau_{\max}$, and consequently $\lambda_\infty=\tau_{\max}$.
Monotonicity extends this convergence to real $L\to\infty$.

Finally, on $\{\tau_{\max}<\infty\}$ outside the null event in \eqref{eq:3.6-blowup-alternative}, every finite level $L$ is exceeded at some time $t<\tau_{\max}$.
Therefore $\lambda_L\le t<\tau_{\max}$.
This proves \eqref{eq:3.6-exit-lifetime} and completes the proof.
\end{proof}
The intrinsic criterion still involves both the running $L^p$ norm and the accumulated $L^{3p}$ norm.
The final result of this section reduces it to the critical Serrin quantity in the native space $L^p$, yielding in particular the usual $L^p$ norm blow-up alternative at a finite maximal lifetime.
\begin{theorem}
[Serrin continuation in the native $L^p$ space] \label{thm:native-serrin-continuation} Let $p>3$ and assume the hypotheses of \cref{thm:3.6-intrinsic-blowup}.
Let $(u,(\tau_n),\tau_{\max})$ be the corresponding maximal local strong solution of \eqref{eq:1_Main}, with $u_0\in L^p(\Omega;E)$.
Retain the coefficient measurability assumption \eqref{eq:2-coefficient-measurability}.
We shall use \eqref{eq:2-taming-local} together with the consequences \eqref{eq:2-P1}--\eqref{eq:2-P2}; equivalently,,
\begin{equation}
\label{eq:3.7-assumptions}
\begin{aligned}
0\le g_N(r)&\le C_g r, &&r\ge0, \\
\|\Sigma(t,v)\|_{\gamma(\mathcal U;L^p)} &\le C_\Sigma \bigl(1+\|v\|_{3p/2}^{3/2}\bigr), &&v\in E\cap L^{3p}, \\
\int_Z\|\mathcal G(t,v,z)\|_p^q\,\mu(\dd z) &\le C_q\bigl(1+\|v\|_p^q\bigr), &&v\in E,\quad q\in\{2,p\}.
\end{aligned}
\end{equation}

Define
\begin{equation*}
\begin{aligned}
r_p&:=\frac{2p}{p-3}, & a_p(v)&:=1+\|v\|_p^{r_p}, \\
\mathcal S_p(t) &:=\int_0^t a_p(u(s))\,\dd s, & \mathcal V_p(v) &:=\left\|\nabla\bigl(|v|^{p/2}\bigr)\right\|_2^2, \qquad t<\tau_{\max}.
\end{aligned}
\end{equation*}
Let $\lambda_L$ denote the intrinsic exit times of \cref{thm:3.6-intrinsic-blowup}, and set
\begin{equation}
\label{eq:3.7-localizations}
\eta_K:=\inf\left\{ 0\le t<\tau_{\max}:\mathcal S_p(t)\ge K \right\}\wedge\tau_{\max}, \quad \rho_{L,K,T} :=T\wedge\lambda_L\wedge\eta_K, \quad L,K\ge1,\quad T>0,
\end{equation}
with $\inf\varnothing=\infty$.
There exist constants $c,C>0$, depending only on $p,\nu$ and the constants in \eqref{eq:3.7-assumptions}, such that, with $\rho=\rho_{L,K,T}$,
\begin{equation}
\label{eq:3.7-weighted-estimate}
\E\Bigg[ \sup_{0\le t\le\rho} e^{-c\mathcal S_p(t)} \bigl(1+\|u(t)\|_p^p\bigr) +\int_0^\rho e^{-c\mathcal S_p(s)}\mathcal V_p(u(s))\,\dd s+ \int_0^\rho e^{-c\mathcal S_p(s)} \bigl(1+\|u(s)\|_p^p\bigr)\, \dd\mathcal S_p(s) \Bigg] \le C\bigl(1+\E\|u_0\|_p^p\bigr).
\end{equation}
In particular,
\begin{equation}
\label{eq:3.7-unweighted-stopped}
\E\mathcal Q_u(\rho_{L,K,T}) \le C e^{cK}\bigl(1+\E\|u_0\|_p^p\bigr),
\end{equation}
with constants independent of $L,K,T$ and all auxiliary regularization parameters.

Consequently,
\begin{equation}
\label{eq:3.7-blowup-alternative}
\P\left( \tau_{\max}<\infty,\; \int_0^{\tau_{\max}} \|u(s)\|_p^{r_p}\,\dd s<\infty \right)=0.
\end{equation}
And,
\begin{equation}
\label{eq:3.7-norm-blowup}
\limsup_{t\uparrow\tau_{\max}}\|u(t)\|_p=\infty \quad\text{almost surely on } \{\tau_{\max}<\infty\}.
\end{equation}
\end{theorem}

\begin{proof}
Put
\begin{equation*}
\Phi(v):=\|v\|_p^p, \qquad \Psi(v):=1+\Phi(v), \qquad J_p(v):=|v|^{p-2}v.
\end{equation*}
All spatial norms below are taken over $\R^3$.
Constants may change between occurrences.

\emph{Step 1:} Let
\begin{equation*}
\ell:=\frac{p^2}{p-2}, \qquad q_f:=\frac{p^2}{2(p-1)}, \qquad q_h:=\frac{p^2}{3p-2}.
\end{equation*}
Since $p>3$, these exponents satisfy $p<\ell<3p$ and $1<q_f,q_h<\infty$.
Moreover,
\begin{equation}
\label{eq:3.7-identities}
\begin{aligned}
\frac1{q_f} &=\frac1p+\frac1\ell, & \frac1{q_f'} &=\frac12+\frac{p-2}{2\ell}, \\
\frac1{q_h} &=\frac2p+\frac1\ell, & (p-1)q_h' &=\ell.
\end{aligned}
\end{equation}

For a divergence-free $v\in L^p$ with $|v|^{p/2}\in H^1$, interpolation and the Sobolev inequality give
\begin{equation}
\label{eq:3.7-spatial-interpolation}
\|v\|_\ell^p \le \|v\|_p^{p-3}\|v\|_{3p}^3 \le C\|v\|_p^{p-3}\mathcal V_p(v)^{3/p}.
\end{equation}
Define $f(v):=-\mathcal P^{(1)}(v\otimes v), \ h(v):=-\mathcal P\bigl(g_N(|v|^2)v\bigr)$, where $\mathcal P^{(1)}$ acts on the first tensor index.
Thus $\nabla\cdot f(v)=-\mathcal P((v\cdot\nabla)v)$ in distributions.
Let
\begin{equation*}
S_m:=e^{m^{-2}\Delta}, \qquad v_m:=S_m v,
\end{equation*}
and, for smooth $z$, write
\begin{equation*}
\mathcal D_p(z) :=\int_{\R^3}|z|^{p-2}|\nabla z|^2\,\dd x.
\end{equation*}
For each $m$, $S_m$ is a deterministic scalar convolution operator, contractive on $L^q, 1 \le q \le \infty$, commuting with spatial derivatives and $\mathcal{P}$; in particular it preserves divergence-free fields, adaptedness and predictability, while $S_m \to I$ strongly on $L^q, 1 \le q < \infty$.
The operators $S_m$ are contractions on every $L^q$, commute with derivatives and the Leray projection, and converge strongly to the identity on $L^q$ for $1\le q<\infty$.

By \eqref{eq:3.7-identities}, the $L^q$-boundedness of the Leray projection and H\"older's inequality,
\begin{equation}
\label{eq:3.7-projected-drift}
\begin{aligned}
\left| \left\langle \nabla\cdot S_m f(v),J_p(v_m) \right\rangle \right| &\le C\|v\|_p\|v\|_\ell^{p/2} \mathcal D_p(v_m)^{1/2}, \\
\left| \left\langle S_mh(v),J_p(v_m)\right\rangle \right| &\le C\|v\|_p^2\|v\|_\ell^p.
\end{aligned}
\end{equation}
Indeed,
\begin{equation*}
\|\nabla J_p(v_m)\|_{q_f'} \le C\mathcal D_p(v_m)^{1/2} \|v_m\|_\ell^{(p-2)/2}, \quad \|h(v)\|_{q_h} \le C\|v\|_p^2\|v\|_\ell.
\end{equation*}
These estimates are applied directly to the projected nonlinearities; in particular, the pressure contributions are retained through the Leray projection, and no commutation of $\mathcal{P}$ with the nonlinear test function $J_p(v_m)$ is used.

The viscous term satisfies
\begin{equation*}
-\langle\Delta v_m,J_p(v_m)\rangle \ge\mathcal D_p(v_m), \qquad \mathcal V_p(v_m) \le\frac{p^2}{4}\mathcal D_p(v_m).
\end{equation*}
Combining these inequalities with \eqref{eq:3.7-spatial-interpolation} and \eqref{eq:3.7-projected-drift}, and applying Young's inequality, yields, for every $\varepsilon>0$,
\begin{equation}
\label{eq:3.7-drift-energy}
p\left\langle \nu\Delta v_m+\nabla\cdot S_mf(v)+S_mh(v), J_p(v_m) \right\rangle\le -\frac{2\nu}{p}\mathcal V_p(v_m) +\varepsilon\mathcal V_p(v) +C_\varepsilon a_p(v)\Psi(v).
\end{equation}
Here the remaining nonlinear expression is bounded by $C\|v\|_p^{p-1}\mathcal V_p(v)^{3/p},$ and the relevant Young exponent is determined by $\frac{p(p-1)}{p-3}=p+r_p.$ The full weighted gradient $\mathcal D_p$ has been used only for the smooth function $v_m$.

\emph{Step 2:} From \eqref{eq:3.7-assumptions}, $\|v\|_{3p/2} \le \|v\|_p^{1/2}\|v\|_{3p}^{1/2},$ and Sobolev's inequality, one obtains
\begin{equation}
\label{eq:3.7-wiener-energy}
\|v\|_p^{p-2} \|\Sigma(t,v)\|_{\gamma(\mathcal U;L^p)}^2 \le C\|v\|_p^{p-2} + C\|v\|_p^{p-1/2} \mathcal V_p(v)^{3/(2p)}\le \varepsilon\mathcal V_p(v) +C_\varepsilon a_p(v)\Psi(v).
\end{equation}
The last step uses
\begin{equation*}
\frac{p-\frac12}{1-\frac{3}{2p}} = p+\frac{2p}{2p-3} < p+r_p.
\end{equation*}
The jump assumptions imply
\begin{equation}
\label{eq:3.7-jump-energy}
\|v\|_p^{p-2} \int_Z\|\mathcal G(t,v,z)\|_p^2\,\mu(\dd z) + \int_Z\|\mathcal G(t,v,z)\|_p^p\,\mu(\dd z)\le C\Psi(v).
\end{equation}
Finally, the Taylor remainder $R_p(v,h):=\Phi(v+h)-\Phi(v)-D\Phi(v)[h]$ satisfies
\begin{equation}
\label{eq:3.7-taylor-remainder}
0\le R_p(v,h) \le C_p\left( \|v\|_p^{p-2}\|h\|_p^2+\|h\|_p^p \right).
\end{equation}

\emph{Step 3:} $\mathcal S_p$ is continuous and adapted on $[0,\tau_{\max})$, and is finite on every compact subinterval of this stochastic interval.
To verify the stopping-time assertion in \eqref{eq:3.7-localizations}, use the announcing sequence from \cref{lem:3.1-strict-localization} and define
\begin{equation*}
\mathcal S_{p,n}(t) :=\int_0^{t\wedge\tau_n}a_p(u(s))\,\dd s, \quad \eta_{K,n} :=\tau_n\wedge \inf\{t\ge0:\mathcal S_{p,n}(t)\ge K\}.
\end{equation*}
Each $\mathcal S_{p,n}$ is continuous and adapted.
Compatibility of the stopped representatives gives $\eta_{K,n}\uparrow\eta_K$.
Hence $\eta_K$ is a stopping time.Fix $L,K\ge1$ and $T>0$, and abbreviate $\rho=\rho_{L,K,T}$.
\cref{thm:3.6-intrinsic-blowup} gives $\rho<\tau_{\max}$ almost surely and
\begin{equation}
\label{eq:3.7-control}
\mathcal Q_u(\rho-)\le L, \qquad \mathcal S_p(\rho)\le K,
\end{equation}
where the preterminal intrinsic quantity is taken to be zero when $\rho=0$.
\cref{lem:3.3-energy-recovery} and \eqref{eq:3.7-control} imply
\begin{equation}
\label{eq:3.7-integrability}
\E\Bigg[ \sup_{0\le t\le\rho}\Psi(u(t)) + \int_0^\rho\mathcal V_p(u(s))\,\dd s + \int_0^\rho\Psi(u(s))\,\dd\mathcal S_p(s) \Bigg] <\infty.
\end{equation}
The last integral is bounded by $K\sup_{t\le\rho}\Psi(u(t))$.
No bound on the post-jump value $\mathcal Q_u(\rho)$ is asserted in \eqref{eq:3.7-control}.
Set
\begin{equation*}
U(t):=u(t\wedge\rho), \qquad U_m(t):=S_mU(t), \qquad w(t):=\exp\bigl(-c\mathcal S_p(t\wedge\rho)\bigr).
\end{equation*}
The processes $U,U_m$ are adapted and c\`adl\`ag.
Their left limits are predictable.
The process $w$ is continuous, adapted and predictable, and $e^{-cK}\le w(t)\le1.$

\emph{Step 4:} Apply $S_m$ to the stopped equation for $U$.
Lemma 3.2 and Gaussian smoothing imply that the resulting drift is integrable with values in $L^p$, and the regularized equation is an $L^p$ valued semimartingale.
On $(0,\rho]$, write
\begin{equation*}
B_m(s):=S_m\Sigma(s,u(s-)), \qquad H_m(s,z):=S_m\mathcal G(s,u(s-),z),
\end{equation*}
and extend these integrands by zero outside $(0,\rho]$.
They are predictable, with product predictability for $H_m$, since $\mathds 1_{(0,\rho]}$ is predictable.

The infinite-dimensional It\^o formula for $\Phi$, followed by integration by parts with $w$, has martingale terms
\begin{equation}
\begin{aligned}
M_W^m(t) &:=\int_0^{t\wedge\rho} w(s)D\Phi(U_m(s-)) \bigl[B_m(s)\,\dd W_s\bigr], \\
M_{J,1}^m(t) &:=\int_0^{t\wedge\rho}\int_Z w(s)D\Phi(U_m(s-))[H_m(s,z)] \,\widetilde N(\dd s,\dd z), \\
M_{J,2}^m(t) &:=\int_0^{t\wedge\rho}\int_Z w(s)R_p(U_m(s-),H_m(s,z)) \,\widetilde N(\dd s,\dd z).
\end{aligned}
\end{equation}
The last integral is defined as the difference between its nonnegative Poisson integral and its compensator.
Its integrability follows from \eqref{eq:3.7-jump-energy}--\eqref{eq:3.7-taylor-remainder}.
For an orthonormal basis $(e_k)$ of $\mathcal U$, the Wiener correction satisfies
\begin{equation}
\label{eq:3.7-trace-bound}
\begin{aligned}
\frac12\sum_{k\ge1} D^2\Phi(U_m(s-)) [B_m(s)e_k,B_m(s)e_k] &\le C_p\|U_m(s-)\|_p^{p-2} \|B_m(s)\|_{\gamma(\mathcal U;L^p)}^2 \\
&\le \varepsilon\mathcal V_p(u(s)) + C_\varepsilon a_p(u(s))\Psi(u(s))
\end{aligned}
\end{equation}
for almost every $(\omega,s)$ on $(0,\rho)$.
Here contraction of $S_m$ on $L^p$ and the ideal property of $\gamma$-radonifying operators have been used.
Introduce the integrable random variables
\begin{equation}
\label{eq:3.7-energy-functionals}
\begin{aligned}
\mathcal A &:=\sup_{0\le t\le\rho}w(t)\Psi(u(t)), \\
\mathcal D &:=\int_0^\rho w(s)\mathcal V_p(u(s))\,\dd s, \\
\mathcal I &:=\int_0^\rho w(s)\Psi(u(s))\,\dd\mathcal S_p(s).
\end{aligned}
\end{equation}
The real-valued BDG inequality, followed by Young's inequality and \eqref{eq:3.7-wiener-energy}, gives, for arbitrary $\delta,\varepsilon>0$,
\begin{equation}
\label{eq:3.7-wiener-bound}
\begin{aligned}
\E\sup_{t\le T}|M_W^m(t)| &\le C\E\Bigg[ \left( \sup_{t\le\rho}w(t)\|U_m(t)\|_p^p \right)^{1/2} \left( \int_0^\rho w(s)\|U_m(s-)\|_p^{p-2} \|B_m(s)\|_{\gamma(\mathcal U;L^p)}^2\,\dd s \right)^{1/2} \Bigg] \\
&\le \delta\E\mathcal A +\varepsilon\E\mathcal D +C_{\delta,\varepsilon}\E\mathcal I .
\end{aligned}
\end{equation}

For the linear jump martingale, Davis' inequality is applied to its optional quadratic variation.
Factoring out the weighted supremum before applying Young's inequality yields
\begin{equation}
\label{eq:3.7-jump-bounds}
\begin{aligned}
\E\sup_{t\le T}|M_{J,1}^m(t)| &\le \delta\E\mathcal A+ C_\delta\E \int_0^\rho\int_Z w(s)\|U_m(s-)\|_p^{p-2} \|H_m(s,z)\|_p^2\,N(\dd s,\dd z) \\
&\le \delta\E\mathcal A +C_\delta\E\mathcal I , \\
\E\sup_{t\le T}|M_{J,2}^m(t)| &\le 2\E \int_0^\rho\int_Z w(s)R_p(U_m(s-),H_m(s,z)) \,\mu(\dd z)\,\dd s \\
&\le C\E\mathcal I .
\end{aligned}
\end{equation}
In the first estimate, compensation is used only after Young's inequality.
Thus no jump moment beyond those in \eqref{eq:3.7-assumptions} is required.

The compensator in the It\^o formula is bounded by the same remainder estimate.
Combining \eqref{eq:3.7-drift-energy}, \eqref{eq:3.7-trace-bound}, \eqref{eq:3.7-wiener-bound} and \eqref{eq:3.7-jump-bounds} with the terminal and supremum forms of the It\^o identity gives
\begin{equation}
\label{eq:3.7-energy}
\begin{aligned}
\E\Bigg[ &\sup_{t\le\rho}w(t)\Psi(U_m(t)) + a_0\int_0^\rho w(s)\mathcal V_p(U_m(s))\,\dd s+ \frac c2\int_0^\rho w(s)\Psi(U_m(s))\,\dd\mathcal S_p(s) \Bigg] \\
&\le C\E\Psi(u_0) + \delta\E\mathcal A + \varepsilon\E\mathcal D + C_{\delta,\varepsilon}\E\mathcal I ,
\end{aligned}
\end{equation}
where $a_0>0$ depends only on $p,\nu$.
After rescaling $\delta,\varepsilon$, these parameters remain arbitrary.
All constants on the right are independent of $m,c,L,K,T$.
Every stochastic integral above is taken over $(0,t\wedge\rho]$.
Its state variable is the predictable left limit, and a possible jump at $\rho$ is retained.

\emph{Step 5:} For almost every $\omega$, the range of the c\`adl\`ag map $U(\cdot,\omega)$ on $[0,T]$ is relatively compact in $L^p$.
Uniform boundedness and strong convergence of $S_m$ therefore imply uniform convergence on this range.
By \eqref{eq:3.7-integrability} and dominated convergence,
\begin{equation}
\label{eq:3.7-path-convergence}
\E\sup_{0\le t\le T} \|U_m(t)-U(t)\|_p^p \longrightarrow0.
\end{equation}
In particular, convergence holds in probability in $\mathbb D([0,T];L^p)$ equipped with the $J_1$ topology.

The inequality
\begin{equation*}
\bigl\||a|^{p/2}-|b|^{p/2}\bigr\|_2 \le C_p\bigl( \|a\|_p^{(p-2)/2}+\|b\|_p^{(p-2)/2} \bigr)\|a-b\|_p
\end{equation*}
after squaring and applying H\"older inequality together with $L^p$ contractivity of $S_m$ and \eqref{eq:3.7-path-convergence} give,
\begin{equation}
\label{eq:3.7-power-convergence}
|U_m|^{p/2} \longrightarrow |U|^{p/2} \quad\text{strongly in } L^2(\Omega\times(0,T);L^2).
\end{equation}
Moreover,
\begin{equation*}
\E\sup_{t\le\rho}w(t)\Psi(U_m(t)) \longrightarrow\E\mathcal A, \quad \E\int_0^\rho w(s)\Psi(U_m(s))\,\dd\mathcal S_p(s) \longrightarrow\E\mathcal I .
\end{equation*}
For the second convergence, the total mass of $\dd\mathcal S_p$ on $[0,\rho]$ is at most $K$.
For fixed $c,L,K,T$, equations \eqref{eq:3.7-integrability} and \eqref{eq:3.7-energy}, together with $w\ge e^{-cK}$, bound
\begin{equation*}
\mathds 1_{\{s<\rho\}} \nabla\bigl(|U_m(s)|^{p/2}\bigr)
\end{equation*}
in $L^2(\Omega\times(0,T);L^2)$.
By \eqref{eq:3.7-power-convergence}, every weak limit is the spatial distributional gradient of
\begin{equation*}
\mathds 1_{\{s<\rho\}}|U(s)|^{p/2}.
\end{equation*}
Weak lower semicontinuity, applied after multiplication by the bounded function $w^{1/2}$, gives
\begin{equation*}
\E\mathcal D \le \liminf_{m\to\infty} \E\int_0^\rho w(s)\mathcal V_p(U_m(s))\,\dd s.
\end{equation*}
Thus the passage to the limit uses strong convergence of the stopped paths and weak convergence of their spatial energy gradients on the original stochastic basis.
Passing to the limit in \eqref{eq:3.7-energy} yields
\begin{equation*}
(1-\delta)\E\mathcal A + (a_0-\varepsilon)\E\mathcal D + \left(\frac c2-C_{\delta,\varepsilon}\right) \E\mathcal I \le C\E\Psi(u_0).
\end{equation*}
Choose $\delta<1/2$, $\varepsilon<a_0/2$, and then $c>2C_{\delta,\varepsilon}+2$.
This proves \eqref{eq:3.7-weighted-estimate}.
Finally, $\|v\|_{3p}^p = \bigl\||v|^{p/2}\bigr\|_6^2 \le C\mathcal V_p(v),$ and $w\ge e^{-cK}$ on $[0,\rho]$.
Equation \eqref{eq:3.7-unweighted-stopped} follows.

\emph{Step 6:} For positive integers $K,T$, define
\begin{equation*}
A_{K,T} := \left\{ \tau_{\max}\le T,\; \mathcal S_p(\tau_{\max}-)<K \right\},
\end{equation*}
where the terminal time is the increasing limit along the announcing sequence.
On $A_{K,T}$ one has $\eta_K=\tau_{\max}$.
The intrinsic blow-up alternative of \cref{thm:3.6-intrinsic-blowup} gives $\lambda_L<\tau_{\max}, \ \mathcal Q_u(\lambda_L)\ge L$ on this event, almost surely, for every integer $L\ge1$.
Consequently, $\rho_{L,K,T}=\lambda_L \quad\text{on }A_{K,T},$ and \eqref{eq:3.7-unweighted-stopped} implies
\begin{equation*}
L\,\P(A_{K,T}) \le \E\mathcal Q_u(\rho_{L,K,T}) \le C e^{cK}\bigl(1+\E\|u_0\|_p^p\bigr).
\end{equation*}
The event $A_{K,T}$ is used only in this pathwise comparison.
Letting $L\to\infty$ and then taking the countable union over $K,T$ proves
\begin{equation*}
\P\left( \tau_{\max}<\infty,\; \mathcal S_p(\tau_{\max}-)<\infty \right)=0.
\end{equation*}
Since $\tau_{\max}$ is finite on this event, this is equivalent to \eqref{eq:3.7-blowup-alternative}.
If $\tau_{\max}<\infty$ and $\limsup_{t\uparrow\tau_{\max}}\|u(t)\|_p<\infty$, then the c\`adl\`ag local paths are bounded on the entire interval $[0,\tau_{\max})$.
Hence their Serrin integral is finite, contradicting \eqref{eq:3.7-blowup-alternative}.
This proves \eqref{eq:3.7-norm-blowup}.
\end{proof}
\section{Energy estimates and positive-time regularization}
\label{Section 4} The continuation criteria of the preceding section must be supplemented by energy bounds for the prescribed maximal $L^p$ solution.
We establish $L^2$ and $H^1$ persistence, followed by positive-time regularization for data in $L^p\cap L^2$.

Throughout this section let $p>3,\ E=L^p_\sigma(\R^3;\R^3), \ \mathcal H=L^2_\sigma(\R^3;\R^3),$ and retain the maximal solution $(u,(\tau_n),\tau_{\max})$ of Section \ref{Sec:3}.
Assume
\begin{equation*}
u_0\in L^p(\Omega,\mathcal F_0;E) \cap L^2(\Omega,\mathcal F_0;\mathcal H),
\end{equation*}
and let $\Sigma, \ \mathcal G$ defined earlier.
We assume the $L^2$ growth condition \eqref{eq:2-raw-L2-growth}.
Whenever gradient estimates are invoked, we additionally assume \eqref{eq:2-H1}.
We assume \eqref{eq:2-taming-local} and \eqref{eq:2-taming-coercive} for the taming term.
Let $S_m={\mathrm P}_{\le m}$ be the smoothing operators \eqref{eq:2-smoothing}.

Choose $\chi\in C^\infty([0,\infty);[0,1])$ such that $\chi=1$ on $[0,1]$ and $\chi=0$ on $[2,\infty)$, and set $\chi_R(r)=\chi(r/R)$.
For $m\in\mathbb N$ and $R\ge1$, consider
\begin{equation}
\label{eq:4.1-regularized-equation}
\begin{aligned}
\dd u_m^R =&\Big[ \nu\Delta u_m^R -a_m^R S_m\mathcal P((v_m^R\cdot\nabla)v_m^R) -a_m^R S_m\mathcal P(g_N(|v_m^R|^2)v_m^R) \Big]\dd t+B_m^R(t)\,\dd W_t +\int_ZK_m^R(t,z)\,\widetilde N(\dd t,\dd z), \\
u_m^R(0)&=S_m u_0,
\end{aligned}
\end{equation}
where
\begin{equation*}
\begin{aligned}
v_m^R&:=S_m u_m^R, & a_m^R(t)&:=\chi_R(|u_m^R(t)|_p)^2, \\
B_m^R(t)&:=a_m^R(t-)S_m\Sigma(t,v_m^R(t-)), & K_m^R(t,z)&:=a_m^R(t-)S_m\mathcal G(t,v_m^R(t-),z).
\end{aligned}
\end{equation*}
When no confusion can arise, we suppress the superscript $R$.
Set $D_0=I$, $D_1=\nabla$, and define
\begin{equation}
\label{eq:4.1-energy-notation}
\begin{aligned}
\mathcal T_0(v) &:=\int_{\R^3}g_N(|v|^2)|v|^2\,\dd x, \\
\mathcal T_1(v) &:=\int_{\R^3}g_N(|v|^2)|\nabla v|^2\,\dd x +2\sum_{\ell=1}^3 \int_{\R^3} g_N'(|v|^2)(v\cdot\partial_\ell v)^2\,\dd x, \\
q_{m,j}(s) &:=|D_jB_m(s)|_{\gamma(\mathcal U;L^2)}^2 +\int_Z|D_jK_m(s,z)|_2^2\,\mu(\dd z), \qquad j=0,1.
\end{aligned}
\end{equation}
Finally, let $\mathcal M_{m,j}=M^W_{m,j}+M^{J,1}_{m,j}+M^{J,2}_{m,j}$, where
\begin{equation*}
\begin{aligned}
M^W_{m,j}(t) &:=2\int_0^t \left\langle (D_jB_m(s))^*D_j u_m(s-),\dd W_s \right\rangle_{\mathcal U}, \\
M^{J,1}_{m,j}(t) &:=2\int_0^t\int_Z (D_j u_m(s-),D_jK_m(s,z))_2\, \widetilde N(\dd s,\dd z), \\
M^{J,2}_{m,j}(t) &:=\int_0^t\int_Z |D_jK_m(s,z)|_2^2\, \widetilde N(\dd s,\dd z).
\end{aligned}
\end{equation*}
As usual, the last integral is interpreted as the corresponding nonnegative Poisson integral minus its compensator.

The next lemma justifies the regularized energy identities and identifies the approximations with the prescribed maximal solution through convergence in the native $L^p$ class.
This permits the transfer of energy estimates without assuming Sobolev regularity of the limit in advance.
\begin{lemma}
[Regularization and energy identities] \label{lem:4.1-regularization} Under the preceding assumptions, the following assertions hold.
\begin{enumerate}
\item[(i)] For every $m\in\mathbb N$ and $R\ge1$, \eqref{eq:4.1-regularized-equation} has a unique global adapted solution.
For every $T<\infty$,
\begin{equation*}
u_m^R\in L^p\bigl(\Omega;\mathbb D([0,T];E)\bigr) \cap L^2\bigl(\Omega;\mathbb D([0,T];H^1_\sigma)\bigr) \cap L^2\bigl(\Omega;L^2(0,T;H^2_\sigma)\bigr),
\end{equation*}
and
\begin{equation}
\label{eq:4.1-local-energy}
\E\int_0^T \left|\nabla\bigl(|u_m^R(s)|^{p/2}\bigr)\right|_2^2 \,\dd s<\infty .
\end{equation}
The bounds in this assertion may depend on $m,R,T$.

\item[(ii)] Indistinguishably on every finite interval,
\begin{equation}
\label{eq:4.1-L2-identity}
|u_m(t)|_2^2 +2\nu\int_0^t|\nabla u_m(s)|_2^2\,\dd s +2\int_0^t a_m(s)\mathcal T_0(v_m(s))\,\dd s=|S_m u_0|_2^2 +\int_0^tq_{m,0}(s)\,\dd s +\mathcal M_{m,0}(t),
\end{equation}
whereas, under \eqref{eq:2-H1},
\begin{equation}
\label{eq:4.1-H1-identity}
\begin{aligned}
|\nabla u_m(t)|_2^2 &+2\nu\int_0^t|\Delta u_m(s)|_2^2\,\dd s +2\int_0^t a_m(s)\mathcal T_1(v_m(s))\,\dd s \\
&=|\nabla S_m u_0|_2^2 +2\int_0^t a_m(s) ((v_m(s)\cdot\nabla)v_m(s),\Delta v_m(s))_2\,\dd s+\int_0^tq_{m,1}(s)\,\dd s +\mathcal M_{m,1}(t).
\end{aligned}
\end{equation}
Both identities remain valid after stopping at $t\wedge\eta$ for every stopping time $\eta\le T$, with the terminal jump included.

\item[(iii)] The estimates
\begin{equation}
\label{eq:4.1-coercivity}
\begin{aligned}
2a_m((v_m\cdot\nabla)v_m,\Delta v_m)_2 -2a_m\mathcal T_1(v_m) \le \nu|\Delta u_m|_2^2 +c_{N,\nu}|\nabla u_m|_2^2, \\
q_{m,0}(s) \le C_0(1+|u_m(s-)|_2^2), \quad q_{m,1}(s) \le C_1(1+|u_m(s-)|_{H^1}^2)
\end{aligned}
\end{equation}
hold, with the last estimate under \eqref{eq:2-H1} and, the constants of the estimates are independent of $m$ and $R$.
Moreover, for $j\in\{0,1\}$, $\delta>0$, and every stopping time $\eta\le T$,
\begin{equation}
\label{eq:4.1-martingale-estimates}
\begin{aligned}
\E\sup_{t\le\eta}|M^W_{m,j}(t)| +\E\sup_{t\le\eta}|M^{J,1}_{m,j}(t)| &\le \delta\E\sup_{t\le\eta}|D_ju_m(t)|_2^2 +C_\delta\E\int_0^\eta q_{m,j}(s)\,\dd s, \\
\E\sup_{t\le\eta}|M^{J,2}_{m,j}(t)| &\le 2\E\int_0^\eta\int_Z |D_jK_m(s,z)|_2^2\,\mu(\dd z)\,\dd s .
\end{aligned}
\end{equation}

\item[(iv)] Let $L\ge1$ and let $\rho\le T$ be a stopping time such that $\rho<\tau_{\max}, \ \mathcal Q_u(\rho-)\le L \ \P\text{-a.s.}$ For every fixed $R>(2L)^{1/p}$,
\begin{equation}
\label{eq:4.1-native-convergence}
\sup_{t\le\rho}|u_m^R(t)-u(t)|_p^p +\int_0^\rho |u_m^R(s)-u(s)|_{3p}^p\,\dd s \longrightarrow0
\end{equation}
in probability as $m\to\infty$.
Consequently, a deterministic subsequence converges almost surely in both terms of \eqref{eq:4.1-native-convergence}.
No $H^1$ regularity of the maximal solution is required for this consistency assertion.
\end{enumerate}
\end{lemma}

\begin{proof}
\emph{Step 1:} Let $\mathcal B_{m,R}(w)$ denote the nonlinear drift in \eqref{eq:4.1-regularized-equation}.
For fixed $m,R$, the smoothing bounds of $S_m$ and the support of $\chi_R$ give
\begin{equation}
\label{eq:4.1-fixed-parameter}
\begin{aligned}
\|\mathcal B_{m,R}(w)\|_{H^1} &\le C_{m,R}\|w\|_2, \\
\|\chi_R(\|w\|_p)^2 S_m\Sigma(t,S_mw) \|_{\gamma(\mathcal U;H^1)}^2+\int_Z \|\chi_R(\|w\|_p)^2 S_m\mathcal G(t,S_mw,z) \|_{H^1}^2\,\mu(\dd z) &\le C_m(1+\|w\|_2^2).
\end{aligned}
\end{equation}
For the first bound, use $\|S_mw\|_\infty+\|\nabla S_mw\|_\infty \le C_{m,p}\|w\|_p$, the cubic taming growth, and $S_m:L^2\to H^1$.
The second follows from \eqref{eq:2-raw-L2-growth}.

The truncated maps are globally Lipschitz in the forcing spaces of \cref{Thm:heat-exis}, with constants depending on $m,R$.
Start the Picard iteration at $V^{(0)}(t)=e^{\nu t\Delta}S_m u_0$ and freeze all nonlinear coefficients at $V^{(j)}$ to define $V^{(j+1)}$.
The linear $L^p$ estimate gives contraction in expected supremum norm on a sufficiently short deterministic interval.
Its limit is adapted and $E$-c\`adl\`ag, and the coefficient convergence in all forcing norms identifies the equation and gives \eqref{eq:4.1-local-energy}.
Iteration on deterministic intervals constructs the unique global $E$-valued solution.

We verify Hilbert-space membership by the Picard argument in Lemma 4.2 of \cite{PodderKumar2026}.
On a contraction interval $[0,h]$, set
\begin{equation*}
A_j(t)=\E\left[ \sup_{s\le t}\|V^{(j)}(s)\|_2^2 +\int_0^t\|\nabla V^{(j)}(s)\|_2^2\,\dd s\right].
\end{equation*}
Applying the standard $L^2$ energy estimate for the linear stochastic heat equation in the Gelfand triple $H^1_\sigma \hookrightarrow L^2_\sigma \hookrightarrow H^{-1}_\sigma$, together with \eqref{eq:4.1-fixed-parameter}, yields
\begin{equation*}
A_{j+1}(t)\le C\E\|u_0\|_2^2 +C_{m,R,T}\int_0^t(1+A_j(s))\,\dd s,
\end{equation*}
and hence $\sup_j A_j(h)<\infty$.
Choose a deterministic subsequence converging almost surely uniformly in $E$.
Pathwise, a further subsequence realizing the lower limit of the Hilbert energies is bounded in $L^\infty(0,h;\mathcal H)\cap L^2(0,h;H^1_\sigma)$.
Its weak-* and weak limits in these respective spaces coincide with the $E$-limit by distributional convergence.
At every time the same distributional convergence gives the lower bound for the $\mathcal H$ norm.
Weak lower semicontinuity and Fatou's lemma therefore yield finite expected Hilbert energy for the limit.
Its $E$-left limits belong to $\mathcal H$ as well.
Separability and testing against a countable dense family give $\mathcal H$-adaptedness of the state and $\mathcal H$-predictability of its left limits.
The $L^2$ energy estimate of its coefficients then supplies an $\mathcal H$-c\`adl\`ag version.
Repeating on deterministic intervals proves the required $L^2$ membership.

The resulting process has finite expected $L^2$ supremum on every finite interval.
Thus \eqref{eq:4.1-fixed-parameter} places all forcing terms in the square-integrable $H^1$ heat-equation classes.
Since $S_m u_0\in L^2(\Omega;H^1_\sigma)$, the Hilbert heat estimate at $H^1_\sigma$ level, gives
\begin{equation}
\label{eq:4.1-preliminary-H1}
\E\left[ \sup_{t\le T}\|u_m^R(t)\|_{H^1}^2 +\int_0^T\|u_m^R(s)\|_{H^2}^2\,\dd s \right] \le C_{m,R,T}(1+\E\|u_0\|_2^2).
\end{equation}
It remains only to identify the resulting $H^1_\sigma$-valued version with the previously constructed $E$-valued solution.
Both $E$ and $H^1_\sigma$ embed continuously into $\mathcal S'(\R^3;\R^3)$.
For every $\phi\in\mathcal S_\sigma$, pairing the two heat representations with $\phi$ gives the same scalar identity, since they have the same initial datum, drift, Wiener integral, and compensated Poisson integral.
Choosing a countable separating family in $\mathcal S_\sigma$ and restricting first to rational times yields a common event of probability one on which the two versions agree in $\mathcal S'$ at every rational time.
Since both versions are c\`adl\`ag, and their embeddings into $\mathcal S'$ are continuous, this equality extends to every $t\in[0,T]$.
Thus the $H^1_\sigma$-c\`adl\`ag version is a version of the original $E$-valued solution, and \eqref{eq:4.1-preliminary-H1} holds for $u_m^R$ itself.
This completes the proof of (i).

\emph{Step 2:} Suppress the superscript $R$, and put $$J_\varepsilon:=e^{\varepsilon\Delta}, \qquad u_m^\varepsilon:=J_\varepsilon u_m , \qquad \varepsilon>0 .$$ Since $J_\varepsilon$ is a scalar self-adjoint contraction commuting with spatial derivatives, $S_m$, and $\mathcal P$, application of $J_\varepsilon$ to \eqref{eq:4.1-regularized-equation} yields an $H^1_\sigma$-valued semimartingale with spatially smoothed drift and noise coefficients.
By \eqref{eq:4.1-preliminary-H1},
\begin{equation}
\label{eq:4.1-mollifier}
\begin{aligned}
\E\sup_{t\le T} \|u_m^\varepsilon(t)-u_m(t)\|_{H^1}^2 &\longrightarrow0, \\
u_m^\varepsilon &\longrightarrow u_m \quad\text{in } L^2(\Omega\times(0,T);H^2_\sigma).
\end{aligned}
\end{equation}
The first convergence in \eqref{eq:4.1-mollifier} follows pathwise from uniform convergence of $J_\varepsilon\to I$ on the relatively compact range of each $H^1$-valued c\`adl\`ag path, and then in expectation by dominated convergence, since $$ \sup_{t\le T} \|(J_\varepsilon-I)u_m(t)\|_{H^1}^2 \le 4\sup_{t\le T}\|u_m(t)\|_{H^1}^2.
$$ The $H^2$ convergence follows similarly from the approximate-identity property and \eqref{eq:4.1-preliminary-H1}.
In particular, the same uniform convergence holds for the left limits.
Set $D_0=I$, $D_1=\nabla$, and, for $j=0,1$, define $$X_{m,j}^\varepsilon:=D_j u_m^\varepsilon, \qquad B_{m,j}^\varepsilon:=J_\varepsilon D_jB_m, \qquad K_{m,j}^\varepsilon:=J_\varepsilon D_jK_m .$$ Commutation of $J_\varepsilon$ with $D_j$ has been used.
For fixed $m$, the square-integrability established in Step 1 and the bounds in \eqref{eq:4.1-fixed-parameter} imply
\begin{equation*}
\begin{aligned}
\E\int_0^T \|B_{m,j}^\varepsilon(s)-D_jB_m(s)\|_ {\gamma(\mathcal U;L^2)}^2\,\dd s &\longrightarrow0, \\
\E\int_0^T\int_Z \|K_{m,j}^\varepsilon(s,z)-D_jK_m(s,z)\|_2^2 \,\mu(\dd z)\,\dd s &\longrightarrow0 .
\end{aligned}
\end{equation*}

Applying the Hilbert-space It\^o formula with jumps to $\|X_{m,j}^\varepsilon(t)\|_2^2$ gives, indistinguishably in $t\in[0,T]$,
\begin{equation}
\label{eq:4.1-mollified-Ito}
\begin{aligned}
\|X_{m,j}^\varepsilon(t)\|_2^2 + & 2\nu\int_0^t \|\nabla X_{m,j}^\varepsilon(s)\|_2^2 \dd s = \|J_\varepsilon D_jS_m u_0\|_2^2 +2\int_0^t \bigl( X_{m,j}^\varepsilon(s), J_\varepsilon D_j\mathcal B_m(s) \bigr)_2 \dd s \\
&+ \int_0^t \|B_{m,j}^\varepsilon(s)\|_ {\gamma(\mathcal U;L^2)}^2\,\dd s + \int_0^t\int_Z \|K_{m,j}^\varepsilon(s,z)\|_2^2 \,\mu(\dd z)\,\dd s+ M_{m,j}^{W,\varepsilon}(t) +M_{m,j}^{J,1,\varepsilon}(t) +M_{m,j}^{J,2,\varepsilon}(t),
\end{aligned}
\end{equation}
where $\mathcal B_m$ denotes the nonlinear deterministic drift in the equation for $u_m$, and
\begin{equation*}
\begin{aligned}
M_{m,j}^{W,\varepsilon}(t) &= 2\int_0^t \left\langle (B_{m,j}^\varepsilon(s))^* X_{m,j}^\varepsilon(s-),\dd W_s \right\rangle_{\mathcal U}, \\
M_{m,j}^{J,1,\varepsilon}(t) &= 2\int_0^t\int_Z \bigl( X_{m,j}^\varepsilon(s-), K_{m,j}^\varepsilon(s,z) \bigr)_2\, \widetilde N(\dd s,\dd z), \\
M_{m,j}^{J,2,\varepsilon}(t) &= \int_0^t\int_Z \|K_{m,j}^\varepsilon(s,z)\|_2^2\, \widetilde N(\dd s,\dd z).
\end{aligned}
\end{equation*}
For $j=1$, the viscous term in \eqref{eq:4.1-mollified-Ito} equals $2\nu\int_0^t\|\Delta u_m^\varepsilon\|_2^2\,\dd s$.
One have, $\|\nabla^2v\|_2=\|\Delta v\|_2$ on $\R^3$.
We next pass to the limit $\varepsilon\downarrow0$.
The initial and deterministic terms converge in $L^1$ by \eqref{eq:4.1-fixed-parameter} and \eqref{eq:4.1-mollifier}.
The quadratic-variation terms also converge.
One gets,
\begin{equation*}
\begin{aligned}
&\E\int_0^T \left| \|B_{m,j}^\varepsilon(s)\|_{\gamma(\mathcal U;L^2)}^2 -\|D_jB_m(s)\|_{\gamma(\mathcal U;L^2)}^2 \right|\dd s \\
&\quad\le \left( \E\int_0^T \bigl( \|B_{m,j}^\varepsilon(s)\|_{\gamma} +\|D_jB_m(s)\|_{\gamma} \bigr)^2\dd s \right)^{1/2} \left( \E\int_0^T \|B_{m,j}^\varepsilon(s)-D_jB_m(s)\|_{\gamma}^2 \,\dd s \right)^{1/2} \longrightarrow0 ,
\end{aligned}
\end{equation*}
and the Poisson compensator is treated identically.
For the Wiener martingale, the BDG inequality and the decomposition
\begin{equation*}
(B_{m,j}^\varepsilon)^*X_{m,j}^\varepsilon -(D_jB_m)^*D_ju_m= (B_{m,j}^\varepsilon)^* (X_{m,j}^\varepsilon-D_ju_m) + (B_{m,j}^\varepsilon-D_jB_m)^*D_ju_m
\end{equation*}
give
\begin{equation*}
\begin{aligned}
&\E\sup_{t\le T} |M_{m,j}^{W,\varepsilon}(t)-M_{m,j}^{W}(t)| \\
&\quad\le C\left( \E\sup_{t\le T} \|X_{m,j}^\varepsilon(t)-D_ju_m(t)\|_2^2 \right)^{1/2} \left( \E\int_0^T \|D_jB_m(s)\|_{\gamma(\mathcal U;L^2)}^2\,\dd s \right)^{1/2} \\
&\qquad+ C\left( \E\sup_{t\le T}\|D_ju_m(t)\|_2^2 \right)^{1/2} \left( \E\int_0^T \|B_{m,j}^\varepsilon(s)-D_jB_m(s)\|_{\gamma(\mathcal U;L^2)}^2 \,\dd s \right)^{1/2} \\
&\quad\longrightarrow0 .
\end{aligned}
\end{equation*}
Here the $L^2$ contractivity of $J_\varepsilon$ was used to replace $B_{m,j}^\varepsilon$ by $D_jB_m$ in the uniform bound.
For the linear jump martingale, Davis' inequality applied to the optional quadratic variation gives
\begin{equation*}
\begin{aligned}
&\E\sup_{t\le T} |M_{m,j}^{J,1,\varepsilon}(t)-M_{m,j}^{J,1}(t)| \\
&\quad\le C\left( \E\sup_{t\le T} \|X_{m,j}^\varepsilon(t)-D_ju_m(t)\|_2^2 \right)^{1/2} \left( \E\int_0^T\int_Z \|D_jK_m(s,z)\|_2^2 \,\mu(\dd z)\,\dd s \right)^{1/2} \\
&\qquad+ C\left( \E\sup_{t\le T}\|D_ju_m(t)\|_2^2 \right)^{1/2} \left( \E\int_0^T\int_Z \|K_{m,j}^\varepsilon(s,z)-D_jK_m(s,z)\|_2^2 \,\mu(\dd z)\,\dd s \right)^{1/2} \\
&\quad\longrightarrow0 .
\end{aligned}
\end{equation*}
Indeed, the expectation of the Poisson integral arising from the optional quadratic variation is evaluated by the compensator identity.
The quadratic jump term requires separate treatment.
Since $\|K_{m,j}^\varepsilon\|_2^2-\|D_jK_m\|_2^2$ has no fixed sign, write the compensated integral as the difference of its Poisson integral and its compensator.
Pathwise,
\begin{equation*}
\begin{aligned}
&\sup_{t\le T} \left| \int_0^t\int_Z \left( \|K_{m,j}^\varepsilon(s,z)\|_2^2 -\|D_jK_m(s,z)\|_2^2 \right)\widetilde N(\dd s,\dd z) \right| \\
&\quad\le \int_0^T\int_Z \left| \|K_{m,j}^\varepsilon(s,z)\|_2^2 -\|D_jK_m(s,z)\|_2^2 \right|N(\dd s,\dd z)+ \int_0^T\int_Z \left| \|K_{m,j}^\varepsilon(s,z)\|_2^2 -\|D_jK_m(s,z)\|_2^2 \right|\mu(\dd z)\,\dd s .
\end{aligned}
\end{equation*}
Taking expectations and using the compensator identity,
\begin{equation*}
\begin{aligned}
&\E\sup_{t\le T} |M_{m,j}^{J,2,\varepsilon}(t)-M_{m,j}^{J,2}(t)|\le 2\E\int_0^T\int_Z \left| \|K_{m,j}^\varepsilon(s,z)\|_2^2 -\|D_jK_m(s,z)\|_2^2 \right|\mu(\dd z)\,\dd s \\
&\quad\le 2 \left[ \E\int_0^T\int_Z \bigl( \|K_{m,j}^\varepsilon(s,z)\|_2 +\|D_jK_m(s,z)\|_2 \bigr)^2 \mu(\dd z)\,\dd s \right]^{1/2} \\
&\qquad\times \left[ \E\int_0^T\int_Z \|K_{m,j}^\varepsilon(s,z)-D_jK_m(s,z)\|_2^2 \,\mu(\dd z)\,\dd s \right]^{1/2} \longrightarrow0 .
\end{aligned}
\end{equation*}
We may therefore let $\varepsilon\downarrow0$ in \eqref{eq:4.1-mollified-Ito}.
Self-adjointness of $S_m$, commutation with derivatives and $\mathcal P$, and $\mathcal Pu_m=u_m$ yield
\begin{equation*}
\begin{aligned}
(u_m,S_m\mathcal P((v_m\cdot\nabla)v_m))_2 &=(v_m,(v_m\cdot\nabla)v_m)_2=0, \\
(u_m,S_m\mathcal P(g_N(|v_m|^2)v_m))_2 &=\mathcal T_0(v_m), \\
(-\Delta u_m,S_m\mathcal P((v_m\cdot\nabla)v_m))_2 &=-(\Delta v_m,(v_m\cdot\nabla)v_m)_2, \\
(-\Delta u_m,S_m\mathcal P(g_N(|v_m|^2)v_m))_2 &=\mathcal T_1(v_m).
\end{aligned}
\end{equation*}
The first identity is the usual incompressible cancellation, $(v_m,(v_m\cdot\nabla)v_m)_2 = \frac12\int_{\R^3} v_m\cdot\nabla|v_m|^2\,\dd x=0.$ The last identity follows by integration by parts:
\begin{equation*}
(-\Delta v, g_N(|v|^2)v)_2= \int_{\R^3} g_N(|v|^2)|\nabla v|^2\,\dd x + 2\sum_{\ell=1}^3 \int_{\R^3} g_N'(|v|^2) (v\cdot\partial_\ell v)^2\,\dd x =\mathcal T_1(v).
\end{equation*}
These integrations by parts are justified by \eqref{eq:4.1-preliminary-H1} and the fixed-$m$ spatial smoothing.
Consequently, \eqref{eq:4.1-L2-identity} and \eqref{eq:4.1-H1-identity} follow.
We next establish the coercive estimate.
For $v=v_m$, Cauchy--Schwarz and Young's inequality give
\begin{equation*}
2((v\cdot\nabla)v,\Delta v)_2 \le 2\int_{\R^3} |v||\nabla v||\Delta v|\,\dd x\le \nu\|\Delta v\|_2^2 +\frac1\nu \int_{\R^3}|v|^2|\nabla v|^2\,\dd x .
\end{equation*}
Since $g_N'\ge0$, $\mathcal T_1(v) \ge \int_{\R^3} g_N(|v|^2)|\nabla v|^2\,\dd x .$ Hence, by the definition of $c_{N,\nu}$,
\begin{equation*}
\begin{aligned}
2((v\cdot\nabla)v,\Delta v)_2 -2\mathcal T_1(v) &\le \nu\|\Delta v\|_2^2+ \int_{\R^3} \left( \frac{|v|^2}{\nu} -2g_N(|v|^2) \right)|\nabla v|^2\,\dd x \\
&\le \nu\|\Delta v\|_2^2 +c_{N,\nu}\|\nabla v\|_2^2 .
\end{aligned}
\end{equation*}
Since $v_m=S_mu_m$, $S_m$ is an $L^2$ contraction commuting with derivatives, and $0\le a_m\le1$, \eqref{eq:4.1-coercivity} follows.

It remains to record the martingale estimates used below.
Let $\eta\le T$ be a stopping time.
For the Wiener term, BDG and Young's inequality yield, for every $\delta>0$,
\begin{equation}
\label{eq:4.1-Wiener-martingale}
\begin{aligned}
\E\sup_{t\le\eta}|M_{m,j}^W(t)| &\le C\E \left( \int_0^\eta \|D_ju_m(s-)\|_2^2 \|D_jB_m(s)\|_{\gamma(\mathcal U;L^2)}^2 \,\dd s \right)^{1/2} \\
&\le \delta\, \E\sup_{t\le\eta}\|D_ju_m(t)\|_2^2 + C_\delta \E\int_0^\eta \|D_jB_m(s)\|_{\gamma(\mathcal U;L^2)}^2\,\dd s \\
&\le \delta\, \E\sup_{t\le\eta}\|D_ju_m(t)\|_2^2 + C_\delta \E\int_0^\eta q_{m,j}(s)\,\dd s .
\end{aligned}
\end{equation}

For the linear jump martingale, its optional quadratic variation is
\begin{equation}
\label{eq:4.1-linear-jump-QV}
[M_{m,j}^{J,1}]_t = 4\int_0^t\int_Z \left| (D_ju_m(s-),D_jK_m(s,z))_2 \right|^2 N(\dd s,\dd z).
\end{equation}
Davis' inequality therefore gives
\begin{equation}
\label{eq:4.1-linear-jump}
\begin{aligned}
\E\sup_{t\le\eta}|M_{m,j}^{J,1}(t)| &\le C\E \left[ \sup_{t\le\eta}\|D_ju_m(t)\|_2^2 \int_0^\eta\int_Z \|D_jK_m(s,z)\|_2^2\,N(\dd s,\dd z) \right]^{1/2} \\
&\le \delta\, \E\sup_{t\le\eta}\|D_ju_m(t)\|_2^2+ C_\delta \E\int_0^\eta\int_Z \|D_jK_m(s,z)\|_2^2\,N(\dd s,\dd z) \\
&= \delta\, \E\sup_{t\le\eta}\|D_ju_m(t)\|_2^2+ C_\delta \E\int_0^\eta\int_Z \|D_jK_m(s,z)\|_2^2\,\mu(\dd z)\,\dd s \\
&\le \delta\, \E\sup_{t\le\eta}\|D_ju_m(t)\|_2^2 + C_\delta \E\int_0^\eta q_{m,j}(s)\,\dd s .
\end{aligned}
\end{equation}
The compensator identity is used only after Young's inequality, so no higher jump moment is introduced.

Finally, put $$H_{m,j}(s,z):=\|D_jK_m(s,z)\|_2^2\ge0 .$$ Since $$ M_{m,j}^{J,2}(t) = \int_0^t\int_Z H_{m,j}(s,z)\,N(\dd s,\dd z) - \int_0^t\int_Z H_{m,j}(s,z)\,\mu(\dd z)\,\dd s , $$ one has pathwise
\begin{equation*}
\sup_{t\le\eta}|M_{m,j}^{J,2}(t)| \le \int_0^\eta\int_ZH_{m,j}(s,z)\,N(\dd s,\dd z) + \int_0^\eta\int_ZH_{m,j}(s,z)\,\mu(\dd z)\,\dd s .
\end{equation*}
Taking expectations and applying the compensation formula,
\begin{equation}
\label{eq:4.1-quadratic-jump}
\E\sup_{t\le\eta}|M_{m,j}^{J,2}(t)| \le 2\E\int_0^\eta\int_Z \|D_jK_m(s,z)\|_2^2\,\mu(\dd z)\,\dd s .
\end{equation}
Equations \eqref{eq:4.1-Wiener-martingale}-- \eqref{eq:4.1-quadratic-jump} give \eqref{eq:4.1-martingale-estimates}.
Stopping is implemented by multiplication of the stochastic integrands by the predictable indicator $\mathds1_{(0,\eta]}$, hence the possible jump at $\eta$ is retained.
All state-dependent stochastic integrands are evaluated at the predictable left limit.
This proves (ii)--(iii).

\emph{Step 3: } Fix $\rho,L,R$ as in (iv), and define the uncut coefficients
\begin{equation*}
\begin{aligned}
f_m(w)&=-S_m\mathcal P^{(1)}(S_mw\otimes S_mw), & h_m(w)&=-S_m\mathcal P(g_N(|S_mw|^2)S_mw), \\
\Sigma_m(s,w)&=S_m\Sigma(s,S_mw),& \mathcal G_m(s,w,z)&=S_m\mathcal G(s,S_mw,z).
\end{aligned}
\end{equation*}
Let $f(w)=-\mathcal P^{(1)}(w\otimes w)$ and $h(w)=-\mathcal P(g_N(|w|^2)w)$.
Choose $q_f,q_h$ as in \cref{lem:3.2-intrinsic-coefficients}.
For a predictable coefficient $J$, put
\begin{equation*}
\mathcal J_\rho(J) =\E\int_0^\rho\int_Z\|J(s,z)\|_p^p\,\mu(\dd z)\,\dd s +\E\left( \int_0^\rho\int_Z\|J(s,z)\|_p^2\,\mu(\dd z)\,\dd s \right)^{p/2}.
\end{equation*}
All stopped coefficients are extended by zero using $\mathds1_{(0,\rho]}$.
Set
\begin{equation}
\label{eq:4.1-consistency-error}
\begin{aligned}
\varepsilon_m :=&\E\|S_m u_0-u_0\|_p^p \\
&+\E\int_0^\rho\left[ \|f_m(u)-f(u)\|_{q_f}^p +\|h_m(u)-h(u)\|_{q_h}^p +\|\Sigma_m(s,u_-)-\Sigma(s,u_-)\|_{\gamma(\mathcal U;E)}^p \right]\dd s \\
&+\mathcal J_\rho\bigl( \mathcal G_m(\cdot,u_-,\cdot) -\mathcal G(\cdot,u_-,\cdot)\bigr).
\end{aligned}
\end{equation}
Then
\begin{equation}
\label{eq:4.1-vanishing}
\varepsilon_m\longrightarrow0.
\end{equation}
Indeed, $S_m\to I$ strongly in every Lebesgue space used in Lemma \ref{lem:3.2-intrinsic-coefficients}.
The taming and noise difference bounds give pointwise convergence of the coefficients for $\dd\P\otimes\dd s$-almost every point in $(0,\rho)$.
Lemma \ref{lem:3.2-intrinsic-coefficients}, contraction of $S_m$, and $\mathcal Q_u(\rho-)\le L$ provide integrable dominating functions for the drift and Wiener terms.
The $L^2(Z)$ and $L^p(Z)$ jump growth bounds provide dominating bounds depending only on $L,T$ for both Poisson terms.
Dominated convergence proves \eqref{eq:4.1-vanishing}.
Here $u_-=u$ almost everywhere in time and no $L^{3p}$-valued left limit is asserted.

\emph{Step 4:} Write $U_m=u_m^R$ and
\begin{equation*}
\begin{aligned}
\mathcal Q_{U_m}(t) &=\sup_{s\le t}\|U_m(s)\|_p^p +\int_0^t\|U_m(s)\|_{3p}^p\,\dd s, \\
\kappa_m&=\inf\{t\ge0:\mathcal Q_{U_m}(t)\ge2L\}, \qquad \theta_m=\rho\wedge\kappa_m.
\end{aligned}
\end{equation*}
These are stopping times, with $\inf\varnothing=\infty$.
Before $\theta_m$, and for the predictable states through $\theta_m$, the $L^p$ norms of $U_m$ and $u$ are bounded by $K=(2L)^{1/p}$ whenever the integration interval is nonempty.
Since $R>K$, the scalar cutoff is identically one there.

Let $d_m=U_m-u$.
For $x,y\in E\cap L^{3p}$ with $\|x\|_p\vee\|y\|_p\le K$, interpolation gives
\begin{equation*}
\begin{aligned}
\|f_m(x)-f_m(y)\|_{q_f} &\le C_K\|x-y\|_p^{\eta_f} \|x-y\|_{3p}^{1-\eta_f}, \\
\|h_m(x)-h_m(y)\|_{q_h} &\le C_K\|x-y\|_p^{\eta_h} \|x-y\|_{3p}^{1-\eta_h}, \\
\|\Sigma_m(s,x)-\Sigma_m(s,y)\|_{\gamma(\mathcal U;E)} &\le C_K\|x-y\|_p^{1/4} \|x-y\|_{3p}^{3/4}, \\
\left(\int_Z \|\mathcal G_m(s,x,z)-\mathcal G_m(s,y,z)\|_p^r \,\mu(\dd z)\right)^{1/r} &\le C_K\|x-y\|_p^{\eta_G} \|x-y\|_{3p}^{1-\eta_G}, \quad r=2,p,
\end{aligned}
\end{equation*}
where
\begin{equation*}
\eta_f=\frac{3p/q_f-4}{2},\qquad \eta_h=\frac{3p/q_h-7}{2},\qquad \eta_G=1-\frac{3\alpha}{2}.
\end{equation*}
Here $\alpha\in[0,2/3)$ is the jump Lipschitz exponent and all three exponents are positive.
For the first two bounds, apply H\"older with
\begin{equation*}
\frac1{\ell_f}=\frac1{q_f}-\frac1p,\qquad \frac1{\ell_h}=\frac1{q_h}-\frac2p, \qquad p<\ell_f,\ell_h<3p,
\end{equation*}
and interpolate the difference between $L^p$ and $L^{3p}$.
The last two bounds use the weighted noise Lipschitz assumptions, with difference exponents $2p$ and $p/(1-\alpha)$, respectively.
All constants are independent of $m$.

Subtract the equations through $\theta_m$, and split each coefficient into its state difference and the consistency error in \eqref{eq:4.1-consistency-error}.
Let $z_m$ be the linear heat solution on $[0,T]$ with initial datum $S_m u_0-u_0$ and these coefficient differences multiplied by $\mathds1_{(0,\theta_m]}$.
Lemma \ref{lem:3.2-intrinsic-coefficients} and the pre-exit bounds place these data in the forcing classes of the linear theorem.
Uniqueness for that theorem gives $z_m=d_m$ on $[0,\theta_m]$; after $\theta_m$, $z_m$ is the heat continuation used in Lemma \ref{lem:3.3-energy-recovery}.

For $I=[t_0,t_1]\subset[0,T]$, $t_1-t_0\le h\le1$, put
\begin{equation*}
X_I=\sup_{t_0\le t\le t_1}\|z_m(t)\|_p^p,\ Y_I=\int_{t_0}^{t_1}\|z_m(s)\|_{3p}^p\,\dd s.
\end{equation*}
For every interpolation exponent $0<\eta<1$,
\begin{equation*}
\begin{aligned}
\E\int_{t_0\wedge\theta_m}^{t_1\wedge\theta_m} \|d_m(s)\|_p^{p\eta} \|d_m(s)\|_{3p}^{p(1-\eta)}\,\dd s &\le h^\eta\E(X_I^\eta Y_I^{1-\eta}) \\
&\le \delta\E Y_I +C_{\delta,\eta}h\,\E X_I.
\end{aligned}
\end{equation*}
For $\eta=1$ the corresponding bound is $h\E X_I$.
For $r=2,p$, set $$ A_{m,r}(s) := \left( \int_Z \|\mathcal{G}_m(s, U_m, z) - \mathcal{G}_m(s, u, z)\|_p^r \,\mu(\dd z) \right)^{1/r}.
$$

The state-difference estimate gives $ A_{m,r}(s) \le C_K \|d_m(s)\|_p^{\eta_G} \|d_m(s)\|_{3p}^{1-\eta_G}.
$ Therefore $$ \E \int_{I \cap (0, \theta_m]} A_{m,p}(s)^p \,\dd s \le \delta \E Y_I + C_{\delta,K} h \E X_I.
$$

For the quadratic Poisson mode, $$
\begin{aligned}
\E \left( \int_{I \cap (0, \theta_m]} A_{m,2}(s)^2 \,\dd s \right)^{p/2}& \le h^{p/2-1} \E \int_{I \cap (0, \theta_m]} A_{m,2}(s)^p \,\dd s \\
&\le C_K h^{p/2-1} (\delta \E Y_I + C_{\delta,K} h \E X_I).
\end{aligned}
$$ Since $p > 3 > 2$, $ h^{p/2-1} \le 1, \ 0 < h \le 1.
$ Apply the linear heat estimate of \cref{Thm:heat-exis} to $z_m$ on $I$ and use $\|z_m\|_{3p}^p \le C_p\|\nabla(|z_m|^{p/2})\|_2^2.$ The Leray projection remains inside the forcing terms throughout this application of the linear estimate.
One gets, $$ \E X_I + c_1 \E Y_I \le C \E \|z_m(t_0)\|_p^p + C \varepsilon_{m,I} + C_K \delta\, \E Y_I + C_{\delta,K} h\, \E X_I.
$$ Now the order of choices matters.
First choose $\delta = \delta(K) > 0$ so that $ C_K \delta \le \frac{c_1}{2}.
$ Then choose $h = h(K) > 0$ so that $ C_{\delta,K} h \le \frac{1}{2}.
$ Hence $$ \E X_I + \E Y_I \le C_K \E \|z_m(t_0)\|_p^p + C_K \varepsilon_{m,I}.
$$ Since $K = (2L)^{1/p}$, one may write $h = h(L)$.
Now let, the partition

$$0 = t_0 < t_1 < \cdots < t_N = T, \qquad t_{j+1} - t_j \le h(L),$$

where $N \le 1 + \frac{T}{h(L)}.$ Let $ X_j = X_{[t_j,t_{j+1}]}, \ Y_j = Y_{[t_j,t_{j+1}]}.$ For the first interval, $z_m(0) = S_m u_0 - u_0,$ so $$\E X_0 + \E Y_0 \le C_L \E \|S_m u_0 - u_0\|_p^p + C_L \varepsilon_{m,0}.$$

For $j \ge 1$, $\|z_m(t_j)\|_p^p \le X_{j-1},$ and therefore $$\E X_j + \E Y_j \le C_L \E X_{j-1} + C_L \varepsilon_{m,j}.
$$

Induction over the finite number $N = N(L,T)$ of intervals yields

$$\sum_{j=0}^{N-1} (\E X_j + \E Y_j) \le C_{L,T} \left[ \E \|S_m u_0 - u_0\|_p^p + \sum_j \varepsilon_{m,j} \right].$$

Hence
\begin{equation}
\label{eq:4.1-stopped-comparison}
\E\left[ \sup_{t\le\theta_m}\|d_m(t)\|_p^p +\int_0^{\theta_m}\|d_m(s)\|_{3p}^p\,\dd s \right] \le C_{L,T}\varepsilon_m\longrightarrow0.
\end{equation}
The terminal value is included in this estimate.
On $\{\kappa_m<\rho\}$, right continuity gives $\mathcal Q_{U_m}(\kappa_m)\ge2L$, whereas $\mathcal Q_u(\kappa_m)\le L$.
Minkowski's inequality for the sum of the supremum and time-integral components yields
\begin{equation*}
\sup_{t\le\theta_m}\|d_m(t)\|_p^p +\int_0^{\theta_m}\|d_m(s)\|_{3p}^p\,\dd s \ge (2^{1/p}-1)^p L \quad\text{on }\{\kappa_m<\rho\}.
\end{equation*}
Thus \eqref{eq:4.1-stopped-comparison} implies $\P(\kappa_m<\rho)\to0$.
On $\{\kappa_m\ge\rho\}$, the comparison interval is $[0,\rho]$.
Markov's inequality now proves \eqref{eq:4.1-native-convergence}.
Choosing a deterministic subsequence with summable error probabilities gives the final assertion.
\end{proof}
We next show that finite initial kinetic energy is preserved by the maximal $L^p$ solution.
The resulting estimate controls the viscous and taming dissipation before the maximal lifetime and yields a strongly $L^2$-c\`adl\`ag representative of the same solution.
\begin{proposition}
[$L^2$ persistence of the maximal solution] \label{prop:4.2-L2-persistence} Retain the hypotheses of Lemma \ref{lem:4.1-regularization}, except for the gradient-noise assumption \eqref{eq:2-H1}, and let $(u,(\tau_n),\tau_{\max})$ be the maximal solution of \cref{Sec:3}.
In particular, assume
\begin{equation*}
u_0\in L^p(\Omega,\mathcal F_0;L^p_\sigma(\R^3)) \cap L^2(\Omega,\mathcal F_0;L^2_\sigma(\R^3)).
\end{equation*}
Set $\mathcal H=L^2_\sigma(\R^3)$ and retain $\mathcal T_0$ from \eqref{eq:4.1-energy-notation}.
Then $u$ has an adapted representative with strongly $\mathcal H$-c\`adl\`ag paths on $[0,\tau_{\max})$.
For every $n\ge1$ and $T<\infty$,
\begin{equation}
\label{eq:4.2-stopped-regularity}
u(\cdot\wedge\tau_n) \in L^2\bigl(\Omega;\mathbb D([0,T];\mathcal H)\bigr),\ \E\int_0^{T\wedge\tau_n} \|u(s)\|_{H^1}^2\,\dd s<\infty,
\end{equation}
where the path-space notation denotes a finite expected square of the supremum norm.
Moreover,
\begin{equation}
\label{eq:4.2-maximal-energy}
\E\Bigg[ \sup_{\substack{0\le t\le T \\
t<\tau_{\max}}}\|u(t)\|_2^2 +\nu\int_0^{T\wedge\tau_{\max}} \|\nabla u(s)\|_2^2\,\dd s +\int_0^{T\wedge\tau_{\max}} \mathcal T_0(u(s))\,\dd s \Bigg] \le C_T\left(1+\E\|u_0\|_2^2\right).
\end{equation}
The constant $C_T$ depends only on $T$, $\nu$, and $C_0$ in \eqref{eq:2-raw-L2-growth}.
It is independent of the intrinsic exit level, the announcing sequence, and all scalar and spatial regularization parameters.
Integrals ending at a finite $\tau_{\max}$ are increasing limits from below and no value at $\tau_{\max}$ is asserted.
\end{proposition}

\begin{proof}
\emph{Step 1:} Let $U_m=u_m^R$ be the solution of \eqref{eq:4.1-regularized-equation}, and retain the notation $v_m=S_mU_m$, $a_m$, $q_{m,0}$, and $M^W_{m,0}$, $M^{J,1}_{m,0}$, $M^{J,2}_{m,0}$ from Lemma \ref{lem:4.1-regularization}.
For $t\le T$, its martingale estimates give
\begin{equation}
\label{eq:4.2-martingale}
\begin{aligned}
\E\sup_{r\le t}|M^W_{m,0}(r)| +\E\sup_{r\le t}|M^{J,1}_{m,0}(r)| &\le \delta\E\sup_{r\le t}\|U_m(r)\|_2^2 +C_\delta\E\int_0^t q_{m,0}(s)\,\dd s, \\
\E\sup_{r\le t}|M^{J,2}_{m,0}(r)| \le 2\E\int_0^t\int_Z \|K_m(s,z)\|_2^2\,\mu(\dd z)\,\dd s, &\ \E\int_0^t q_{m,0}(s)\,\dd s \le C_0\int_0^t \left(1+\E\sup_{r\le s}\|U_m(r)\|_2^2\right)\dd s.
\end{aligned}
\end{equation}
The last jump estimate uses the nonnegative integrand $\|K_m\|_2^2$ and the compensator identity.
Apply \eqref{eq:4.1-L2-identity}, retain the nonnegative taming integral, and choose $\delta$ sufficiently small in \eqref{eq:4.2-martingale}.
Gronwall's inequality and $L^2$ contractivity of $S_m$ yield
\begin{equation}
\label{eq:4.2-regularized-energy}
\sup_{m\ge1,\,R\ge1}\E\left[ \sup_{t\le T}\|U_m(t)\|_2^2 +\nu\int_0^T\|\nabla U_m(s)\|_2^2\,\dd s +\int_0^T a_m(s)\mathcal T_0(v_m(s))\,\dd s \right] \le C_T\left(1+\E\|u_0\|_2^2\right).
\end{equation}
The integrability needed for this calculation is supplied by Lemma \ref{lem:4.1-regularization} at fixed $m,R$.

\emph{Step 2.} Fix $n\ge1$, $L\ge1$, and $T<\infty$, and set $$\rho=\rho_{n,L,T}:=T\wedge\tau_n\wedge\lambda_L,$$ where $\lambda_L$ is defined in \eqref{eq:3.6-intrinsic-exits}.
Lemma \ref{lem:3.1-strict-localization} and Theorem \ref{thm:3.6-intrinsic-blowup} give $ \rho<\tau_{\max},\ \mathcal Q_u(\rho-)\le L \ \P\text{-a.s.}$ Fix $R>(2L)^{1/p}$.
By Lemma \ref{lem:4.1-regularization}(iv), a deterministic subsequence, still indexed by $m$, satisfies
\begin{equation}
\label{eq:4.2-localized-Lp-convergence}
\sup_{t\le\rho}\|U_m(t)-u(t)\|_p^p +\int_0^\rho\|U_m(s)-u(s)\|_{3p}^p\,\dd s \longrightarrow0 \quad\P\text{-a.s.}
\end{equation}
Strong convergence and contractivity of $S_m$ also imply $v_m\to u$ in $L^p(0,\rho;L^{3p})$, pathwise.
Furthermore, since $\sup_{s<\rho}\|u(s)\|_p\le L^{1/p}<R$, \eqref{eq:4.2-localized-Lp-convergence} implies that $a_m(s)=1$ for every $s<\rho$ once $m$ is sufficiently large, pathwise.
Define
\begin{equation}
X_m:=\sup_{t\le\rho}\|U_m(t)\|_2^2 +\nu\int_0^\rho\|\nabla U_m(s)\|_2^2\,\dd s +\int_0^\rho a_m(s)\mathcal T_0(v_m(s))\,\dd s.
\end{equation}
By \eqref{eq:4.2-regularized-energy} and Fatou's lemma, $\liminf_mX_m<\infty$ almost surely.
For each such path with $\rho>0$, take a further subsequence realizing this lower limit.
It is bounded in $L^\infty(0,\rho;\mathcal H)\cap L^2(0,\rho;H^1_\sigma)$.
Its weak-* and weak limits in these respective spaces are identified with $u$ by \eqref{eq:4.2-localized-Lp-convergence}.
The same uniform $L^p$ convergence identifies the distributional limit at every time, including $\rho$, and hence gives the corresponding lower bound for the supremum of the $\mathcal H$ norm.
After further extraction, $v_m\to u$ almost everywhere on $(0,\rho)\times\R^3$.
Continuity and nonnegativity of $g_N$ therefore permit Fatou's lemma for the taming integral.
Consequently,
\begin{equation*}
\sup_{t\le\rho}\|u(t)\|_2^2 +\nu\int_0^\rho\|\nabla u(s)\|_2^2\,\dd s +\int_0^\rho\mathcal T_0(u(s))\,\dd s \le\liminf_mX_m
\end{equation*}
almost surely.
On $\{\rho=0\}$ this follows directly from $S_m u_0\to u_0$ in $\mathcal H$.
Taking expectations gives,
\begin{equation}
\label{eq:4.2-localized-energy}
\E\left[ \sup_{t\le\rho}\|u(t)\|_2^2 +\nu\int_0^\rho\|\nabla u(s)\|_2^2\,\dd s +\int_0^\rho\mathcal T_0(u(s))\,\dd s \right] \le C_T\left(1+\E\|u_0\|_2^2\right).
\end{equation}
In particular, the constant is independent of $n,L$.
The terminal post-jump value is included in the supremum.

\emph{Step 3:} Write $\widehat u(t)=u(t\wedge\rho)$ for $t\in[0,T]$.
The pointwise $\mathcal H$ membership established above, the pathwise $\mathcal H$ bound, and $L^p$ c\`adl\`ag regularity imply that $\widehat u$ is weakly c\`adl\`ag in $\mathcal H$.
Indeed, its scalar pairings with smooth compactly supported fields have the required limits, and the uniform $\mathcal H$ bound extends this assertion to every element of $\mathcal H$ by density.
The $L^p$ left limits are the resulting weak $\mathcal H$ left limits.
The pointwise $\mathcal H$ membership, the pathwise $\mathcal H$ bound, and the $L^p$-c\`adl\`ag regularity imply that $\widehat u$ is weakly c\`adl\`ag in $\mathcal H$.
Indeed, scalar pairings with smooth compactly supported fields have the required right and left limits, and the uniform $\mathcal H$ bound extends this assertion to all elements of $\mathcal H$ by density.
Moreover, if $t_k\uparrow t$, then $u(t_k)\to u(t-)$ in $L^p$ and the uniform $\mathcal H$ bound and uniqueness of distributional limits therefore show that $u(t_k)\rightharpoonup u(t-)$ in $\mathcal H$.
Thus the $L^p$ left limit is also the weak $\mathcal H$ left limit.
By separability, $\widehat u$ admits an $\mathcal H$-adapted version and $\widehat u_-$ an $\mathcal H$-valued predictable version.

Condition \eqref{eq:2-taming-local}, \eqref{eq:2-taming-coercive} gives, pointwise, $|v|^4\le2\nu g_N(|v|^2)|v|^2 +\nu c_{N,\nu}|v|^2.$ Together with \eqref{eq:4.2-localized-energy}, this yields
\begin{equation}
\label{eq:4.2-L4-integrability}
\E\int_0^\rho \left(\|u(s)\|_{H^1}^2+\|u(s)\|_4^4\right)\dd s <\infty.
\end{equation}
On $(0,\rho]$, define
\begin{equation*}
\begin{aligned}
\mathcal A_1(s) &=\nu\Delta u(s)-\mathcal P\nabla\cdot(u(s)\otimes u(s)), & \mathcal A_2(s) &=-\mathcal P(g_N(|u(s)|^2)u(s)), \\
b(s)&=\Sigma(s,u(s-)), & c(s,z)&=\mathcal G(s,u(s-),z),
\end{aligned}
\end{equation*}
and extend these coefficients by zero outside $(0,\rho]$.
The drift coefficients are defined up to $\dd\P\otimes\dd s$-null sets, The noise coefficients use their predictable versions.
The indicator $\mathds1_{(0,\rho]}$ is predictable, and
\begin{equation*}
\|\mathcal A_1(s)\|_{H^{-1}} \le C\left(\nu\|\nabla u(s)\|_2+\|u(s)\|_4^2\right), \qquad \|\mathcal A_2(s)\|_{4/3} \le C\|u(s)\|_4^3, \qquad 0<s<\rho.
\end{equation*}
Hence the zero extensions satisfy
\begin{equation*}
\begin{aligned}
\mathcal A_1&\in L^2(\Omega\times(0,T);H^{-1}_\sigma), & \mathcal A_2&\in L^{4/3}(\Omega\times(0,T);L^{4/3}_\sigma), \\
b&\in L^2(\Omega\times(0,T);\gamma(\mathcal U;\mathcal H)), & c&\in L^2(\Omega\times(0,T)\times Z;\mathcal H).
\end{aligned}
\end{equation*}
The last two assertions follow from \eqref{eq:2-raw-L2-growth} and \eqref{eq:4.2-localized-energy}.
Let $J_\varepsilon=e^{\varepsilon\Delta}$ and set $u^\varepsilon=J_\varepsilon\widehat u$.
Applying $J_\varepsilon$ to the stopped distributional equation gives the $\mathcal H$-valued semimartingale identity
\begin{equation}
\label{eq:4.2-mollified-equation}
u^\varepsilon(t) =J_\varepsilon u_0 +\int_0^t J_\varepsilon (\mathcal A_1(s)+\mathcal A_2(s))\,\dd s +\int_0^t J_\varepsilon b(s)\,\dd W_s +\int_0^t\int_Z J_\varepsilon c(s,z)\, \widetilde N(\dd s,\dd z).
\end{equation}
$J_\varepsilon:H^{-1}\to \mathcal H$ and $J_\varepsilon:L^{4/3}\to L^2$ are bounded, so the drift is pathwise integrable in $\mathcal H$.
The stochastic integrals are square-integrable $\mathcal H$-valued martingales.
Thus $u^\varepsilon$ has an adapted strongly $\mathcal H$-c\`adl\`ag version.

For $\varepsilon,\delta>0$, put $D_{\varepsilon,\delta}=J_\varepsilon-J_\delta$ and $w^{\varepsilon,\delta}=u^\varepsilon-u^\delta$.
It\^o's formula for $\|w^{\varepsilon,\delta}\|_2^2$, self-adjointness of the mollifiers, and the BDG, Davis, and compensator estimates used in \eqref{eq:4.2-martingale} give
\begin{equation}
\label{eq:4.2-Hilbert-Cauchy}
\begin{aligned}
\E\sup_{t\le T} \|w^{\varepsilon,\delta}(t)\|_2^2 \le&C\E\|D_{\varepsilon,\delta}u_0\|_2^2 +C\sum_{i=1}^2\E\int_0^\rho \left|\left\langle\mathcal A_i(s), D_{\varepsilon,\delta}^2u(s)\right\rangle\right|\dd s \\
&+C\E\int_0^T \left[\|D_{\varepsilon,\delta}b(s)\|_{\gamma(\mathcal U;H)}^2 +\int_Z\|D_{\varepsilon,\delta}c(s,z)\|_2^2\, \mu(\dd z)\right]\dd s .
\end{aligned}
\end{equation}
The two duality pairings are those of $H^{-1},H^1$ and $L^{4/3},L^4$, respectively.
The square on $D_{\varepsilon,\delta}$ denotes operator composition.
For the quadratic jump martingale in this calculation, positivity of $\|D_{\varepsilon,\delta}c\|_2^2$ again gives its expected supremum bound from the second moment alone.

The right-hand side of \eqref{eq:4.2-Hilbert-Cauchy} converges to zero as $\varepsilon,\delta\downarrow0$.
For the drift terms, use \eqref{eq:4.2-L4-integrability}, strong convergence of $J_\varepsilon$ in $H^1$ and $L^4$, and H\"older's inequality with exponents $(2,2)$ and $(4/3,4)$.
For the stochastic terms, use strong convergence in $H$, the ideal property of $\gamma$-radonifying operators, and dominated convergence in the displayed square-integrable coefficient spaces.
Consequently, $(u^\varepsilon)$ is Cauchy in expected $\mathcal H$-supremum norm.
A deterministic subsequence converges almost surely uniformly on $[0,T]$, and its limit is adapted and strongly $\mathcal H$-c\`adl\`ag.
Since $J_\varepsilon\widehat u(t)\to\widehat u(t)$ in $\mathcal H$ at every time on the common full probability set where $\widehat u(t)\in \mathcal H$ for all $t$, this limit is $\widehat u$ itself.
Thus
\begin{equation*}
u(\cdot\wedge\rho) \in L^2\bigl(\Omega;\mathbb D([0,T];\mathcal H)\bigr).
\end{equation*}
The convergence just proved is uniform in time in mean square, and hence also implies convergence in the Skorokhod $J_1$ topology.
The stochastic intervals in \eqref{eq:4.2-mollified-equation} include $\rho$, so this argument retains its possible Poisson jump.

\emph{Step 4: .} Fix $n$ and $T$, and let $L$ increase through the positive integers.
Since $T\wedge\tau_n<\tau_{\max}$ and $\mathcal Q_u(T\wedge\tau_n)<\infty$ almost surely,
\begin{equation*}
\rho_{n,L,T}=T\wedge\tau_n \quad\text{for all sufficiently large }L, \quad\text{pathwise}.
\end{equation*}
The strongly $\mathcal H$-c\`adl\`ag stopped representatives from Step 3 therefore give \eqref{eq:4.2-stopped-regularity} on a common full-probability set.
Monotone convergence in \eqref{eq:4.2-localized-energy} yields the same bound with $\rho$ replaced by $T\wedge\tau_n$, including that terminal value.

Finally, $\tau_n\uparrow\tau_{\max}$ with $\tau_n<\tau_{\max}$ almost surely.
The intervals $[0,T\wedge\tau_n]$ exhaust $\{t:0\le t\le T,\ t<\tau_{\max}\}$.
Their compatible strongly $\mathcal H$-c\`adl\`ag representatives define the asserted representative on $[0,\tau_{\max})$.
Another application of monotone convergence proves \eqref{eq:4.2-maximal-energy}.
\end{proof}
Under the gradient-noise hypothesis, the coercivity of the taming term also controls the differentiated equation.
For $H^1$ initial data, this yields $H^1$ persistence and $H^2$ dissipation needed for the weighted $L^p$ estimates in the next section.
\begin{theorem}
[$H^1$-persistence] \label{thm:4.3-H1-persistence} Retain the hypotheses of Proposition \ref{prop:4.2-L2-persistence} and assume the gradient-noise bound \eqref{eq:2-H1}.
Suppose, in addition, that
\begin{equation*}
u_0\in L^p(\Omega,\mathcal F_0;L^p_\sigma(\R^3)) \cap L^2(\Omega,\mathcal F_0;H^1_\sigma(\R^3)).
\end{equation*}
Then the maximal solution $u$ has an adapted representative with strongly $H^1_\sigma$-c\`adl\`ag paths on $[0,\tau_{\max})$.
For every $n\ge1$ and $T<\infty$,
\begin{equation}
\label{eq:4.3-stopped-regularity}
u(\cdot\wedge\tau_n) \in L^2\bigl(\Omega;\mathbb D([0,T];H^1_\sigma)\bigr), \quad \E\int_0^{T\wedge\tau_n} \|u(s)\|_{H^2}^2\,\dd s<\infty,
\end{equation}
where the path-space moment is understood in the expected-supremum sense.
Moreover,
\begin{equation}
\label{eq:4.3-maximal-energy}
\E\left[ \sup_{\substack{0\le t\le T \\
t<\tau_{\max}}} \|u(t)\|_{H^1}^2 +\nu\int_0^{T\wedge\tau_{\max}} \|u(s)\|_{H^2}^2\,\dd s \right] \le C_T\left(1+\E\|u_0\|_{H^1}^2\right).
\end{equation}
The constant $C_T$ depends only on $T,\nu,c_{N,\nu}$ and the constants $C_0,C_1$ in \eqref{eq:2-raw-L2-growth}--\eqref{eq:2-H1}.
It is independent of the intrinsic exit level, the announcing sequence, and all scalar and spatial regularization parameters.
An integral ending at a finite $\tau_{\max}$ denotes its increasing limit from below.
No state at $\tau_{\max}$ is asserted.
In particular, defined the accumulated energy functional as,
\begin{equation}
\label{eq:4.3-accumalted-regularity}
\Gamma(t):=\int_0^t \left(1+\|u(s)\|_{H^1}^{2p}+\|u(s)\|_\infty\right)\dd s, \qquad t<\tau_{\max},
\end{equation}
is adapted and locally absolutely continuous, and
\begin{equation}
\label{eq:4.3-regularity}
\Gamma\bigl((T\wedge\tau_{\max})-\bigr)<\infty \quad\P\text{-a.s.}
\end{equation}
\end{theorem}

\begin{proof}
\emph{Step 1:} Let $U_m=u_m^R$, $v_m=S_mU_m$, and retain the notation $a_m,B_m,K_m,q_{m,1},\mathcal T_1$ of Lemma \ref{lem:4.1-regularization}.
Self-adjointness of $S_m$, commutation with derivatives, and its $L^2$ contractivity give
\begin{equation}
\label{eq:4.3-tamed-coercivity}
\begin{aligned}
2a_m((v_m\cdot\nabla)v_m,\Delta v_m)_2 -2a_m\mathcal T_1(v_m) &\le \nu a_m\|\Delta v_m\|_2^2 +a_m\int_{\R^3} \left(\frac{|v_m|^2}{\nu}-2g_N(|v_m|^2)\right) |\nabla v_m|^2\,\dd x \\
&\quad\le \nu\|\Delta U_m\|_2^2 +c_{N,\nu}\|\nabla U_m\|_2^2.
\end{aligned}
\end{equation}
Here the term containing $g_N'$ was discarded using its nonnegativity, and $0\le a_m\le1$.
No idempotence of $S_m$ is used.
The Wiener BDG inequality, the Davis inequality, and the compensator identity yield, for $t\le T$ and $\delta>0$,
\begin{equation}
\label{eq:4.3-gradient-martingales}
\begin{aligned}
\E\sup_{r\le t}|M^W_{m,1}(r)| +\E\sup_{r\le t}|M^{J,1}_{m,1}(r)| &\le \delta\E\sup_{r\le t}\|\nabla U_m(r)\|_2^2 +C_\delta\E\int_0^t q_{m,1}(s)\,\dd s, \\
\E\sup_{r\le t}|M^{J,2}_{m,1}(r)| &\le 2\E\int_0^t\int_Z \|\nabla K_m(s,z)\|_2^2\,\mu(\dd z)\,\dd s, \\
q_{m,1}(s)&\le C_1(1+\|U_m(s-)\|_{H^1}^2).
\end{aligned}
\end{equation}
These are the estimates of Lemma \ref{lem:4.1-regularization}(iii).
The quadratic jump term is treated as the difference of its nonnegative uncompensated integral and compensator, so only second moments of the $H^1$-valued jump coefficient are used.
Apply \eqref{eq:4.1-H1-identity}, use \eqref{eq:4.3-tamed-coercivity}--\eqref{eq:4.3-gradient-martingales}, and absorb the supremum term with a fixed sufficiently small $\delta$.
With constants independent of $m,R$,
\begin{equation*}
\begin{aligned}
\E\left[ \sup_{r\le t}\|\nabla U_m(r)\|_2^2 +\nu\int_0^t\|\Delta U_m(s)\|_2^2\,\dd s\right] \le C\E\|\nabla u_0\|_2^2 +C\int_0^t\left( 1+\E\|U_m(s)\|_2^2 +\E\sup_{r\le s}\|\nabla U_m(r)\|_2^2 \right)\dd s.
\end{aligned}
\end{equation*}
The uniform $L^2$ estimate \eqref{eq:4.2-regularized-energy} and Gronwall's inequality therefore imply
\begin{equation}
\label{eq:4.3-regularized-energy}
\sup_{m\ge1,\,R\ge1}\E\left[ \sup_{t\le T}\|U_m(t)\|_{H^1}^2 +\nu\int_0^T\|U_m(s)\|_{H^2}^2\,\dd s\right] \le C_T\left(1+\E\|u_0\|_{H^1}^2\right).
\end{equation}
Here we used $\|v\|_{H^2}^2\le C(\|v\|_2^2+\|\Delta v\|_2^2)$ on $\R^3$.

\emph{Step 2:} Fix $n,L\ge1$ and $T<\infty$, and set $\rho:=T\wedge\tau_n\wedge\lambda_L.$ Then $\rho<\tau_{\max}$ and $\mathcal Q_u(\rho-)\le L$.
Fix $R>(2L)^{1/p}$.
By Lemma \ref{lem:4.1-regularization}(iv), a deterministic subsequence of $U_m$ converges almost surely to $u$ uniformly on $[0,\rho]$ in $L^p$.

For each path on which the lower limit of the energies in \eqref{eq:4.3-regularized-energy} is finite, take a further subsequence realizing that lower limit.
It is bounded in $L^\infty(0,\rho;H^1_\sigma)\cap L^2(0,\rho;H^2_\sigma)$.
Weak-star compactness in the first space and weak compactness in the second identify both limits with $u$, by convergence in distributions.
At every $t\in[0,\rho]$, the same uniform $H^1$ bound and the $L^p$ convergence identify every weak $H^1$ subsequential limit of $U_m(t)$ with $u(t)$.
Weak lower semicontinuity therefore controls the full supremum, including the terminal value.
Fatou's lemma gives
\begin{equation}
\label{eq:4.3-localized-energy}
\E\left[ \sup_{t\le\rho}\|u(t)\|_{H^1}^2 +\nu\int_0^\rho\|u(s)\|_{H^2}^2\,\dd s\right] \le C_T\left(1+\E\|u_0\|_{H^1}^2\right).
\end{equation}
On $\{\rho=0\}$ this follows directly from the initial datum.
In particular, every stopped state belongs to $H^1_\sigma$ on a common event of probability one.

\emph{Step 3:} By Proposition \ref{prop:4.2-L2-persistence}, $u(\cdot\wedge\rho)$ is strongly $L^2_\sigma$-c\`adl\`ag.
Together with its pathwise $H^1$ bound, this implies weak $H^1_\sigma$ right-continuity and weak $H^1_\sigma$ left limits.
Indeed, for smooth $\psi$, $(u,\psi)_{H^1}=(u,(I-\Delta)\psi)_2$, and density extends this conclusion to every $H^1$ test vector.
The weak $H^1$ left limits coincide with the $L^2$ left limits.
Separability consequently gives an $H^1$-valued progressively measurable version and predictable $H^1$-valued left limits.
For almost every $s\in(0,\rho)$, define the full drift
\begin{equation*}
A(s):=\nu\Delta u(s) -\mathcal P((u(s)\cdot\nabla)u(s)) -\mathcal P(g_N(|u(s)|^2)u(s)),
\end{equation*}
and extend it by zero outside this interval and at states outside $H^2$.
This gives a progressively measurable $L^2_\sigma$-valued process.
$H^2$ is a Borel subset of $H^1$, and the drift map is continuous from $H^2$ to $L^2$.
The Sobolev and Gagliardo--Nirenberg inequalities on $\R^3$ give, for $v\in H^2$,
\begin{equation}
\label{eq:4.3-spatial-interpolation}
\|v\|_6\le C\|\nabla v\|_2, \qquad \|v\|_\infty \le C\|\nabla v\|_2^{1/2}\|\Delta v\|_2^{1/2}.
\end{equation}
The $L^2$ boundedness of $\mathcal P$, cubic taming growth, and Young's inequality imply
\begin{equation}
\label{eq:4.3-drift-integrability}
\begin{aligned}
\|A(s)\|_2^2 &\le C\left( \|\Delta u(s)\|_2^2 +\|u(s)\|_\infty^2\|\nabla u(s)\|_2^2 +\|u(s)\|_6^6\right) \\
&\le C\left(\|\Delta u(s)\|_2^2 +\|u(s)\|_{H^1}^6\right).
\end{aligned}
\end{equation}
Thus $A\in L^2(0,\rho;L^2_\sigma)$ almost surely.
No expectation of the sixth power is required.
Set
\begin{equation*}
\Lambda(t):=\int_0^{t\wedge\rho}\|A(s)\|_2^2\,\dd s, \qquad \theta_K:=\inf\{t\ge0:\Lambda(t)\ge K\}\wedge\rho, \qquad K\ge1.
\end{equation*}
Since $\Lambda$ is continuous and adapted, $\theta_K$ is a stopping time.
Moreover,
\begin{equation}
\label{eq:4.3-stopped-drift-bound}
\E\int_0^{\theta_K}\|A(s)\|_2^2\,\dd s\le K, \qquad \theta_K=\rho\ \text{for all sufficiently large }K \quad\P\text{-a.s.}
\end{equation}

Write $u^K(t)=u(t\wedge\theta_K)$ and define $b_K,c_K$ on $(0,\theta_K]$ by $b_K(s)=\Sigma(s,u(s-))$ and $c_K(s,z)=\mathcal G(s,u(s-),z)$, with zero extension elsewhere.
The interval indicator is predictable.
The coefficient measurability assumptions, the predictable left limits, and \eqref{eq:2-raw-L2-growth}--\eqref{eq:2-H1} give
\begin{equation}
\label{eq:4.3-H1-noise}
\E\int_0^T\left[ \|b_K(s)\|_{\gamma(\mathcal U;H^1)}^2 +\int_Z\|c_K(s,z)\|_{H^1}^2\,\mu(\dd z) \right]\dd s<\infty.
\end{equation}
Here the Hilbert-space radonifying norm is the Hilbert--Schmidt norm.
In particular, the displayed $H^1$ norm of each coefficient includes its $L^2$ component.

Let $J_\varepsilon=e^{\varepsilon\Delta}$ and $u^{K,\varepsilon}=J_\varepsilon u^K$.
The stopped distributional equation yields the $H^1_\sigma$-valued semimartingale identity
\begin{equation*}
u^{K,\varepsilon}(t) = J_\varepsilon u_0 +\int_0^{t\wedge\theta_K}J_\varepsilon A(s)\ \dd s +\int_0^tJ_\varepsilon b_K(s)\ \dd W_s +\int_0^t\int_ZJ_\varepsilon c_K(s,z)\ \widetilde N(\dd s,\dd z).
\end{equation*}
$J_\varepsilon:L^2\to H^1$ is bounded, the drift satisfies \eqref{eq:4.3-stopped-drift-bound}, and the noise satisfies \eqref{eq:4.3-H1-noise}.

Put $D_{\varepsilon,\eta}=J_\varepsilon-J_\eta$ and $w=D_{\varepsilon,\eta}u^K$.
The Hilbert-space It\^o formula for $\|w\|_{H^1}^2$, followed by the same Wiener and jump estimates as in \eqref{eq:4.3-gradient-martingales}, gives
\begin{equation}
\label{eq:4.3-H1-Cauchy}
\begin{aligned}
\E\sup_{t\le T}\|w(t)\|_{H^1}^2 \le&C\E\|D_{\varepsilon,\eta}u_0\|_{H^1}^2 +C\E\int_0^{\theta_K} \left|(A(s),(I-\Delta)D_{\varepsilon,\eta}^2u(s))_2\right| \ \dd s \\
&+C\E\int_0^T \|D_{\varepsilon,\eta}b_K(s)\|_{\gamma(\mathcal U;H^1)}^2 \ \dd s +C\E\int_0^T\int_Z \|D_{\varepsilon,\eta}c_K(s,z)\|_{H^1}^2 \ \mu(\dd z)\ \dd s.
\end{aligned}
\end{equation}
Self-adjointness and commutation with derivatives identify the drift pairing as $(D_{\varepsilon,\eta}u,D_{\varepsilon,\eta}A)_{H^1} =(A,(I-\Delta)D_{\varepsilon,\eta}^2u)_2$.
The linear martingale terms are absorbed by Young's inequality.
The quadratic jump term is controlled by its $L^1$ compensator.

Every term on the right of \eqref{eq:4.3-H1-Cauchy} tends to zero as $\varepsilon,\eta\downarrow0$.
For the drift term, Cauchy--Schwarz gives the bound
\begin{equation*}
K^{1/2}\left( \E\int_0^{\theta_K} \|(I-\Delta)D_{\varepsilon,\eta}^2u(s)\|_2^2\,\dd s \right)^{1/2}\longrightarrow0
\end{equation*}
by \eqref{eq:4.3-localized-energy} and strong convergence of the uniformly bounded operators $J_\varepsilon$ on $H^2$.
The initial and noise terms converge by strong convergence on $H^1$, the radonifying ideal property, and dominated convergence using \eqref{eq:4.3-H1-noise}.

Hence $(u^{K,\varepsilon})$ is Cauchy in the expected square of the uniform $H^1$ norm.
A deterministic subsequence converges almost surely uniformly in $H^1$ to an adapted strongly c\`adl\`ag process.
This convergence is stronger than convergence in probability in the Skorokhod $J_1$ topology.
Since $J_\varepsilon u^K(t)\to u^K(t)$ in $H^1$ for every $t$ on the common event by Step 2, the limit is a version of $u^K$.
Taking integer $K\to\infty$ in \eqref{eq:4.3-stopped-drift-bound} proves strong $H^1$-c\`adl\`ag regularity of $u(\cdot\wedge\rho)$.
All stochastic integrations use $(0,\theta_K]$, so the terminal jump is retained throughout.

\emph{Step 4:} For fixed $n,T$, the quantity $\mathcal Q_u(T\wedge\tau_n)$ is finite almost surely.
Consequently, as integer $L\to\infty$, $\rho_{n,L,T}=T\wedge\tau_n$ eventually on each such path.
Thus the regularity and estimate obtained above hold on the closed interval $[0,T\wedge\tau_n]$, with the same constant.
This eventual equality also retains its terminal value if that value is attained by a jump.
The strict announcing sequence then exhausts $\{t:0\le t\le T,\ t<\tau_{\max}\}$.
Monotone convergence proves \eqref{eq:4.3-maximal-energy}, and a countable exhaustion in $n$ and integer $T$ yields the claimed adapted $H^1_\sigma$-c\`adl\`ag representative and \eqref{eq:4.3-stopped-regularity}.
Set $\zeta=T\wedge\tau_{\max}$ and
\begin{equation*}
M_T:=\sup_{\substack{0\le t\le T \\
t<\tau_{\max}}} \|u(t)\|_{H^1}^2, \qquad D_T:=\int_0^\zeta\|u(s)\|_{H^2}^2\,\dd s.
\end{equation*}
Both variables are finite almost surely.
By \eqref{eq:4.3-spatial-interpolation} and H\"older's inequality,
\begin{equation*}
\begin{aligned}
\int_0^\zeta\|u(s)\|_\infty\,\dd s &\le C T^{3/4}M_T^{1/4}D_T^{1/4}, & \Gamma(\zeta-) &\le T(1+M_T^p) +C T^{3/4}M_T^{1/4}D_T^{1/4}<\infty \quad\P\text{-a.s.}
\end{aligned}
\end{equation*}
The integrand in \eqref{eq:4.3-accumalted-regularity} has a progressively measurable representative and is locally integrable, which proves the remaining assertions about $\Gamma$.
\end{proof}
To treat finite-energy data without an initial $H^1$ assumption, we multiply the differentiated energy estimate by time.
The vanishing weight removes the initial gradient term and yields $H^1$-c\`adl\`ag regularity and $H^2$ dissipation on compact subintervals of $(0,\tau_{\max})$.
\begin{proposition}
[Time-weighted $H^1$ regularization] \label{prop:4.4-time-weighted-H1} Assume the hypotheses of Lemma \ref{lem:4.1-regularization}, including \eqref{eq:2-raw-L2-growth}, \eqref{eq:2-H1}, and \eqref{eq:2-taming-local}, \eqref{eq:2-taming-coercive}.
In particular, let
\begin{equation}
\label{eq:4.4-initial-data}
u_0\in L^p(\Omega,\mathcal F_0;L^p_\sigma(\R^3)) \cap L^2(\Omega,\mathcal F_0;L^2_\sigma(\R^3)), \qquad p>3,
\end{equation}
and let $(u,(\tau_n),\tau_{\max})$ be the maximal solution of \cref{Sec:3}.
Then $u$ has an adapted representative whose paths are strongly $H^1_\sigma$-c\`adl\`ag on every compact subinterval of $(0,\tau_{\max})$.
For every deterministic $T>0$,
\begin{equation}
\label{eq:4.4-time-weighted-energy}
\E\left[ \sup_{\substack{0<t\le T \\
t<\tau_{\max}}} t\|u(t)\|_{H^1}^2 +\nu\int_0^{T\wedge\tau_{\max}} s\|u(s)\|_{H^2}^2\,\dd s \right] \le C_T\left(1+\E\|u_0\|_2^2\right).
\end{equation}
The constant $C_T$ depends only on $T,\nu,c_{N,\nu},C_0,C_1$ and is independent of the announcing sequence, the intrinsic exit level, and all scalar and spatial regularization parameters.
An integral ending at a finite $\tau_{\max}$ is understood as its increasing limit from below.

Consequently, for $0<\varepsilon<T$,
\begin{equation}
\label{eq:4.4-positive-time-energy}
\E\left[ \sup_{\substack{\varepsilon\le t\le T \\
t<\tau_{\max}}} \|u(t)\|_{H^1}^2 +\nu\int_{(\varepsilon,T)\cap(0,\tau_{\max})} \|u(s)\|_{H^2}^2\,\dd s \right] \le \frac{C_T}{\varepsilon} \left(1+\E\|u_0\|_2^2\right),
\end{equation}
where the supremum over an empty set is zero.
In particular, on $\{T<\tau_{\max}\}$,
\begin{equation}
\label{eq:4.4-positive-time-paths}
u\in\mathbb D([\varepsilon,T]; L^p_\sigma\cap H^1_\sigma) \cap L^2(\varepsilon,T;H^2_\sigma) \quad\P\text{-a.s.},
\end{equation}
with the sum norm on $L^p_\sigma\cap H^1_\sigma$.
\end{proposition}

\begin{proof}
\emph{Step 1: Uniform time-weighted estimates.} Let $U_m=u_m^R$ solve \eqref{eq:4.1-regularized-equation}, and use the notation of Lemma \ref{lem:4.1-regularization}.
Put
\begin{equation*}
X_m(t):=\|\nabla U_m(t)\|_2^2, \qquad Y_m(t):=\|\Delta U_m(t)\|_2^2.
\end{equation*}
At fixed $m,R$, the mollified It\^o argument in that lemma justifies \eqref{eq:4.1-H1-identity} with finite expected suprema.
Integration by parts against the deterministic function $t$, followed by \eqref{eq:4.1-coercivity}, gives
\begin{equation}
\label{eq:4.4-weighted-gradient}
tX_m(t)+\nu\int_0^t sY_m(s)\,\dd s \le\int_0^t X_m(s)\,\dd s +c_{N,\nu}\int_0^t sX_m(s)\,\dd s +\int_0^t s q_{m,1}(s)\,\dd s +\mathcal M_m(t),
\end{equation}
where $\mathcal M_m=\mathcal M_m^W+ \mathcal M_m^{J,1}+\mathcal M_m^{J,2}$ and
\begin{equation*}
\begin{aligned}
\mathcal M_m^W(t) &:=2\int_0^t s \left\langle(\nabla B_m(s))^*\nabla U_m(s-), \dd W_s\right\rangle_{\mathcal U}, \\
\mathcal M_m^{J,1}(t) &:=2\int_0^t\int_Z s (\nabla U_m(s-),\nabla K_m(s,z))_2 \,\widetilde N(\dd s,\dd z), \\
\mathcal M_m^{J,2}(t) &:=\int_0^t\int_Z s\|\nabla K_m(s,z)\|_2^2 \,\widetilde N(\dd s,\dd z).
\end{aligned}
\end{equation*}
The initial gradient term in \eqref{eq:4.4-weighted-gradient} is zero.

Set $Z_m(t):=\sup_{0\le r\le t}rX_m(r)$.
The Wiener BDG inequality and the Davis inequality, with Young's inequality and the compensator identity, yield for every $\delta>0$,
\begin{equation}
\label{eq:4.4-weighted-martingale}
\begin{aligned}
\E\sup_{r\le t}|\mathcal M_m^W(r)| +\E\sup_{r\le t}|\mathcal M_m^{J,1}(r)| &\le \delta\E Z_m(t) +C_\delta\E\int_0^t s q_{m,1}(s)\,\dd s, \\
\E\sup_{r\le t}|\mathcal M_m^{J,2}(r)| &\le 2\E\int_0^t\int_Z s\|\nabla K_m(s,z)\|_2^2\,\mu(\dd z)\,\dd s.
\end{aligned}
\end{equation}

Indeed, since $U_m$ is $H^1_\sigma$-c\`adl\`ag, for $0<s\le t$ one has $$sX_m(s-) = \lim_{r\uparrow s}rX_m(r) \le Z_m(t).$$ Consequently, the quadratic variation of the Wiener martingale satisfies $$ [\mathcal M_m^W]_t = 4\int_0^t s^2 \|(\nabla B_m(s))^*\nabla U_m(s-)\|_{\mathcal U}^2 \,\dd s \le 4Z_m(t)\int_0^t s\|\nabla B_m(s)\|_{\gamma(\mathcal U;L^2)}^2\,\dd s .
$$ Hence the BDG inequality and Young's inequality give $$\E\sup_{r\le t}|\mathcal M_m^W(r)| \le \delta\E Z_m(t) + C_\delta\E\int_0^t s\|\nabla B_m(s)\|_{\gamma(\mathcal U;L^2)}^2\,\dd s$$ Similarly, the optional quadratic variation of the linear jump martingale is $$ \ [\mathcal M_m^{J,1}]_t = 4\int_0^t\int_Z s^2 |(\nabla U_m(s-),\nabla K_m(s,z))_2|^2 \,N(\dd s,\dd z) \le 4Z_m(t)\int_0^t\int_Z s\|\nabla K_m(s,z)\|_2^2 \,N(\dd s,\dd z).
$$ Davis' inequality followed by Young's inequality and the compensator identity therefore yields $$ \E\sup_{r\le t}|\mathcal M_m^{J,1}(r)| \le \delta\E Z_m(t) + C_\delta\E\int_0^t\int_Z s\|\nabla K_m(s,z)\|_2^2 \,\mu(\dd z)\,\dd s .
$$ Combining the preceding two estimates with $\delta/2 $ iwth $\delta$ gives the first inequality in \eqref{eq:4.4-weighted-martingale}.
For the quadratic jump term, set $$H_m(s,z):=s\|\nabla K_m(s,z)\|_2^2\ge0.$$ Since $$\mathcal M_m^{J,2}(t) = \int_0^t\int_Z H_m(s,z)\,N(\dd s,\dd z) - \int_0^t\int_ZH_m(s,z)\,\mu(\dd z)\,\dd s,$$ one has pathwise $$\sup_{r\le t}|\mathcal M_m^{J,2}(r)| \le \int_0^t\int_ZH_m\,N + \int_0^t\int_ZH_m\,\mu(\dd z)\,\dd s .$$ Taking expectations and using the compensator identity proves the second inequality in \eqref{eq:4.4-weighted-martingale}.
In particular, only second moments of the $H^1$-valued jump coefficient are required.
These estimates may first be applied after the usual martingale localization.
Their right-hand sides are finite for fixed $m,R$, and Fatou's lemma removes this auxiliary localization.
Finally, since $U_m(s-)=U_m(s)$ for Lebesgue-a.e. $s$, the bound for $q_{m,1}$ in \eqref{eq:4.1-coercivity} gives $$
\begin{aligned}
\E\int_0^T s q_{m,1}(s)\,\dd s &\le C_1\int_0^T s\left( 1+\E\|U_m(s)\|_2^2+\E X_m(s) \right)\dd s \le C_T + C_T\E\sup_{s\le T}\|U_m(s)\|_2^2 + C_T\E\int_0^T X_m(s)\,\dd s \\
&\le C_T\left(1+\E\|u_0\|_2^2\right),
\end{aligned}
$$ where the last step follows from \eqref{eq:4.2-regularized-energy}.
The same estimate also controls $\E\int_0^T(1+c_{N,\nu}s)X_m(s)\,\dd s$.

Taking suprema in \eqref{eq:4.4-weighted-gradient}, using \eqref{eq:4.4-weighted-martingale}, and absorbing the term containing $\delta$ therefore gives
\begin{equation*}
\E\left[ Z_m(T)+\nu\int_0^T sY_m(s)\,\dd s \right] \le C_T\left(1+\E\|u_0\|_2^2\right).
\end{equation*}
The $L^2$ bound and $\|v\|_{H^2}^2\le C(\|v\|_2^2+\|\Delta v\|_2^2)$ imply
\begin{equation}
\label{eq:4.4-approximate}
\sup_{m\ge1,\,R\ge1}\E\left[ \sup_{0<t\le T}t\|U_m(t)\|_{H^1}^2 +\nu\int_0^T s\|U_m(s)\|_{H^2}^2\,\dd s \right] \le C_T\left(1+\E\|u_0\|_2^2\right).
\end{equation}

\emph{Step 2:} Fix $n,L\ge1$ and set $\rho:=T\wedge\tau_n\wedge\lambda_L,$ where $\lambda_L$ is defined by \eqref{eq:3.6-intrinsic-exits}.
Then $\rho<\tau_{\max}$ and $\mathcal Q_u(\rho-)\le L$ almost surely.
For fixed $R>(2L)^{1/p}$, Lemma \ref{lem:4.1-regularization}(iv) supplies a deterministic subsequence such that
\begin{equation}
\label{eq:4.4-native-convergence}
\sup_{t\le\rho}\|U_m(t)-u(t)\|_p\longrightarrow0 \quad\P\text{-a.s.}
\end{equation}
By \eqref{eq:4.4-approximate} and Fatou's lemma, the lower limit of the corresponding stopped weighted energies is finite almost surely.

Fix such a path with $\rho>0$ and take a further subsequence realizing this lower limit.
The functions $s^{1/2}U_m(s)$ are bounded in $L^\infty(0,\rho;H^1_\sigma)\cap L^2(0,\rho;H^2_\sigma)$.
Weak-star compactness in the first space and weak compactness in the second identify both limits with $s^{1/2}u(s)$, by \eqref{eq:4.4-native-convergence} and distributional convergence on positive time intervals.
For every $t\in(0,\rho]$, the same bound makes $U_m(t)$ bounded in $H^1$ and its distributional limit is $u(t)$.
Weak lower semicontinuity at each time and in $L^2(0,\rho;H^2)$ yields
\begin{equation*}
\sup_{0<t\le\rho}t\|u(t)\|_{H^1}^2 +\nu\int_0^\rho s\|u(s)\|_{H^2}^2\,\dd s \le\liminf_m\left[ \sup_{0<t\le\rho}t\|U_m(t)\|_{H^1}^2 +\nu\int_0^\rho s\|U_m(s)\|_{H^2}^2\,\dd s \right].
\end{equation*}
For $\rho=0$ both sides are defined to be zero.
The extended $H^1$ norm is lower semicontinuous on $L^2$.
Together with the stopped $L^2$ c\`adl\`ag representative from Proposition \ref{prop:4.2-L2-persistence}, this also gives measurability of the supremum.
Taking expectations proves
\begin{equation}
\label{eq:4.4-weighted-energy}
\E\left[ \sup_{0<t\le\rho}t\|u(t)\|_{H^1}^2 +\nu\int_0^\rho s\|u(s)\|_{H^2}^2\,\dd s \right] \le C_T\left(1+\E\|u_0\|_2^2\right).
\end{equation}
In particular, the terminal state at a positive $\rho$ belongs to $H^1$.

\emph{Step 3:} Fix $0<h<T$, choose $\chi\in C^\infty([0,T];[0,1])$ with $\chi=0$ on $[0,h/2]$ and $\chi=1$ on $[h,T]$, and define the adapted process
\begin{equation*}
V(t):=\chi(t\wedge\rho)u(t\wedge\rho), \qquad 0\le t\le T.
\end{equation*}
It is strongly $L^2$ c\`adl\`ag and $V(0)=0$.
Since $\chi(s)^2\le C_hs$, \eqref{eq:4.4-weighted-energy} gives
\begin{equation}
\label{eq:4.4-Hilbert-energy}
\E\sup_{t\le T}\|V(t)\|_{H^1}^2 +\E\int_0^\rho\|V(s)\|_{H^2}^2\,\dd s<\infty.
\end{equation}
Its $L^2$ c\`adl\`ag paths and pathwise $H^1$ bound imply weak $H^1$ right-continuity and weak $H^1$ left limits.
These left limits coincide with its $L^2$ left limits.
The Lusin--Souslin theorem applies to the continuous injection $H^1_\sigma\hookrightarrow L^2_\sigma$ and gives a Borel inverse on its image.
Hence $V$ is $H^1$-progressively measurable and $V_-$ is $H^1$-predictable.
The same assertions hold for $u$ on the positive-time intervals on which it is bounded in $H^1$.

For $v\in H^2_\sigma$, set
\begin{equation*}
A(v):=\nu\Delta v-\mathcal P((v\cdot\nabla)v) -\mathcal P(g_N(|v|^2)v).
\end{equation*}
The Sobolev and Gagliardo--Nirenberg inequalities give
\begin{equation*}
\begin{aligned}
\|v\|_6&\le C\|v\|_{H^1}, & \|v\|_\infty&\le C\|v\|_{H^1}^{1/2}\|v\|_{H^2}^{1/2}, & \|A(v)\|_2^2&\le C\left(\|v\|_{H^2}^2+\|v\|_{H^1}^6\right).
\end{aligned}
\end{equation*}
The last inequality uses the $L^2$ boundedness of $\mathcal P$ and cubic taming growth.
Define, for $h/2<s\le\rho$,
\begin{equation*}
\begin{aligned}
F(s)&:=\chi(s)A(u(s))+\chi'(s)u(s), & B(s)&:=\chi(s)\Sigma(s,u(s-)), & K(s,z)&:=\chi(s)\mathcal G(s,u(s-),z),
\end{aligned}
\end{equation*}
and set these coefficients to zero elsewhere.
The drift is defined up to $\dd\P\otimes\dd s$ null sets, assigning zero when the state is outside $H^2$.
The drift map $A:H^2_\sigma\to L^2_\sigma$ is continuous.
The preceding measurability and bounds show that $F$ is progressively measurable and
\begin{equation}
\label{eq:4.4-coefficient-integrability}
\begin{aligned}
\int_0^T\|F(s)\|_2^2\,\dd s&<\infty \quad\P\text{-a.s.}, \\
\E\int_0^T\left[ \|B(s)\|_{\gamma(\mathcal U;H^1)}^2 +\int_Z\|K(s,z)\|_{H^1}^2\,\mu(\dd z) \right]\dd s&<\infty.
\end{aligned}
\end{equation}
The second assertion follows from \eqref{eq:2-raw-L2-growth}--\eqref{eq:2-H1}.
The noise coefficients are predictable, respectively predictable $\otimes\,\mathcal Z$-measurable, because $\mathds1_{(h/2,\rho]}$ is predictable.
Multiplying the stopped equation by $\chi(t\wedge\rho)$ gives, in distributions,
\begin{equation}
\label{eq:4.4-cutoff-equation}
V(t)=\int_0^t F(s)\,\dd s +\int_0^t B(s)\,\dd W_s +\int_0^t\int_Z K(s,z)\,\widetilde N(\dd s,\dd z).
\end{equation}

Set
\begin{equation*}
\Lambda(t):=\int_0^t\|F(s)\|_2^2\,\dd s, \qquad \theta_j:=\inf\{t\in[0,T]:\Lambda(t)\ge j\}\wedge\rho,
\end{equation*}
with $\inf\varnothing=\infty$.
The process $\Lambda$ is continuous and adapted.
Thus $\theta_j$ is a stopping time, $\Lambda(\theta_j)\le j$, and $\theta_j=\rho$ for all sufficiently large $j$, almost surely.
Write $V^j(t)=V(t\wedge\theta_j)$.
For $J_\alpha=e^{\alpha\Delta}$, $J_\alpha V^j$ satisfies the $H^1_\sigma$-valued semimartingale equation obtained by smoothing \eqref{eq:4.4-cutoff-equation} and restricting all integrals to $(0,\theta_j]$.

Let $D_{\alpha,\beta}=J_\alpha-J_\beta$.
The Hilbert-space It\^o formula, self-adjointness and commutation of the mollifiers with derivatives, and the BDG, Davis, and compensator bounds give
\begin{equation*}
\begin{aligned}
\E\sup_{t\le T} \|D_{\alpha,\beta}V^j(t)\|_{H^1}^2 & \le C\E\int_0^{\theta_j} \left|(F(s),(I-\Delta)D_{\alpha,\beta}^2V(s))_2\right| \,\dd s \\
&\qquad+ C\E\int_0^{\theta_j} \|D_{\alpha,\beta}B(s)\|_{\gamma(\mathcal U;H^1)}^2 \,\dd s + C\E\int_0^{\theta_j}\int_Z \|D_{\alpha,\beta}K(s,z)\|_{H^1}^2 \,\mu(\dd z)\,\dd s.
\end{aligned}
\end{equation*}
Here $D_{\alpha,\beta}^2$ denotes operator composition.
The initial term vanishes since $V^j(0)=0$.
For the quadratic jump martingale, the nonnegative integrand $\|D_{\alpha,\beta}K\|_{H^1}^2$ is again controlled by its uncompensated integral and compensator.

The first term on the right is bounded by
\begin{equation*}
Cj^{1/2}\left( \E\int_0^\rho \|(I-\Delta)D_{\alpha,\beta}^2V(s)\|_2^2\,\dd s \right)^{1/2}\longrightarrow0
\end{equation*}
as $\alpha,\beta\downarrow0$, by \eqref{eq:4.4-Hilbert-energy} and strong convergence of the uniformly bounded mollifiers on $H^2$.
The remaining terms tend to zero by \eqref{eq:4.4-coefficient-integrability}, strong convergence on $H^1$, and the radonifying ideal property.

Thus $(J_\alpha V^j)_{\alpha>0}$ is Cauchy in the expected square of the uniform $H^1$ norm.
A deterministic subsequence converges almost surely uniformly in $H^1$ to an adapted strongly c\`adl\`ag process.
Strong convergence of the mollifiers on $L^2$ is uniform on the relatively compact range of each $L^2$ c\`adl\`ag path, so this limit is $V^j$.
The convergence therefore holds in the expected square of the uniform $H^1$ norm, and in particular in probability for the Skorokhod $J_1$ topology.
Letting integer $j\to\infty$ gives a strongly $H^1$-c\`adl\`ag version of $V$.
Since $V=u$ on $[h,\rho]$ whenever $h\le\rho$, this proves the required positive-time path regularity on the localized interval.
All stopped noise integrals include their upper endpoints, so terminal Poisson jumps are retained.

\emph{Step 4:} For fixed $n,T$, local regularity gives $\mathcal Q_u(T\wedge\tau_n)<\infty$ almost surely.
Consequently, $\rho=T\wedge\tau_n$ for all sufficiently large integer $L$, pathwise.
Monotone convergence in \eqref{eq:4.4-weighted-energy}, first as $L\to\infty$ and then as $n\to\infty$, proves \eqref{eq:4.4-time-weighted-energy}, using $\tau_n<\tau_{\max}$ and $\tau_n\uparrow\tau_{\max}$.
The representatives from Step 3 agree with the original $L^2$ c\`adl\`ag representative and hence agree on overlaps.
Taking a countable intersection over $n,L$, integer $T$, and positive rational $h<T$ gives the asserted strong $H^1$-c\`adl\`ag regularity on $(0,\tau_{\max})$.

Finally, since $t\ge\varepsilon$ on the time sets in \eqref{eq:4.4-positive-time-energy}, division by $\varepsilon$ gives \eqref{eq:4.4-positive-time-energy}.
By (ii) of \cref{thm:2-Paper-I}, the maximal solution is locally $L^p_\sigma$-c\`adl\`ag on $[0,\tau_{\max})$, while Step~3 gives strong $H^1_\sigma$-c\`adl\`ag regularity on compact subintervals of $(0,\tau_{\max})$.
On $\{T<\tau_{\max}\}$, the corresponding left limits on $[\varepsilon,T]$ agree as distributions and hence coincide.
Thus $u$ is c\`adl\`ag in $L^p_\sigma\cap H^1_\sigma$ endowed with the sum norm, and \eqref{eq:4.4-positive-time-paths} follows.
\end{proof}
\section{Pressure control and weighted $L^p$ estimates}
\label{Section 5} Under the $H^1$-persistence hypotheses, we now estimate the intrinsic $L^p$--$L^{3p}$ quantities governing continuation.
The argument combines pressure control with exponentially weighted stochastic estimates to obtain bounds uniform in the intrinsic exit level.

The pressure associated with the projected drift has already been fixed in \eqref{eq:2-pressure-convention}.
Since the native $L^p$ test field $|v|^{p-2}v$ is not divergence free, the pressure term does not disappear from the $L^p$ energy balance.
We therefore first establish an $L^3$ bound for the pressure and then estimate its pairing with the native $L^p$ test field.
For later use, we write the pressure from \eqref{eq:2-pressure-convention} as
\begin{equation}
\label{eq:5.1-pressure-definition}
\begin{aligned}
\pi(v) &= \pi_{\mathrm{conv}}(v) + \pi_{\mathrm{tame}}(v), \\
\pi_{\mathrm{conv}}(v) := \sum_{i,j=1}^3(-\Delta)^{-1} \partial_i\partial_j(v_iv_j), & \qquad \pi_{\mathrm{tame}}(v) := (-\Delta)^{-1}\nabla\cdot\bigl(g_N(|v|^2)v\bigr),
\end{aligned}
\end{equation}
where the operators are understood through their standard singular-integral and Riesz-potential realizations.
For $p>3$, set $$J_p(v) := |v|^{p-2}v,$$ and, whenever $v\in H^2_\sigma\cap L^p$, set
\begin{equation}
\label{eq:5.1-drift-definition}
\begin{aligned}
A(v) := \nu\Delta v - \mathcal P((v\cdot\nabla)v) - \mathcal P(g_N(|v|^2)v), \qquad \mathcal D_p(v) := \int_{\R^3}|v|^{p-2}|\nabla v|^2\,\dd x.
\end{aligned}
\end{equation}
As the pressure contribution remains, in the following lemma we estimate the pressures generated by convection and taming, absorbing part of their contribution into the viscous dissipation and controlling the remainder through the $H^1$ norm.
\begin{lemma}
[Native $L^p$ pressure estimate] \label{lem:5.1-pressure} Let $p>3$ and let $g_N$ satisfy \eqref{eq:2-taming-local}--\eqref{eq:2-taming-coercive}.
If $v\in H^1_\sigma(\R^3)$, then $\pi(v)\in L^3(\R^3)$ and
\begin{equation}
\label{eq:5.1-pressure-L3}
\|\pi(v)\|_3 \le C\|v\|_6^2 + C_{C_\nu}\|v\|_{9/2}^3 \le C_{C_\nu}\bigl(\|v\|_{H^1}^2+\|v\|_{H^1}^3\bigr).
\end{equation}
If, in addition, $v\in L^p$ and $|v|^{p/2}\in H^1$, then $\nabla\cdot J_p(v)\in L^{3/2}$ and the pairing
\begin{equation}
\label{eq:5.1-pressure-pairing}
\langle\nabla\pi(v),J_p(v)\rangle := -\int_{\R^3} \pi(v)\nabla\cdot J_p(v)\,\dd x
\end{equation}
is well defined.
For every $\varepsilon>0$,
\begin{equation}
\label{eq:5.1-pressure-native}
p\bigl|\langle\nabla\pi(v),J_p(v)\rangle\bigr| \le \varepsilon \|\nabla\bigl(|v|^{p/2}\bigr)\|_2^2 + C_{\varepsilon,p} \|\pi(v)\|_3^{2p/3}\|v\|_p^{p/3},
\end{equation}
and consequently
\begin{equation}
\label{eq:5.1-pressure-H1}
p\bigl|\langle\nabla\pi(v),J_p(v)\rangle\bigr| \le \varepsilon \|\nabla\bigl(|v|^{p/2}\bigr)\|_2^2 + C_{\varepsilon,p,C_\nu} \bigl(1+\|v\|_{H^1}^{2p}\bigr) \bigl(1+\|v\|_p^p\bigr).
\end{equation}
If $v\in H^2_\sigma\cap L^p$, then
\begin{equation}
\label{eq:5.1-native-drift}
p(A(v),J_p(v))_2 \le -\frac{p\nu}{2}\mathcal D_p(v) - p\int_{\R^3} g_N(|v|^2)|v|^p\,\dd x + C_{p,\nu,C_\nu} \bigl(1+\|v\|_{H^1}^{2p}\bigr) \bigl(1+\|v\|_p^p\bigr),
\end{equation}
and
\begin{equation}
\label{eq:5.1-weighted-dissipation}
\|v\|_{3p}^p \le C\|\nabla\bigl(|v|^{p/2}\bigr)\|_2^2, \qquad \|\nabla\bigl(|v|^{p/2}\bigr)\|_2^2 \le \frac{p^2}{4}\mathcal D_p(v).
\end{equation}
Under the hypotheses of Theorem~\ref{thm:4.3-H1-persistence}, $\pi(u)$ has adapted $L^3$-c\`adl\`ag paths on $[0,\tau_{\max})$, the preceding estimates hold for Lebesgue-a.e.
$t<\tau_{\max}$ on a common event of probability one, and
\begin{equation}
\label{eq:5.1-pressure-operator}
p\bigl|\langle\nabla\pi(u(t)),J_p(u(t))\rangle\bigr| \le \varepsilon \|\nabla\bigl(|u(t)|^{p/2}\bigr)\|_2^2 + C_{\varepsilon,p,C_\nu} \Gamma'(t) \bigl(1+\|u(t)\|_p^p\bigr)
\end{equation}
for almost every $t<\tau_{\max}$.
\end{lemma}

\begin{proof}
\emph{Step 1} Let $\mathcal R_j=\partial_j(-\Delta)^{-1/2}$ and $I_1=(-\Delta)^{-1/2}$.
Since $v\in H^1$, Sobolev embedding gives $v\in L^6\cap L^{9/2}$.
The operators $\mathcal R_i\mathcal R_j$ are bounded on $L^3$, whereas the Hardy--Littlewood--Sobolev inequality gives $I_1:L^{3/2}(\R^3)\to L^3(\R^3)$.
Thus \eqref{eq:5.1-pressure-definition}, equivalently
\begin{equation*}
\pi_{\mathrm{conv}} =\sum_{i,j=1}^3\mathcal R_i\mathcal R_j(v_i v_j), \qquad \pi_{\mathrm{tame}} =I_1\sum_{j=1}^3\mathcal R_j(g_N(|v|^2)v_j),
\end{equation*}
defines $L^3$ functions satisfying
\begin{equation*}
\begin{aligned}
\|\pi_{\mathrm{conv}}\|_3&\le C\|v\|_6^2, \\
\|\pi_{\mathrm{tame}}\|_3 &\le C\|g_N(|v|^2)v\|_{3/2} \le C C_\nu\|v\|_{9/2}^3.
\end{aligned}
\end{equation*}
Interpolation gives $\|v\|_{9/2}\le\|v\|_2^{1/6}\|v\|_6^{5/6}$, which proves \eqref{eq:5.1-pressure-L3}.
The multiplier identities show that this decomposition realizes the pressure convention \eqref{eq:2-pressure-convention}; uniqueness in $L^3(\R^3)$ follows since an $L^3$ harmonic function on $\R^3$ vanishes.

Set $F(v)=(v\cdot\nabla)v+g_N(|v|^2)v$.
Both summands belong to $L^{3/2}$, and $\nabla\cdot(v\otimes v)=(v\cdot\nabla)v$ in distributions.
The Leray multiplier therefore gives
\begin{equation}
\label{eq:5.1-Leray-pressure}
\mathcal PF(v)=F(v)+\nabla\pi(v), \qquad A(v)=\nu\Delta v-F(v)-\nabla\pi(v).
\end{equation}
The second identity is used when $v\in H^2_\sigma\cap L^p$.

\emph{Step 2:} Write $f=|v|^{p/2}\in H^1$ and $q=3(p-2)$.
Sobolev embedding and interpolation yield
\begin{equation}
\label{eq:5.1-native-interpolation}
\|v\|_{3p}^p=\|f\|_6^2\le C\|\nabla f\|_2^2, \qquad \|v\|_q^{(p-2)/2} \le\|v\|_p^{1/2}\|v\|_{3p}^{(p-3)/2}.
\end{equation}
Indeed, $p<q<3p$ and
\begin{equation*}
\frac1{3(p-2)} =\frac1{p-2}\frac1p +\frac{p-3}{p-2}\frac1{3p}.
\end{equation*}

To justify the divergence formula without assuming $\mathcal D_p(v)<\infty$, let $J_{p,M}(v)=(|v|\wedge M)^{p-2}v$.
The truncated Sobolev chain rule and $\nabla\cdot v=0$ give
\begin{equation}
\nabla\cdot J_{p,M}(v) =\frac{2(p-2)}p\, \mathds 1_{\{0<|v|<M\}} |v|^{p/2-2}v\cdot\nabla f,
\end{equation}
with value zero on $\{v=0\}$.
Since $3(p-1)/2\in(p,3p)$, $J_{p,M}(v)\to J_p(v)$ in $L^{3/2}$.
Moreover,
\begin{equation}
\left\||v|^{(p-2)/2}|\nabla f|\right\|_{3/2} \le\|v\|_q^{(p-2)/2}\|\nabla f\|_2<\infty.
\end{equation}
Dominated convergence in $L^{3/2}$ identifies the weak divergence of $J_p(v)$ and gives
\begin{equation}
\label{eq:5.1-divergence-identity}
\nabla\cdot J_p(v) =\frac{2(p-2)}p|v|^{p/2-2}v\cdot\nabla f \quad\text{in }L^{3/2},
\end{equation}
again with value zero on $\{v=0\}$.
Thus \eqref{eq:5.1-pressure-pairing} is finite.
Spatial cutoffs followed by mollification approximate $J_p(v)$ and its divergence in $L^{3/2}$.
So this definition is the corresponding continuous extension of distributional pairing.

By H\"older's inequality and \eqref{eq:5.1-native-interpolation}--\eqref{eq:5.1-divergence-identity},
\begin{equation*}
\begin{aligned}
p\bigl|\langle\nabla\pi(v),J_p(v)\rangle\bigr| &\le C_p\|\pi(v)\|_3 \|v\|_q^{(p-2)/2}\|\nabla f\|_2 \\
&\le C_p\|\pi(v)\|_3\|v\|_p^{1/2} \bigl(\|\nabla f\|_2^2\bigr)^{(2p-3)/(2p)}.
\end{aligned}
\end{equation*}
Young's inequality with conjugate exponents $2p/(2p-3)$ and $2p/3$ proves \eqref{eq:5.1-pressure-native}.
Finally, \eqref{eq:5.1-pressure-L3} implies
\begin{equation*}
\|\pi(v)\|_3^{2p/3} \le C_{p,C_\nu} \left(\|v\|_{H^1}^{4p/3}+\|v\|_{H^1}^{2p}\right) \le C_{p,C_\nu}(1+\|v\|_{H^1}^{2p}),
\end{equation*}
and $\|v\|_p^{p/3}\le1+\|v\|_p^p$ gives \eqref{eq:5.1-pressure-H1}.

\emph{Step 3:} Suppose $v\in H^2_\sigma\cap L^p$.
Since $H^2(\R^3)\hookrightarrow L^\infty$,
\begin{equation*}
J_p(v)\in H^1\cap L^{3/2}, \qquad F(v),A(v),\nabla\pi(v)\in L^2, \qquad \mathcal D_p(v)<\infty.
\end{equation*}
Here $\|g_N(|v|^2)v\|_2\le C_\nu\|v\|_6^3$ and $\|(v\cdot\nabla)v\|_2\le\|v\|_\infty\|\nabla v\|_2$.

Choose $\chi\in C_c^\infty(\R^3)$ equal to one on $B_1$, and set $\chi_\ell(x)=\chi(x/\ell)$.
Integration by parts with $\chi_\ell J_p(v)$ is justified by spatial mollification.
The pressure boundary term satisfies
\begin{equation*}
\left|\int_{\R^3} \pi(v)J_p(v)\cdot\nabla\chi_\ell\,\dd x\right| \le\frac C\ell\|\pi(v)\|_3\|J_p(v)\|_{3/2} \longrightarrow0.
\end{equation*}
Consequently, \eqref{eq:5.1-pressure-pairing} agrees with $(\nabla\pi(v),J_p(v))_2$.
Similarly, $|v|^p v\in L^1$ and $\nabla\cdot v=0$ imply
\begin{equation*}
\begin{aligned}
\int_{\R^3}J_p(v)\cdot(v\cdot\nabla)v\,\dd x &=\lim_{\ell\to\infty} \left(-\frac1p\int_{\R^3} |v|^p v\cdot\nabla\chi_\ell\,\dd x\right)=0.
\end{aligned}
\end{equation*}

The Sobolev chain rule gives
\begin{equation}
\label{eq:5.1-viscous-identity}
\begin{aligned}
p\nu(\Delta v,J_p(v))_2 &=-p\nu\mathcal D_p(v) -\frac{4\nu(p-2)}p \left\|\nabla\bigl(|v|^{p/2}\bigr)\right\|_2^2, \\
\left\|\nabla\bigl(|v|^{p/2}\bigr)\right\|_2^2 &\le\frac{p^2}{4}\mathcal D_p(v).
\end{aligned}
\end{equation}
Combining the unprojected convective cancellation with \eqref{eq:5.1-Leray-pressure} and \eqref{eq:5.1-viscous-identity}, we obtain
\begin{equation*}
p(A(v),J_p(v))_2 =-p\nu\mathcal D_p(v) -\frac{4\nu(p-2)}p \left\|\nabla\bigl(|v|^{p/2}\bigr)\right\|_2^2 -p\int_{\R^3}g_N(|v|^2)|v|^p\,\dd x -p\langle\nabla\pi(v),J_p(v)\rangle.
\end{equation*}
Apply \eqref{eq:5.1-pressure-H1} with $\varepsilon=2\nu/p$ and use \eqref{eq:5.1-viscous-identity} to absorb its dissipation term into $p\nu\mathcal D_p(v)/2$.
This proves \eqref{eq:5.1-native-drift}.
\eqref{eq:5.1-weighted-dissipation} follows from \eqref{eq:5.1-native-interpolation} and \eqref{eq:5.1-viscous-identity}.

The maps $v\mapsto v_i v_j$ from $H^1$ to $L^3$ and $v\mapsto g_N(|v|^2)v$ from $H^1$ to $L^{3/2}$ are continuous.
For the latter, use $H^1\hookrightarrow L^{9/2}$ and continuity of the Nemytskii map under the cubic growth bound.
The preceding operator estimates therefore show that $v\mapsto\pi(v)$ is continuous from $H^1_\sigma$ to $L^3$.
Theorem \ref{thm:4.3-H1-persistence} consequently gives the adapted $L^3$-c\`adl\`ag version of $\pi(u)$.
A countable exhaustion by the stopped intervals of that theorem gives, on a common event of probability one, $u(t)\in H^2_\sigma\cap L^p$ for almost every $t<\tau_{\max}$.
All energy inequalities therefore apply at these times.
Finally,
\begin{equation*}
1+\|u(t)\|_{H^1}^{2p}\le\Gamma'(t) \quad\text{for almost every }t<\tau_{\max},
\end{equation*}
which proves \eqref{eq:5.1-pressure-operator}.
\end{proof}
The next lemma supplies the stochastic estimates needed for the exponentially weighted $L^p$ energy argument.
It controls the jump remainder using only the prescribed second and $p$th moments and establishes convergence of the Wiener and Poisson terms under spatial mollification.

The coefficient estimates needed below have already been recorded in \eqref{eq:2-P1}--\eqref{eq:2-P2}.
In particular, no additional growth assumption on the Wiener or jump coefficients is imposed in this section.
We only collect the consequences of these estimates for the native $L^p$ energy functional and for the exponentially weighted stochastic terms that will enter the stopped $L^p$ estimate.
Set
\begin{equation}
\label{eq:5.2-notation}
\Phi(v) := \|v\|_p^p ,
\end{equation}
and, for $v,h\in L^p$ and $B\in\gamma_p$, write
\begin{equation}
\label{eq:5.2-corrections}
\begin{aligned}
\mathcal R_p(v,h) &:= \Phi(v+h)-\Phi(v)-\Phi'(v)[h],& \mathcal C_p(v,B) &:= \frac12\sum_j\Phi''(v)[Be_j,Be_j],
\end{aligned}
\end{equation}
where $(e_j)$ is an orthonormal basis of $\mathcal U$; the well-definedness and basis independence of $\mathcal C_p$ are part of the lemma below.

For the stochastic estimates, fix $n\ge1$, $T>0$, $K>0$, and a stopping time $\rho$ satisfying
\begin{equation}
\label{eq:5.2-localization}
0\le\rho\le T\wedge\tau_n, \qquad \Gamma(\rho)\le K \quad\P\text{-a.s.}
\end{equation}
For $c\ge0$, set
\begin{equation}
\label{eq:5.2-weight}
U(t) := u(t\wedge\rho), \qquad \vartheta(t) := e^{-c\Gamma(t\wedge\rho)}, \qquad Y_\rho := \sup_{t\le\rho}\vartheta(t)\Phi(u(t)).
\end{equation}
On $(0,\rho]$, let
\begin{equation*}
B_s := \Sigma(s,U(s-)), \qquad H_s(z) := \mathcal G(s,U(s-),z),
\end{equation*}
and extend these processes by zero outside $(0,\rho]$.
Define
\begin{equation}
\label{eq:5.2-martingales}
\begin{aligned}
M^W(t) &:= \int_0^t\vartheta(s) \left\langle B_s^*\Phi'(U(s-)),\dd W_s\right\rangle_{\mathcal U}, \\
M^{J,1}(t) &:= \int_0^t\int_Z \vartheta(s)\Phi'(U(s-))[H_s(z)] \, \widetilde N(\dd s,\dd z), \\
M^{J,2}(t) &:= \int_0^t\int_Z \vartheta(s)\mathcal R_p(U(s-),H_s(z)) \, \widetilde N(\dd s,\dd z),
\end{aligned}
\end{equation}
and
\begin{equation}
\label{eq:5.2-compensators}
\begin{aligned}
Q^W(t) &:= \int_0^t\vartheta(s) \mathcal C_p(U(s-),B_s)\, \dd s,& Q^J(t) &:= \int_0^t\int_Z \vartheta(s)\mathcal R_p(U(s-),H_s(z)) \, \mu(\dd z)\, \dd s.
\end{aligned}
\end{equation}
Finally, for the spatial regularization used below, let $S_m = \mathrm{P}_{\le m}$ be the smoothing operators of \eqref{eq:2-smoothing} and set
\begin{equation}
\label{eq:5.2-mollification}
U_m := S_mU,\qquad B_{m,s} := S_mB_s,\qquad H_{m,s}(z) := S_mH_s(z).
\end{equation}
The weight $\vartheta$ is not mollified.

\begin{lemma}
[Noise and jump remainder estimates] \label{lem:5.2-noise-jumps} Let $p>3$ and assume the hypotheses of Theorem~\ref{thm:4.3-H1-persistence}.

\emph{(i) Pointwise estimates.} The correction $\mathcal C_p$ in \eqref{eq:5.2-corrections} is well defined, nonnegative, and independent of the chosen orthonormal basis.
Moreover,
\begin{equation}
\label{eq:5.2-Taylor-bound}
0\le\mathcal R_p(v,h) \le C_p\bigl( \|v\|_p^{p-2}\|h\|_p^2+\|h\|_p^p \bigr), \qquad v,h\in L^p.
\end{equation}
For $v\in L^p_\sigma\cap L^\infty$,
\begin{equation}
\label{eq:5.2-Wiener-pointwise}
\begin{aligned}
\mathcal C_p(v,\Sigma(t,v)) &\le C(1+\|v\|_\infty)(1+\Phi(v)),& \|\Sigma(t,v)^*\Phi'(v)\|_{\mathcal U}^2 &\le C\Phi(v)(1+\|v\|_\infty)(1+\Phi(v)),
\end{aligned}
\end{equation}
while, for every $v\in L^p_\sigma$,
\begin{equation}
\label{eq:5.2-jump-pointwise}
\begin{aligned}
\int_Z\mathcal R_p(v,\mathcal G(t,v,z))\, \mu(\dd z) &\le C(1+\Phi(v)), & \int_Z \bigl|\Phi'(v)[\mathcal G(t,v,z)]\bigr|^2\, \mu(\dd z) &\le C\Phi(v)(1+\Phi(v)).
\end{aligned}
\end{equation}

\emph{(ii) Weighted stochastic estimates.} Under \eqref{eq:5.2-localization}, the processes in \eqref{eq:5.2-martingales} are well defined adapted c\`adl\`ag martingales with integrable suprema, where $M^{J,2}$ is understood as its integrable uncompensated integral minus the compensator.
For every $\delta>0$,
\begin{equation}
\label{eq:5.2-weighted-martingale-bounds}
\begin{aligned}
\E\sup_{t\le T}|M^W(t)| &\le \delta\E Y_\rho +C_\delta\E\int_0^\rho \vartheta(s)(1+\Phi(u(s-)))\, \dd\Gamma(s), \\
\E\sup_{t\le T}|M^{J,1}(t)| &\le \delta\E Y_\rho +C_\delta\E\int_0^\rho \vartheta(s)(1+\Phi(u(s-)))\, \dd s, \\
\E\sup_{t\le T}|M^{J,2}(t)| &\le 2\E Q^J(T) \le C\E\int_0^\rho \vartheta(s)(1+\Phi(u(s-)))\, \dd s,
\end{aligned}
\end{equation}
and
\begin{equation}
\label{eq:5.2-weighted-correction-bounds}
\begin{aligned}
\E Q^W(T) &\le C\E\int_0^\rho \vartheta(s)(1+\Phi(u(s-)))\, \dd\Gamma(s),& \E Q^J(T) &\le C\E\int_0^\rho \vartheta(s)(1+\Phi(u(s-)))\, \dd s.
\end{aligned}
\end{equation}

\emph{(iii) Compatibility with spatial mollification.} Define $M_m^W,M_m^{J,1},M_m^{J,2},Q_m^W,Q_m^J$ from \eqref{eq:5.2-martingales}--\eqref{eq:5.2-compensators} by replacing $(U,B,H)$ with $(U_m,B_m,H_m)$.
The estimates in \emph{(ii)} hold with constants uniform in $m$, and
\begin{equation}
\label{eq:5.2-mollified-convergence}
\begin{aligned}
\E\sup_{t\le T}\|U_m(t)-U(t)\|_p^p &\longrightarrow0, \\
\E\sup_{t\le T} \left( |M_m^W(t)-M^W(t)| +\sum_{j=1}^2|M_m^{J,j}(t)-M^{J,j}(t)| \right) &\longrightarrow0, \\
\E\sup_{t\le T} \left( |Q_m^W(t)-Q^W(t)| +|Q_m^J(t)-Q^J(t)| \right) &\longrightarrow0.
\end{aligned}
\end{equation}
The constants above depend only on the structural coefficient constants, on $p$, and, where applicable, on $\delta$, and are independent of the localization and regularization parameters.
\end{lemma}

\begin{proof}
\emph{Step 1:} The functional $\Phi$ belongs to $C^2(L^p;\R)$, with
\begin{equation}
\label{eq:5.2-derivatives}
\begin{aligned}
\Phi'(v)[h]&=p\int_{\R^3}|v|^{p-2}v\cdot h\,\dd x, \\
\Phi''(v)[h,k] &=p\int_{\R^3}|v|^{p-2}h\cdot k\,\dd x +p(p-2)\int_{\R^3} |v|^{p-4}(v\cdot h)(v\cdot k)\,\dd x.
\end{aligned}
\end{equation}
The second integrand is defined as zero on $\{v=0\}$.
H\"older's inequality gives
\begin{equation}
\label{eq:5.2-derivative-bounds}
\begin{aligned}
\|\Phi'(v)\|_{(L^p)^*}&=p\|v\|_p^{p-1}, & 0\le\Phi''(v)[h,h]&\le p(p-1)\|v\|_p^{p-2}\|h\|_p^2.
\end{aligned}
\end{equation}
Convexity and the integral Taylor formula
\begin{equation*}
\mathcal R_p(v,h) =\int_0^1(1-r)\Phi''(v+rh)[h,h]\,\dd r
\end{equation*}
prove \eqref{eq:5.2-Taylor-bound}.

The $\gamma$-Fubini identification \cite{MR2330977,MR3617205}, followed by Tonelli's and H\"older's inequalities, yields
\begin{equation}
\label{eq:5.2-trace-bound}
\begin{aligned}
0\le\mathcal C_p(v,B) &\le C_p\|v\|_p^{p-2}\|B\|_{\gamma_p}^2, & \|B^*\Phi'(v)\|_{\mathcal U}^2 &\le p^2\|v\|_p^{2p-2}\|B\|_{\gamma_p}^2.
\end{aligned}
\end{equation}
In particular, the trace series converges and is basis independent, as is also immediate from the pointwise Hilbert--Schmidt representation.
By \eqref{eq:2-P1}, for $v\in L^p_\sigma\cap L^\infty$, $ \|\Sigma(t,v)\|_{\gamma_p}^2 \le C\bigl(1+\|v\|_{3p/2}^3\bigr).
$ Interpolation between $L^p$ and $L^\infty$ gives $\|v\|_{3p/2}^3 \le \|v\|_p^2\|v\|_\infty,$ and hence $ \|\Sigma(t,v)\|_{\gamma_p}^2 \le C\bigl(1+\|v\|_p^2\|v\|_\infty\bigr).
$ Combining this with \eqref{eq:5.2-trace-bound} proves \eqref{eq:5.2-Wiener-pointwise}.

For the jump terms, \eqref{eq:5.2-Taylor-bound} and \eqref{eq:2-P2} with $r=2,p$ give
\begin{equation*}
\int_Z \mathcal R_p(v,\mathcal G(t,v,z))\,\mu(\dd z) \le C\Big[ \|v\|_p^{p-2}(1+\|v\|_p^2) +1+\|v\|_p^p \Big]\le C(1+\Phi(v)).
\end{equation*}
Likewise, by \eqref{eq:5.2-derivative-bounds} and \eqref{eq:2-P2} with $r=2$,
\begin{equation*}
\int_Z \bigl|\Phi'(v)[\mathcal G(t,v,z)]\bigr|^2 \,\mu(\dd z) \le C\|v\|_p^{2p-2}(1+\|v\|_p^2)\le C\Phi(v)(1+\Phi(v)).
\end{equation*}
This proves \eqref{eq:5.2-jump-pointwise}.

\emph{Step 2:} The stopped process $U$ is adapted and $L^p$ c\`adl\`ag, with
\begin{equation}
\label{eq:5.2-local-moment}
\E\sup_{t\le T}\|U(t)\|_p^p<\infty
\end{equation}
by the local-energy class of (ii) of \cref{thm:2-Paper-I} and $\rho\le\tau_n$.
The process $U_-$ and the indicator $\mathds 1_{(0,\rho]}$ are predictable.
The weight $\vartheta$ is continuous, adapted, and bounded by one, hence predictable.
Thus all integrands in \eqref{eq:5.2-martingales}--\eqref{eq:5.2-compensators} have the required predictable, or predictable-product, measurability.

Theorem \ref{thm:4.3-H1-persistence} gives $u\in L^2(0,\rho;H^2)$ almost surely.
Since $u=u_-$ almost everywhere in time, the Wiener estimates apply for almost every $s\le\rho$, and
\begin{equation}
\label{eq:5.2-control}
1+\|u(s-)\|_\infty\le\Gamma'(s) \quad\text{for almost every }s\le\rho.
\end{equation}
In particular, \eqref{eq:5.2-localization} gives
\begin{equation}
\int_0^\rho\|B_s\|_{\gamma_p}^2\,\dd s \le C\left(T+K\sup_{t\le T}\|U(t)\|_p^2\right).
\end{equation}
Together with the jump growth bounds and \eqref{eq:5.2-local-moment}, this proves
\begin{equation}
\label{eq:5.2-integrand-spaces}
\begin{aligned}
\E\left(\int_0^T\|B_s\|_{\gamma_p}^2\,\dd s\right)^{p/2} &<\infty, \\
\E\left(\int_0^T\int_Z \|H_s(z)\|_p^2\,\mu(\dd z)\,\dd s\right)^{p/2} +\E\int_0^T\int_Z\|H_s(z)\|_p^p\,\mu(\dd z)\,\dd s &<\infty.
\end{aligned}
\end{equation}
These are admissible Wiener and Poisson integrand spaces \cite{MR2330977,MR3265175}.
The scalar linear integrands in \eqref{eq:5.2-martingales} are locally square integrable: their squared compensators are finite almost surely by \eqref{eq:5.2-integrand-spaces} and the pathwise boundedness of $U$.
The remainder integrand is nonnegative and integrable with respect to $\P\otimes\dd s\otimes\mu$, by \eqref{eq:5.2-jump-pointwise}.
It therefore defines the stated difference of an uncompensated integral and its compensator.

The right-hand sides in (ii) are finite, since
\begin{equation*}
\begin{aligned}
\int_0^\rho\vartheta(s)(1+\Phi(u(s-)))\,\dd\Gamma(s) &\le K\left(1+\sup_{t\le T}\Phi(U(t))\right), & \int_0^\rho\vartheta(s)(1+\Phi(u(s-)))\,\dd s &\le T\left(1+\sup_{t\le T}\Phi(U(t))\right).
\end{aligned}
\end{equation*}

\emph{Step 3:} Continuity of $\vartheta$ and the left limits of $U$ imply
\begin{equation*}
\vartheta(s)\Phi(U(s-))\le Y_\rho, \qquad 0<s\le\rho.
\end{equation*}
Define the nonnegative random variables
\begin{equation*}
\begin{aligned}
V_W&:=\int_0^\rho\vartheta(s) \|U(s-)\|_p^{p-2}\|B_s\|_{\gamma_p}^2\,\dd s, & V_J&:=\int_0^\rho\int_Z\vartheta(s) \|U(s-)\|_p^{p-2}\|H_s(z)\|_p^2\,N(\dd s,\dd z).
\end{aligned}
\end{equation*}
The growth bounds and the compensator identity \cite{MR1464694} give
\begin{equation}
\label{eq:5.2-majorant-moments}
\begin{aligned}
\E V_W &\le C\E\int_0^\rho \vartheta(s)(1+\Phi(u(s-)))\,\dd\Gamma(s), & \E V_J &\le C\E\int_0^\rho \vartheta(s)(1+\Phi(u(s-)))\,\dd s.
\end{aligned}
\end{equation}
In particular, both variables are integrable.
Their role is determined by the quadratic variations:
\begin{equation*}
[M^W]_T\le C_pY_\rho V_W, \qquad [M^{J,1}]_T\le C_pY_\rho V_J.
\end{equation*}
For the jump martingale the bracket is the integral of the squared scalar integrand against $N$.
The Wiener BDG and Davis inequalities, followed by Young's inequality, yield
\begin{equation*}
\begin{aligned}
\E\sup_{t\le T}|M^W(t)| &\le C_p\E(Y_\rho V_W)^{1/2} \le\delta\E Y_\rho+C_\delta\E V_W, \\
\E\sup_{t\le T}|M^{J,1}(t)| &\le C_p\E(Y_\rho V_J)^{1/2} \le\delta\E Y_\rho+C_\delta\E V_J.
\end{aligned}
\end{equation*}
These estimates are first applied after scalar martingale localization and then passed to the limit by Fatou's lemma.
They prove the first two bounds in \eqref{eq:5.2-weighted-martingale-bounds} and give martingales with integrable suprema.

For $r(s,z):=\vartheta(s)\mathcal R_p(U(s-),H_s(z))\ge0$,
\begin{equation*}
\sup_{t\le T}|M^{J,2}(t)| \le\int_0^T\int_Z r(s,z)\,N(\dd s,\dd z) +\int_0^T\int_Z r(s,z)\,\mu(\dd z)\,\dd s.
\end{equation*}
Taking expectations proves $\E\sup_{t\le T}|M^{J,2}(t)|\le2\E Q^J(T)$.
The remaining assertions in \eqref{eq:5.2-weighted-martingale-bounds}--\eqref{eq:5.2-weighted-correction-bounds} follow from the pointwise estimates and \eqref{eq:5.2-control}.
The support $(0,\rho]$ preserves a possible terminal jump, and every coefficient is evaluated at the predictable state $U(s-)$.

\emph{Step 4:} By \eqref{eq:2-smoothing-properties} and since the range of every $L^p$ c\`adl\`ag path on $[0,T]$ is relatively compact , so
\begin{equation*}
\sup_{t\le T}\|S_mU(t)-U(t)\|_p\longrightarrow0 \quad\P\text{-a.s.}
\end{equation*}
Dominated convergence using \eqref{eq:5.2-local-moment} proves the first limit in \eqref{eq:5.2-mollified-convergence}.

Moreover, $ U_m(s-) = S_mU(s-), $ and, by \eqref{eq:2-smoothing-properties} and \eqref{eq:2-gamma-notation}, $ B_{m,s} \longrightarrow B_s \quad\text{in } \gamma_p, \ H_{m,s}(z) \longrightarrow H_s(z) \ \text{in } L^p.
$ Using \eqref{eq:2-smoothing-properties} together with \eqref{eq:5.2-trace-bound} and \eqref{eq:5.2-Taylor-bound},
\begin{equation*}
\begin{aligned}
\mathcal C_p(U_m(s-),B_{m,s}) &\le C_p\|U(s-)\|_p^{p-2}\|B_s\|_{\gamma_p}^2, \\
\mathcal R_p(U_m(s-),H_{m,s}(z)) &\le C_p\left( \|U(s-)\|_p^{p-2}\|H_s(z)\|_p^2 +\|H_s(z)\|_p^p \right).
\end{aligned}
\end{equation*}

These majorants are integrable after multiplication by $\vartheta$, by Steps 2--3.
The maps $(v,B)\mapsto\mathcal C_p(v,B)$ and $(v,h)\mapsto\mathcal R_p(v,h)$ are continuous; for the first map use continuity of $\Phi''$ and the Gaussian representation of the trace.
Dominated convergence therefore gives
\begin{equation*}
\begin{aligned}
\E\int_0^T\vartheta(s) \bigl|\mathcal C_p(U_m(s-),B_{m,s}) -\mathcal C_p(U(s-),B_s)\bigr|\,\dd s&\longrightarrow0, \\
\E\int_0^T\int_Z\vartheta(s) \bigl|\mathcal R_p(U_m(s-),H_{m,s}(z)) -\mathcal R_p(U(s-),H_s(z))\bigr| \,\mu(\dd z)\,\dd s&\longrightarrow0.
\end{aligned}
\end{equation*}
This proves the convergence of both compensators.
The same $L^1$ compensator estimate for a signed integrand proves the asserted convergence of $M_m^{J,2}$.

For the two linear martingales, set
\begin{equation*}
\begin{aligned}
d_m^W(s) &:=\vartheta(s)\Big[ B_{m,s}^*\Phi'(U_m(s-)) -B_s^*\Phi'(U(s-)) \Big], \\
d_m^J(s,z) &:=\vartheta(s)\Big[ \Phi'(U_m(s-))[H_{m,s}(z)] -\Phi'(U(s-))[H_s(z)] \Big].
\end{aligned}
\end{equation*}
The preceding strong convergences and continuity of $\Phi'$ imply $d_m^W(s)\to0$ for $\dd\P\otimes\dd s$-almost every point and $d_m^J(s,z)\to0$ for $\dd\P\otimes\dd s\otimes\mu(\dd z)$-almost every point.
Moreover,
\begin{equation*}
\begin{aligned}
\ [M_m^W-M^W]_T &=\int_0^\rho\|d_m^W(s)\|_{\mathcal U}^2\,\dd s, & [M_m^{J,1}-M^{J,1}]_T &=\int_0^\rho\int_Z|d_m^J(s,z)|^2\,N(\dd s,\dd z).
\end{aligned}
\end{equation*}
By \eqref{eq:2-smoothing-properties} and \eqref{eq:5.2-derivative-bounds}, the corresponding integrands admit the same majorants as in Step~3.
Consequently, $$ [M_m^W-M^W]_T\le C_pY_\rho V_W, \qquad [M_m^{J,1}-M^{J,1}]_T\le C_pY_\rho V_J.
$$ Dominated convergence with respect to $\dd s$ and $N(\dd s,\dd z)$ therefore yields convergence of both quadratic variations to zero almost surely.
Furthermore, $$ \E(Y_\rho V_W)^{1/2} +\E(Y_\rho V_J)^{1/2}<\infty $$ by \eqref{eq:5.2-local-moment}, \eqref{eq:5.2-majorant-moments}, and Cauchy--Schwarz.
Hence dominated convergence, followed by the Wiener BDG and Davis inequalities, gives $$ \E\sup_{t\le T}|M_m^W(t)-M^W(t)| + \E\sup_{t\le T}|M_m^{J,1}(t)-M^{J,1}(t)| \longrightarrow0.
$$ proving the remaining limits in \eqref{eq:5.2-mollified-convergence}.The same majorants yield the uniformity in $m$, completing the proof.
\end{proof}

The preceding estimates are now combined with an exponential weight determined by $\Gamma$.
At each fixed integral level, the resulting intrinsic energy bound is uniform in the exit level and yields the probability estimate needed for the global continuation argument.

\begin{proposition}
[Uniform stopped native $L^p$ estimate] \label{prop:5.3-stopped-Lp} Let $p>3$ and retain the hypotheses of Theorem \ref{thm:4.3-H1-persistence} and Lemma \ref{lem:5.2-noise-jumps}.
Let $\Gamma$ be defined by \eqref{eq:4.3-accumalted-regularity}, and let $\lambda_L$ be the intrinsic exits of Theorem \ref{thm:3.6-intrinsic-blowup}.
For $K,T>0$, $L\ge1$, and $n\in\mathbb N$, set
\begin{equation*}
\begin{aligned}
\eta_K :=\inf\{t\in[0,\tau_{\max}):&\Gamma(t)\ge K\} \wedge\tau_{\max}, \\
\rho_{n,L,K,T}:=T\wedge\tau_n\wedge\lambda_L\wedge\eta_K, \quad & \rho_{L,K,T}:=T\wedge\lambda_L\wedge\eta_K,
\end{aligned}
\end{equation*}
with $\inf\varnothing=\infty$.
Then $\eta_K$ is a stopping time, $\eta_K\uparrow\tau_{\max}$ almost surely, and $\rho_{L,K,T}<\tau_{\max}$ almost surely.

Write $\Phi(v)=\|v\|_p^p$, retain $\mathcal D_p$ from \eqref{eq:5.1-drift-definition}, and define
\begin{equation}
\label{eq:5.3-taming-dissipation}
\mathcal T_p(v):=\int_{\R^3} g_N(|v|^2)|v|^p\,\dd x.
\end{equation}
there exists deterministic constant $c\ge1$ and $C<\infty$, depending only on $p,\nu$ and the coefficient constants, such that the following holds for either $\rho=\rho_{n,L,K,T}$ or $\rho=\rho_{L,K,T}$.
With $\vartheta(t)=e^{-c\Gamma(t\wedge\rho)}$,
\begin{equation}
\label{eq:5.3-weighted-estimate}
\begin{aligned}
\E\bigg[ &\sup_{0\le t\le\rho}\vartheta(t)\Phi(u(t)) +\frac{p\nu}{2}\int_0^\rho \vartheta(s)\mathcal D_p(u(s))\,\dd s \\
&+p\int_0^\rho\vartheta(s)\mathcal T_p(u(s))\,\dd s +\int_0^\rho\vartheta(s)\Phi(u(s))\,\dd\Gamma(s) \bigg] \le C\left(1+\E\|u_0\|_p^p\right).
\end{aligned}
\end{equation}
In particular,
\begin{equation}
\label{eq:5.3-unweighted-estimate}
\E\bigg[ \mathcal Q_u(\rho) +\int_0^\rho \left\|\nabla\bigl(|u(s)|^{p/2}\bigr)\right\|_2^2\,\dd s +\int_0^\rho\mathcal T_p(u(s))\,\dd s \bigg] \le Ce^{cK}\left(1+\E\|u_0\|_p^p\right),
\end{equation}
and
\begin{equation}
\label{eq:5.3-exit-probability}
\P(\lambda_L\le T\wedge\eta_K) \le\frac{Ce^{cK}}{L} \left(1+\E\|u_0\|_p^p\right).
\end{equation}
The constants $c,C$ are independent of $n,L,K,T$ and all auxiliary scalar and spatial regularization parameters.
All suprema include their terminal values.
\end{proposition}

\begin{proof}
\emph{Step 1:} Extend the integrand defining $\Gamma$ by zero outside $[0,\tau_{\max})$ and denote its time integral by $\overline\Gamma$.
By \eqref{eq:4.3-regularity}, this defines an adapted, continuous, increasing process, finite on every bounded time interval.
Moreover,
\begin{equation*}
\eta_K=\inf\{t\ge0:\overline\Gamma(t)\ge K\} \wedge\tau_{\max}.
\end{equation*}
Thus $\eta_K$ is a stopping time.
Monotonicity and local finiteness of $\Gamma$ give $\eta_K\uparrow\tau_{\max}$.
Theorem \ref{thm:3.6-intrinsic-blowup} gives $\lambda_L<\tau_{\max}$ on $\{\tau_{\max}<\infty\}$.
On $\{\tau_{\max}=\infty\}$, the bound $\rho_{L,K,T}\le T$ gives the asserted strict inequality.

Fix $n,L,K,T$ and initially write $\rho=\rho_{n,L,K,T}$.
Then $\Gamma(\rho)\le K$.
Set $U(t)=u(t\wedge\rho)$ and use the weight $\vartheta$ and stopped noise integrands $B,H$ of Lemma \ref{lem:5.2-noise-jumps}.
The value of $c\ge1$ will be fixed in Step 4.
The weight is continuous and predictable, and $\mathds 1_{(0,\rho]}$ is predictable.
The local-energy class and Theorem \ref{thm:4.3-H1-persistence} give
\begin{equation}
\label{eq:5.3-preliminary-integrability}
\begin{aligned}
\E\bigg[ \sup_{t\le T}\Phi(U(t)) +\int_0^\rho \left\|\nabla\bigl(|u(s)|^{p/2}\bigr)\right\|_2^2\,\dd s \bigg]&<\infty, \\
U\in\mathbb D([0,T];H^1_\sigma), \qquad u\in L^2(0,\rho;H^2_\sigma) &\quad\P\text{-a.s.}
\end{aligned}
\end{equation}
In particular,
\begin{equation}
\label{eq:5.3-integrability}
\E\int_0^\rho \vartheta(s)(1+\Phi(u(s)))\,\dd\Gamma(s) \le K\left(1+\E\sup_{t\le T}\Phi(U(t))\right) <\infty.
\end{equation}
These preliminary finiteness statements need not be uniform in $n$.

Let $A$ be the drift operator in \eqref{eq:5.1-drift-definition} and write
\begin{equation}
\label{eq:5.3-drift}
A_u(s):=A(u(s))=\nu\Delta u(s) -\mathcal P((u(s)\cdot\nabla)u(s)) -\mathcal P(g_N(|u(s)|^2)u(s)),\qquad s<\rho.
\end{equation}
The estimate \eqref{eq:4.3-drift-integrability} yields $A_u\in L^2(0,\rho;L^2_\sigma)$ almost surely.

\emph{Step 2:} For $\ell\ge1$ and $r\ge0$, define
\begin{equation*}
\begin{aligned}
\beta_\ell(r) &:=r^{p-2}(1+r^2/\ell^2)^{-(p-2)/2}, & \phi_\ell(x)&:=p\int_0^{|x|}\beta_\ell(r)r\,\dd r, \\
\Phi_\ell(v) &:=\int_{\R^3}\phi_\ell(v(x))\,\dd x, & J_\ell(v)&:=\beta_\ell(|v|)v.
\end{aligned}
\end{equation*}
The functional $\Phi_\ell$ belongs to $C^2(L^p;\R)$, $\Phi_\ell'(v)=pJ_\ell(v)$, and
\begin{equation}
\label{eq:5.3-truncation-bounds}
\begin{aligned}
0\le\phi_\ell(x)&\le |x|^p, &|\nabla\phi_\ell(x)|&\le p|x|^{p-1}, \\
0\le D^2\phi_\ell(x)[h,h] &\le p(p-1)|x|^{p-2}|h|^2, &\|D^2\phi_\ell(x)\|&\le C_p\ell^{p-2}.
\end{aligned}
\end{equation}
Furthermore,
\begin{equation}
\label{eq:5.3-truncation-derivative}
\beta_\ell'(r) =(p-2)r^{p-3}(1+r^2/\ell^2)^{-p/2}.
\end{equation}
Both $\beta_\ell(r)$ and $\beta_\ell'(r)$ increase to $r^{p-2}$ and $(p-2)r^{p-3}$, respectively.

Let $S_m={\mathrm P}_{\le m}$ be defined in \eqref{eq:2-smoothing}.
The stopped weak equation gives
\begin{equation*}
S_mU(t)=S_mu_0+\int_0^{t\wedge\rho}S_mA_u(s)\,\dd s +\int_0^t S_mB_s\,\dd W_s +\int_0^t\int_Z S_mH_s(z)\,\widetilde N(\dd s,\dd z).
\end{equation*}
For fixed $m$, the drift belongs to $L^1(0,T;L^p)$ almost surely, since $S_m:L^2\to L^p$ is bounded.
The noise integrability follows from Lemma \ref{lem:5.2-noise-jumps}.
The $L^p$ It\^o formula with jumps used in Theorem 3.1 of \cite{PodderKumar2026} therefore applies to $\Phi_\ell(S_mU)$ after localization.

For fixed $\ell$, the map $J_\ell:L^2\to L^2$ is globally Lipschitz.
The relatively compact ranges of the c\`adl\`ag paths $U$ in $L^2$ and $L^p$, strong convergence of $S_m$, and its uniform boundedness give
\begin{equation}
\begin{aligned}
\sup_{t\le T}\|S_mU(t)-U(t)\|_2&\longrightarrow0, & \sup_{t\le T} \|S_mJ_\ell(S_mU(t))-J_\ell(U(t))\|_2&\longrightarrow0 \quad\P\text{-a.s.}
\end{aligned}
\end{equation}
Self-adjointness of $S_m$ consequently yields
\begin{equation}
\label{eq:5.3-drift-limit}
\begin{aligned}
&\sup_{t\le T}\left| \int_0^{t\wedge\rho}\vartheta(s) \left[\Phi_\ell'(S_mu(s))[S_mA_u(s)] -p(J_\ell(u(s)),A_u(s))_2\right]\dd s\right| \\
&\qquad \qquad\le p\sup_{t\le T}\|S_mJ_\ell(S_mU(t))-J_\ell(U(t))\|_2 \int_0^\rho\|A_u(s)\|_2\,\dd s \longrightarrow0 \quad\P\text{-a.s.}
\end{aligned}
\end{equation}

The derivative bounds \eqref{eq:5.3-truncation-bounds} give the Taylor, trace, and linear-integrand bounds of Lemma \ref{lem:5.2-noise-jumps}, with constants independent of $\ell$.
Thus its mollification argument applies to $\Phi_\ell$.
The three martingales converge in expected supremum norm, and the two correction processes converge in the same norm.
The corresponding Wiener and Poisson integrals are those of \cite{MR2330977,MR3265175}.
Also,
\begin{equation}
\E\sup_{t\le T}\|S_mU(t)-U(t)\|_p^p \longrightarrow0.
\end{equation}
Consequently all limiting stochastic terms are associated with the existing adapted $L^p$-c\`adl\`ag process $U$ and its predictable left-limit process $U-$.
For $v\in H^2_\sigma\cap L^p$, write $r=|v|$ and set
\begin{equation*}
\mathcal V_\ell(v):=p\nu\int_{\R^3} \left[\beta_\ell(r)|\nabla v|^2 +r\beta_\ell'(r)|\nabla r|^2\right]\dd x +p\int_{\R^3}g_N(r^2)\beta_\ell(r)r^2\,\dd x.
\end{equation*}
The pressure decomposition of Lemma \ref{lem:5.1-pressure} and integration by parts give
\begin{equation*}
p(A(v),J_\ell(v))_2 =-\mathcal V_\ell(v) -p\langle\nabla\pi(v),J_\ell(v)\rangle.
\end{equation*}
Here $pJ_\ell(v)\cdot(v\cdot\nabla)v =v\cdot\nabla\phi_\ell(v)$ has zero integral.
The integrations by parts follow by spatial cutoffs and $\phi_\ell(v)v\in L^1$, $J_\ell(v)\in H^1\cap L^{3/2}$, and $\pi(v)\in L^3$.

Multiply the mollified identity by the continuous finite variation process $\vartheta$ and let $m\to\infty$.
Writing $M_\ell=M_\ell^W+M_\ell^{J,1}+M_\ell^{J,2}$ and $Q_\ell^W,Q_\ell^J$ for the processes of Lemma \ref{lem:5.2-noise-jumps} with $\Phi$ replaced by $\Phi_\ell$, we obtain, indistinguishably for $t\le T$,
\begin{equation}
\label{eq:5.3-truncated-Ito}
\begin{aligned}
&\vartheta(t)\Phi_\ell(U(t)) +\int_0^{t\wedge\rho}\vartheta(s)\mathcal V_\ell(u(s))\,\dd s +c\int_0^{t\wedge\rho}\vartheta(s)\Phi_\ell(u(s))\,\dd\Gamma(s) \\
&\quad=\Phi_\ell(u_0) -p\int_0^{t\wedge\rho}\vartheta(s) \langle\nabla\pi(u(s)),J_\ell(u(s))\rangle\,\dd s +Q_\ell^W(t)+Q_\ell^J(t)+M_\ell(t).
\end{aligned}
\end{equation}

\emph{Step 3:} Put $f=|u|^{p/2}$.
The Sobolev chain rule and \eqref{eq:5.3-truncation-derivative} give, almost everywhere in space and time,
\begin{equation*}
\begin{aligned}
\nabla\cdot J_\ell(u) &=\beta_\ell'(|u|)u\cdot\nabla|u|, & |\nabla\cdot J_\ell(u)| &\le\frac{2(p-2)}p|u|^{(p-2)/2}|\nabla f|.
\end{aligned}
\end{equation*}
The interpolation argument of Lemma \ref{lem:5.1-pressure} therefore gives, uniformly in $\ell$,
\begin{equation}
\label{eq:5.3-truncated-pressure}
p|\langle\nabla\pi(u),J_\ell(u)\rangle| \le\varepsilon\|\nabla f\|_2^2 +C_\varepsilon\Gamma'(s)(1+\Phi(u)).
\end{equation}
For almost every $s$, dominated convergence in $L^{3/2}$ also gives $\nabla\cdot J_\ell(u(s))\to\nabla\cdot J_p(u(s))$.
Equations \eqref{eq:5.3-preliminary-integrability} and \eqref{eq:5.3-integrability}, with \eqref{eq:5.3-truncated-pressure}, consequently imply
\begin{equation*}
\E\int_0^\rho\vartheta(s) \left|\langle\nabla\pi(u(s)), J_\ell(u(s))-J_p(u(s))\rangle\right|\dd s \longrightarrow0.
\end{equation*}

The functions $\Phi_\ell$ increase pointwise to $\Phi$.
Dini's theorem on the compact closure of each $L^p$ path range and \eqref{eq:5.3-preliminary-integrability} give
\begin{equation*}
\E\sup_{t\le T} |\Phi_\ell(U(t))-\Phi(U(t))|\longrightarrow0.
\end{equation*}
The integrals converge in expected supremum norm by \eqref{eq:5.3-integrability}.
The derivatives and Taylor remainders converge pointwise, with the uniform majorants in \eqref{eq:5.3-truncation-bounds}.
The trace and remainder arguments of Lemma \ref{lem:5.2-noise-jumps} therefore give
\begin{equation*}
\E\sup_{t\le T} \left(|M_\ell(t)-M(t)| +|Q_\ell^W(t)-Q^W(t)|+|Q_\ell^J(t)-Q^J(t)|\right) \longrightarrow0.
\end{equation*}
For the linear martingales, use dominated convergence of their quadratic variations and the Wiener BDG and Davis inequalities.
Their square roots have the integrable majorants established in that lemma.
For the remainder martingale, use its $L^1$ compensator bound, without squaring the remainder.

Finally, the nonnegative dissipations increase pointwise to
\begin{equation*}
\mathcal V_p(v) :=p\nu\mathcal D_p(v) +\frac{4\nu(p-2)}p \left\|\nabla\bigl(|v|^{p/2}\bigr)\right\|_2^2 +p\mathcal T_p(v).
\end{equation*}
Taking expectations in \eqref{eq:5.3-truncated-Ito} at $T$ and using the preceding integrable bounds shows that
\begin{equation*}
\sup_{\ell\ge1}\E\int_0^\rho \vartheta(s)\mathcal V_\ell(u(s))\,\dd s<\infty.
\end{equation*}
Monotone convergence identifies the limiting dissipation and proves its integrability.
Passing to the limit in \eqref{eq:5.3-truncated-Ito} gives the identity for $\Phi$, with dissipation $\mathcal V_p$, pressure test $J_p$, martingale $M$, and corrections $Q^W,Q^J$.
The convergence of the increasing dissipation integrals is uniform in time, since their difference is bounded by the difference of their terminal integrals.
Thus the limiting identity holds indistinguishably on $[0,T]$, including the possible jump at $\rho$.

\emph{Step 4: } Set $a=p\nu/2$ and
\begin{equation*}
\begin{aligned}
Y&:=\sup_{t\le\rho}\vartheta(t)\Phi(u(t)), &I&:=\int_0^\rho\vartheta(s)\,\dd\Gamma(s), &D&:=\int_0^\rho\vartheta(s)\mathcal D_p(u(s))\,\dd s, \\
Z&:=\int_0^\rho\vartheta(s)\Phi(u(s))\,\dd\Gamma(s), &\mathcal T&:=\int_0^\rho\vartheta(s)\mathcal T_p(u(s))\,\dd s, &M^*&:=\sup_{t\le T}|M(t)|.
\end{aligned}
\end{equation*}
All these variables are integrable.
Apply \eqref{eq:5.1-native-drift} to the limiting identity and bound the two correction terms by Lemma \ref{lem:5.2-noise-jumps}.
Since $\Gamma'\ge1$, this gives, for a deterministic $C_0$ which is independent of the stopping parameters,
\begin{equation}
\label{eq:5.3-weighted-energy}
\begin{aligned}
\vartheta(t)\Phi(U(t)) +&a\int_0^{t\wedge\rho}\vartheta(s)\mathcal D_p(u(s))\,\dd s +p\int_0^{t\wedge\rho}\vartheta(s)\mathcal T_p(u(s))\,\dd s +(c-C_0)\int_0^{t\wedge\rho} \vartheta(s)\Phi(u(s))\,\dd\Gamma(s) \\
&\quad\le\Phi(u_0) +C_0\int_0^{t\wedge\rho}\vartheta(s)\,\dd\Gamma(s)+M(t).
\end{aligned}
\end{equation}
The pressure term has been retained and estimated through Lemma \ref{lem:5.1-pressure}.
No cancellation with the Leray projection is used.
Moreover,
\begin{equation*}
I=\frac{1-e^{-c\Gamma(\rho)}}c\le\frac1c.
\end{equation*}
For $c\ge C_0$, taking the supremum and separately evaluating \eqref{eq:5.3-weighted-energy} at $T$ yields
\begin{equation}
\label{eq:5.3-absorption-input}
\E\left[Y+aD+p\mathcal T+(c-C_0)Z\right] \le2\E\Phi(u_0)+\frac{2C_0}{c} +2\E M^*.
\end{equation}
Lemma \ref{lem:5.2-noise-jumps}, together with $\dd s\le\dd\Gamma(s)$, gives for every $\delta>0$
\begin{equation*}
\E M^* \le\delta\E Y+C_\delta\E(I+Z) \le\delta\E Y+C_\delta\left(c^{-1}+\E Z\right).
\end{equation*}
Here $u=u_-$ almost everywhere for both $\dd s$ and $\dd\Gamma(s)$, because $\Gamma$ is absolutely continuous.
Choose $\delta=1/4$ and then fix $c\ge\max\{1,C_0+2C_\delta+1\}$.
Absorption in \eqref{eq:5.3-absorption-input} proves \eqref{eq:5.3-weighted-estimate}, with constants independent of $n,L,K,T$.

\emph{Step 5:} For fixed $L,K,T$, the strict inequality $\rho_{L,K,T}<\tau_{\max}$ implies
\begin{equation*}
\rho_{n,L,K,T}=\rho_{L,K,T} \quad\text{for all sufficiently large }n, \quad\P\text{-a.s.}
\end{equation*}
Thus the stopped intervals increase to the desired closed interval, including its terminal state.
Monotone convergence proves \eqref{eq:5.3-weighted-estimate} for $\rho=\rho_{L,K,T}$.

For either stopping time, $\Gamma(s)\le K$ on $[0,\rho]$, so $\vartheta(s)\ge e^{-cK}$.
The chain rule and Sobolev inequality \eqref{eq:5.1-weighted-dissipation} give
\begin{equation*}
\|u(s)\|_{3p}^p \le C\left\|\nabla\bigl(|u(s)|^{p/2}\bigr)\right\|_2^2 \le C_p\mathcal D_p(u(s))
\end{equation*}
for almost every $s<\rho$.
Consequently, \eqref{eq:5.3-weighted-estimate} implies \eqref{eq:5.3-unweighted-estimate}.
On $\{\lambda_L\le T\wedge\eta_K\}$, $\rho_{L,K,T}=\lambda_L<\tau_{\max}$ and $\mathcal Q_u(\lambda_L)\ge L$ by \eqref{eq:3.6-exit-levels}.
Markov's inequality proves \eqref{eq:5.3-exit-probability}.
Only this inequality at the exit level is used, so jump overshoots, including the case $\lambda_L=0$, are retained.
\end{proof}
\section{Global well-posedness and moment bounds}
\label{Section 6} This section completes the passage from the maximal local solution to global well-posedness.
We first exclude explosion for $H^1$ data, extend the conclusion to $L^p\cap L^2$ data, and then establish unweighted moment bounds and continuous dependence under their stated assumptions.

We now combine the intrinsic continuation criterion of Theorem~\ref{thm:3.6-intrinsic-blowup} with the uniform stopped $L^p$ estimate of Proposition~\ref{prop:5.3-stopped-Lp}.
Set
\begin{equation}
\label{eq:6.1-intersection-space}
X_p:=L^p_\sigma(\R^3)\cap H^1_\sigma(\R^3), \qquad \|v\|_{X_p}:=\|v\|_p+\|v\|_{H^1},
\end{equation}
and retain the taming dissipation $\mathcal T_p$ from \eqref{eq:5.3-taming-dissipation}.
We also continue to write $$ \Gamma(t) = \int_0^t \left( 1+\|u(s)\|_{H^1}^{2p} +\|u(s)\|_\infty \right)\,\dd s $$ for the accumulated regularity functional introduced in \eqref{eq:4.3-accumalted-regularity}.
Under the hypotheses below, Theorem~\ref{thm:4.3-H1-persistence} yields $$ \Gamma((T\wedge\tau_{\max})-)<\infty \qquad\P\text{-a.s.} $$ for every finite $T$.

The stopped $L^p$ estimate controls the probability that the intrinsic quantity reaches a large level while the accumulated regularity functional $\Gamma$ remains bounded.
Since $\Gamma$ stays finite up to every finite maximal lifetime, this estimate, combined with the intrinsic blow-up alternative, excludes finite-time explosion and yields global continuation.

\begin{theorem}
[Non-explosion and global continuation] \label{thm:6.1-global-continuation} Let $p>3$ and assume the hypotheses of Proposition~\ref{prop:5.3-stopped-Lp}; in particular, $$ u_0\in L^p(\Omega,\mathcal F_0;L^p_\sigma) \cap L^2(\Omega,\mathcal F_0;H^1_\sigma).
$$ Let $(u,(\tau_n),\tau_{\max})$ be the maximal solution constructed in Theorem~\ref{thm:2-Paper-I} (ii).
Then
\begin{equation}
\label{eq:6.1-nonexplosion}
\P(\tau_{\max}=\infty)=1.
\end{equation}
Consequently, $u$ is the unique global strong solution of \eqref{eq:1_Main} on the prescribed stochastic basis in the local-energy uniqueness class of Paper~I.
On a common event of probability one, for every $T<\infty$,
\begin{equation}
\label{eq:6.1-global-path-regularity}
\begin{aligned}
u&\in \mathbb D([0,T];X_p) \cap L^2(0,T;H^2_\sigma) \cap L^p(0,T;L^{3p}_\sigma), \\
|u|^{p/2}&\in L^2(0,T;H^1), \qquad \int_0^T\mathcal T_p(u(s))\,\dd s<\infty.
\end{aligned}
\end{equation}
Moreover,
\begin{equation}
\label{eq:6.1-global-H1-estimate}
\E\left[ \sup_{0\le t\le T}\|u(t)\|_{H^1}^2 + \nu\int_0^T\|u(s)\|_{H^2}^2\,\dd s \right] \le C_T\left(1+\E\|u_0\|_{H^1}^2\right),
\end{equation}
and, for the constants $c,C$ of Proposition~\ref{prop:5.3-stopped-Lp},
\begin{equation}
\label{eq:6.1-global-weighted-estimate}
\begin{aligned}
\E\bigg[ &\sup_{0\le t\le T} e^{-c\Gamma(t)}\|u(t)\|_p^p +\frac{p\nu}{2} \int_0^T e^{-c\Gamma(s)} \mathcal D_p(u(s))\,\dd s \\
&\quad +p\int_0^T e^{-c\Gamma(s)} \mathcal T_p(u(s))\,\dd s +\int_0^T e^{-c\Gamma(s)} \|u(s)\|_p^p\,\dd\Gamma(s) \bigg] \le C\left(1+\E\|u_0\|_p^p\right).
\end{aligned}
\end{equation}
Here $C$ and $c$ are independent of $T$ and of all localization and regularization parameters.
\end{theorem}

\begin{proof}
\emph{Step 1:} Fix $T>0$ and write
\begin{equation*}
\Gamma_T^-:=\lim_{t\uparrow T\wedge\tau_{\max}}\Gamma(t).
\end{equation*}
If $\overline\Gamma$ denotes the time integral of the integrand defining $\Gamma$ extended by zero after $\tau_{\max}$, as in Proposition \ref{prop:5.3-stopped-Lp}, then $\Gamma_T^-=\overline\Gamma(T)$.
It is therefore $\mathcal F_T$-measurable and is finite almost surely by \eqref{eq:4.3-regularity}.
For an integer $K\ge1$, set
\begin{equation*}
A_{K,T}:=\{\tau_{\max}\le T,\ \Gamma_T^-<K\}.
\end{equation*}
On $A_{K,T}$, the the accumulated regularity functional does not reach $K$ before $\tau_{\max}$, so $\eta_K=\tau_{\max}$.
By \eqref{eq:3.6-exit-lifetime}, every integer $L\ge1$ satisfies $\lambda_L<\tau_{\max}$ on this event, outside a common null set.
Hence

$$A_{K,T}\subseteq\{\lambda_L\le T\wedge\eta_K\} \quad\text{up to a null set}$$.

Proposition \ref{prop:5.3-stopped-Lp} gives $$\P(A_{K,T}) \le\frac{Ce^{cK}}{L} \left(1+\E\|u_0\|_p^p\right).$$

Letting integer $L\to\infty$ with $K,T$ fixed yields $\P(A_{K,T})=0$.
Since $\Gamma_T^-<\infty$ almost surely,
\begin{equation*}
\{\tau_{\max}\le T\} =\bigcup_{K=1}^{\infty}A_{K,T} \quad\text{up to a null set}.
\end{equation*}
Thus $\P(\tau_{\max}\le T)=0$.
Taking $T$ through the positive integers proves \eqref{eq:6.1-nonexplosion}.

\emph{Step 2:} By Theorem~\ref{thm:2-Paper-I} (ii) , each $\tau_n$ admits a closed stopped representative $U_n$, the representatives agree on their closed overlaps, and $$\tau_n\uparrow\tau_{\max}\qquad\P\text{-a.s.}$$ Theorem~\ref{thm:4.3-H1-persistence} provides the compatible $H^1_\sigma$-c\`adl\`ag realization of the same maximal flow.

On the event $\{\tau_{\max}=\infty\}$, for every finite $T$ one has $T<\tau_n$ for all sufficiently large $n$, pathwise.
Hence the compatible stopped representatives determine an adapted process on $[0,\infty)$ with c\`adl\`ag paths in both $L^p_\sigma$ and $H^1_\sigma$.
The two realizations agree as distributions on their common intervals.
Their right and left limits therefore coincide as distributions, and hence belong to $X_p=L^p_\sigma\cap H^1_\sigma.$ Since $X_p$ carries the sum norm \eqref{eq:6.1-intersection-space} , the resulting process is $X_p$-c\`adl\`ag on every finite interval.
Modifying it on the common exceptional null set by a fixed path yields an adapted global representative.

For every finite $T$, choose $n$ such that $T<\tau_n$.
Then the stopped identity \eqref{eq:2-stopped-equation}, applied to $U_n$, coincides with the original equation on $[0,T]$.
Thus the global representative solves \eqref{eq:1_Main} on every finite horizon.
The Wiener and compensated Poisson terms are precisely the localized integrals appearing in \eqref{eq:2-stopped-equation}.
Their integrands are evaluated at predictable left limits, and the stopping convention \eqref{eq:2-predictable-intervals} retains the upper endpoint of each Poisson integral.
Moreover, $\mathcal E_p(U_n;0,\tau_n)<\infty$ by Theorem~\ref{thm:2-Paper-I} (ii), so the global representative belongs locally to the uniqueness class of Definition~\ref{def:2-local-solution}.

Let $v$ be another global solution in this class and let $(V_m,\sigma_m)$ be an exhausting sequence of its admissible closed localizations.
By Theorem~\ref{thm:2-Paper-I} (i), $$U_n=V_m \qquad\text{on }[0,\tau_n\wedge\sigma_m]$$ almost surely for every $n,m$.
Taking a countable intersection of these full-probability events and then letting $n,m\to\infty$ proves that $u$ and $v$ are indistinguishable on $[0,\infty)$.

\emph{Step 3:} Equation \eqref{eq:6.1-global-H1-estimate} follows directly from \eqref{eq:4.3-maximal-energy} and \eqref{eq:6.1-nonexplosion}.
Fix $T<\infty$.
On the event of non-explosion, choose $n$ such that $T<\tau_n$.
By Theorem~\ref{thm:2-Paper-I} (ii), $ \mathcal E_p(U_n;0,\tau_n)<\infty, $ and $u=U_n$ on $[0,T]$.
Hence, by the definition \eqref{eq:2-local-energy} and \eqref{eq:2-local-intrinsic-energy}, $$ \mathcal Q_u(T) = \sup_{0\le t\le T}\|u(t)\|_p^p +\int_0^T\|u(s)\|_{3p}^p\,\dd s <\infty \qquad\P\text{-a.s.} $$ Furthermore, \eqref{eq:4.3-regularity} and continuity of the accumulated regularity functional $\Gamma$ give $\Gamma(T)<\infty \ \P\text{-a.s.} $ For $m\in\mathbb N$, set $\rho_m:=T\wedge\lambda_m\wedge\eta_m $.
The sequence $(\rho_m)$ is nondecreasing, and
\begin{equation}
\label{eq:6.1-eventual-equality}
\rho_m=T \qquad\text{for all sufficiently large }m, \quad\P\text{-a.s.}
\end{equation}
Indeed, it suffices to take $$ m>\max\{\mathcal Q_u(T),\Gamma(T),1\}.
$$ On $[0,\rho_m]$ the weight in \eqref{eq:5.3-weighted-estimate} is the same function $e^{-c\Gamma(s)}$, independently of $m$.
Since $\rho_m$ is nondecreasing, every nonnegative term on the left-hand side of \eqref{eq:5.3-weighted-estimate} is nondecreasing in $m$.
Monotone convergence together with \eqref{eq:6.1-eventual-equality} therefore yields \eqref{eq:6.1-global-weighted-estimate}, including the terminal value at $T$.
Since $\Gamma(T)<\infty$ almost surely, $$ e^{-c\Gamma(s)}\ge e^{-c\Gamma(T)}>0, \qquad 0\le s\le T.
$$ The random variable on the left-hand side of \eqref{eq:6.1-global-weighted-estimate} is finite almost surely, hence $$ \sup_{t\le T}\|u(t)\|_p^p +\int_0^T\mathcal D_p(u(s))\,\dd s +\int_0^T\mathcal T_p(u(s))\,\dd s <\infty \qquad\P\text{-a.s.} $$ The chain rule and \eqref{eq:5.1-weighted-dissipation} then give $$ u\in L^p(0,T;L^{3p}_\sigma), \qquad |u|^{p/2}\in L^2(0,T;H^1), $$ while the $H^1$- and $H^2$-parts of \eqref{eq:6.1-global-path-regularity} follow from Theorem~\ref{thm:4.3-H1-persistence}.
Taking a countable intersection over integer $T$ makes all these assertions simultaneous on every finite horizon.

\end{proof}

\begin{remark}
[Direct continuation for $3<p\le6$] \label{rem:6.1-low-p-direct} When $3<p\le6$, the non-explosion conclusion of Theorem~\ref{thm:6.1-global-continuation} follows directly from the intrinsic blow-up criterion of Theorem~\ref{thm:3.6-intrinsic-blowup} and the $H^1$-persistence estimate of Theorem~\ref{thm:4.3-H1-persistence}, without using the weighted native $L^p$ estimate of Proposition~\ref{prop:5.3-stopped-Lp}.
Indeed, fix $T>0$ and set $$ \zeta:=T\wedge\tau_{\max},\qquad M_T:=\sup_{\substack{0\le t\le T\\ t<\tau_{\max}}} \|u(t)\|_{H^1}^2,\qquad D_T:=\int_0^\zeta\|u(s)\|_{H^2}^2\,\dd s .
$$ By Theorem~\ref{thm:4.3-H1-persistence}, $$ M_T+D_T<\infty \qquad\P\text{-a.s.} $$ For $v\in H^2(\R^3)$ and $3<p\le6$, Sobolev embedding gives $ \|v\|_p\le C_p\|v\|_{H^1}, $ while interpolation between $L^6$ and $L^\infty$, together with the Gagliardo--Nirenberg estimate $$ \|v\|_\infty \le C\|v\|_{H^1}^{1/2}\|v\|_{H^2}^{1/2}, $$ yields $$ \|v\|_{3p}^p \le \|v\|_6^2\|v\|_\infty^{p-2} \le C_p \|v\|_{H^1}^{(p+2)/2} \|v\|_{H^2}^{(p-2)/2}.
$$ Since $(p-2)/2\le2$, H\"older's inequality in time therefore gives
\begin{equation}
\label{eq:6.1-low-p-bound}
\begin{aligned}
\lim_{t\uparrow\zeta}\mathcal Q_u(t) &\le C_p M_T^{p/2} + C_p M_T^{(p+2)/4} \int_0^\zeta \|u(s)\|_{H^2}^{(p-2)/2}\,\dd s \\
&\le C_p M_T^{p/2} + C_p T^{(6-p)/4} M_T^{(p+2)/4} D_T^{(p-2)/4} <\infty \qquad\P\text{-a.s.}
\end{aligned}
\end{equation}
On the event $\{\tau_{\max}\le T\}$ one has $\zeta=\tau_{\max}$, and \eqref{eq:6.1-low-p-bound} contradicts the intrinsic blow-up alternative of Theorem~\ref{thm:3.6-intrinsic-blowup}.
Hence $$ \P(\tau_{\max}\le T)=0 .
$$ Since $T>0$ is arbitrary, $\tau_{\max}=\infty$ almost surely.

The restriction $p\le6$ is precisely the range in which the $H^1$--$H^2$ persistence estimates alone control the intrinsic quantity $\mathcal Q_u$.
One has $H^1(\R^3)\hookrightarrow L^p(\R^3)$ and $(p-2)/2\le2$.
For $p>6$, both of these features fail, and the weighted native $L^p$ argument of Proposition~\ref{prop:5.3-stopped-Lp} provides the continuation mechanism for the full range $p>3$.
\end{remark}

Positive-time regularization allows the preceding global theorem to be applied after a strictly positive stopping time.
By localizing the restart datum and invoking maximality, we obtain global continuation for $L^p\cap L^2$ data without an initial $H^1$ assumption.

\begin{theorem}
[Global continuation for $L^p\cap L^2$ initial data] \label{thm:6.2-finite-energy-global} Let $p>3$.
Retain the coefficient assumptions, taming condition, and joint compatibility hypotheses of Theorem \ref{thm:6.1-global-continuation}, but replace its initial-data assumption by
\begin{equation*}
u_0\in L^p(\Omega,\mathcal F_0;L^p_\sigma(\R^3)) \cap L^2(\Omega,\mathcal F_0;L^2_\sigma(\R^3)).
\end{equation*}
In particular, both noise bounds \eqref{eq:2-raw-L2-growth} and \eqref{eq:2-H1} remain in force.
Let $(u,(\tau_n),\tau_{\max})$ be the maximal construction of \cref{thm:2-Paper-I} (ii).
Then
\begin{equation}
\label{eq:6.2-finite-energy-nonexplosion}
\P(\tau_{\max}=\infty)=1.
\end{equation}
Thus $u$ is the unique global strong solution of \eqref{eq:1_Main} on the prescribed stochastic basis, with the prescribed driving noises, in the local-energy uniqueness class.

All intersections below carry their sum norms.
On a common event of probability one, for every $0<\varepsilon<T<\infty$,
\begin{equation}
\label{eq:6.2-finite-energy-paths}
\begin{aligned}
u&\in\mathbb D([0,T];L^p_\sigma\cap L^2_\sigma), & u&\in L^2(0,T;H^1_\sigma) \cap L^p(0,T;L^{3p}_\sigma), \\
|u|^{p/2}&\in L^2(0,T;H^1), & u&\in\mathbb D([\varepsilon,T];L^p_\sigma\cap H^1_\sigma) \cap L^2(\varepsilon,T;H^2_\sigma).
\end{aligned}
\end{equation}
Moreover, with $\mathcal T_0$ as in \eqref{eq:4.1-energy-notation},
\begin{equation}
\label{eq:6.2-L2-energy}
\E\bigg[ \sup_{0\le t\le T}\|u(t)\|_2^2 +\nu\int_0^T\|\nabla u(s)\|_2^2\,\dd s +\int_0^T\mathcal T_0(u(s))\,\dd s \bigg] \le C_T^{(0)} \left(1+\E\|u_0\|_2^2\right),
\end{equation}
and
\begin{equation}
\label{eq:6.2-weighted-H1}
\E\left[ \sup_{0<t\le T}t\|u(t)\|_{H^1}^2 +\nu\int_0^T s\|u(s)\|_{H^2}^2\,\dd s \right] \le C_T^{(1)}\left(1+\E\|u_0\|_2^2\right).
\end{equation}
Here $C_T^{(0)}$ depends only on $T,\nu,C_0$, and $C_T^{(1)}$ only on $T,\nu,c_{N,\nu},C_0,C_1$.
Both constants are independent of the announcing sequence and all localization and regularization parameters.
\end{theorem}

\begin{proof}
\emph{Step 1:} For a process $v$ with closed lifetime $\theta$, write
\begin{equation*}
\mathcal E_p(v;0,\theta) :=\E\left[ \sup_{0\le s\le\theta}\|v(s)\|_p^p +\int_0^\theta \left\|\nabla\bigl(|v(s)|^{p/2}\bigr)\right\|_2^2\,\dd s \right].
\end{equation*}
Set $\rho=\tau_1$ and let $U=u(\cdot\wedge\rho)$ be its closed stopped representative.
\cref{thm:2-Paper-I}(ii), and Lemma \ref{lem:3.1-strict-localization} give
\begin{equation*}
\begin{aligned}
0<\rho\le1,\qquad \rho<\tau_{\max} &\quad\P\text{-a.s.}, & \mathcal E_p(U;0,\rho)&<\infty.
\end{aligned}
\end{equation*}
Proposition \ref{prop:4.4-time-weighted-H1} therefore applies at $\rho$.
For $k\ge1$, define
\begin{equation*}
A_k:=\{\rho\ge k^{-1}\},\qquad \xi_k:=
\begin{cases}
U(\rho),&\text{on }A_k, \\
0,&\text{on }A_k^c.
\end{cases}
\end{equation*}
Then $A_k\in\mathcal F_\rho$ and $A_k\uparrow\Omega$ up to a null set.
Optional evaluation gives $\mathcal F_\rho$-measurability of $U(\rho)$ in $L^p_\sigma$, and positive-time regularization gives $U(\rho)\in H^1_\sigma$ almost surely.
The continuous injection of the separable Banach space $L^p_\sigma\cap H^1_\sigma$ into $L^p_\sigma$ has a Borel inverse on its image.
Consequently, $\xi_k$ is $\mathcal F_\rho$-measurable with values in this intersection.
Furthermore, \eqref{eq:4.4-time-weighted-energy} with $T=1$ yields
\begin{equation*}
\begin{aligned}
\E\|\xi_k\|_p^p &\le\mathcal E_p(U;0,\rho)<\infty,& \E\|\xi_k\|_{H^1}^2 &\le k\E\bigl[\rho\|u(\rho)\|_{H^1}^2\bigr] \le Ck\left(1+\E\|u_0\|_2^2\right).
\end{aligned}
\end{equation*}

\emph{Step 2:} For $s\ge0$ and $B\in\mathcal Z$ with $\mu(B)<\infty$, set
\begin{equation*}
\begin{aligned}
\mathcal F_s^\rho&:=\mathcal F_{\rho+s}, &W_s^\rho&:=W_{\rho+s}-W_\rho, & N^\rho((0,s]\times B)&:=N((\rho,\rho+s]\times B).
\end{aligned}
\end{equation*}
By the bounded stopping-time translation in \cref{thm:2-Paper-I}(iii), this stochastic basis satisfies the usual conditions and the required joint compatibility.
The shifted coefficients
\begin{equation*}
\begin{aligned}
\Sigma^\rho(\omega,s,v) &:=\Sigma(\omega,\rho(\omega)+s,v), & \mathcal G^\rho(\omega,s,v,z) &:=\mathcal G(\omega,\rho(\omega)+s,v,z)
\end{aligned}
\end{equation*}
retain their predictable measurability and all structural bounds with the same constants.
Thus Theorem \ref{thm:6.1-global-continuation} applies to the shifted equation with initial datum $\xi_k$.
Let $v_k$ be its global solution.
Its maximal construction provides increasing $(\mathcal F_s^\rho)$-stopping times $(\sigma_{k,j})_{j\ge1}$ such that
\begin{equation*}
0<\sigma_{k,j}\le j,\qquad \sigma_{k,j}\uparrow\infty\quad\P\text{-a.s.}, \qquad \mathcal E_p(v_k;0,\sigma_{k,j})<\infty.
\end{equation*}
Define
\begin{equation*}
\widehat v_k(t):=v_k(t-\rho),\quad t\ge\rho, \qquad \eta_{k,j}:=\rho+\sigma_{k,j}.
\end{equation*}
The same translation result shows that $\eta_{k,j}$ is an $(\mathcal F_t)$-stopping time and that the stopped branch $\widehat v_k(\cdot\wedge\eta_{k,j})$ on $[\rho,\infty)$ has an adapted c\`adl\`ag extension by zero before $\rho$.
It solves the original equation on $(\rho,\eta_{k,j}]$ and starts from $\xi_k$ at $\rho$.

\emph{Step 3: } Fix $k,j$ and set
\begin{equation*}
\zeta_{k,j}:=
\begin{cases}
\eta_{k,j},&\text{on }A_k, \\
\rho,&\text{on }A_k^c.
\end{cases}
\end{equation*}
Since $A_k\in\mathcal F_\rho$ and $\rho\le\eta_{k,j}$,
\begin{equation*}
\{\zeta_{k,j}\le t\} =\bigl(A_k\cap\{\eta_{k,j}\le t\}\bigr) \cup\bigl(A_k^c\cap\{\rho\le t\}\bigr) \in\mathcal F_t.
\end{equation*}
Hence $\zeta_{k,j}$ is a positive stopping time bounded by $j+1$.
Define directly
\begin{equation*}
w_{k,j}(t):=
\begin{cases}
U(t),&0\le t\le\rho, \\
\widehat v_k(t\wedge\eta_{k,j}), &t>\rho,\quad\omega\in A_k, \\
U(\rho),&t>\rho,\quad\omega\in A_k^c.
\end{cases}
\end{equation*}
For deterministic $t$, the branch selection after $\rho$ is $\mathcal F_t$-measurable because $A_k\cap\{\rho<t\}\in\mathcal F_t$.
The stopped-branch adaptedness established above therefore makes $w_{k,j}$ adapted.
On $A_k$, $\widehat v_k(\rho)=\xi_k=U(\rho)$.
On $A_k^c$ the process is constant after $\rho$.
Thus $w_{k,j}$ is c\`adl\`ag and stopped at $\zeta_{k,j}$, with predictable left limits.

The process $\mathds 1_{A_k}\mathds 1_{(\rho,\infty)}$ is adapted and left-continuous.
In particular, both terms on the left of
\begin{equation}
\label{eq:6.2-predictable-partition}
\mathds 1_{(0,\rho]} +\mathds 1_{A_k}\mathds 1_{(\rho,\eta_{k,j}]} =\mathds 1_{(0,\zeta_{k,j}]}
\end{equation}
are predictable.
Multiply the translated stopped weak formulation for $\widehat v_k$ by $\mathds 1_{A_k}$ and add the stopped formulation for $U$.
On their respective integration intervals, the states and predictable left limits agree with those of $w_{k,j}$.
The stochastic integral restriction identities and \eqref{eq:6.2-predictable-partition} give the original stopped equation through $\zeta_{k,j}$, exactly as in \cref{thm:2-Paper-I} (iv).
The Poisson interval $(0,\rho]$ contains the possible jump at $\rho$, while the restarted interval $(\rho,\eta_{k,j}]$ excludes it and retains its own terminal jump.
Moreover,
\begin{equation*}
\mathcal E_p(w_{k,j};0,\zeta_{k,j}) \le\mathcal E_p(U;0,\rho) +\mathcal E_p(v_k;0,\sigma_{k,j})<\infty.
\end{equation*}
Thus $\zeta_{k,j}$ is an attainable lifetime for the original initial datum and driving noises in precisely the class defining $\tau_{\max}$.
\cref{thm:2-Paper-I}(ii), implies
\begin{equation*}
\zeta_{k,j}\le\tau_{\max}\quad\P\text{-a.s.}
\end{equation*}
Taking a common full-probability event for all $k,j$, on $A_k$ we have
\begin{equation*}
\tau_{\max}\ge\rho+\sigma_{k,j} \longrightarrow\infty\qquad (j\to\infty).
\end{equation*}
Since $A_k\uparrow\Omega$ almost surely, this proves \eqref{eq:6.2-finite-energy-nonexplosion}.

\emph{Step 4: Global regularity and uniqueness.} By Theorem~\ref{thm:2-Paper-I}\textup{(ii)}, the maximal construction admits compatible closed stopped representatives $U_n$ on $[0,\tau_n]$, with $$\mathcal E_p(U_n;0,\tau_n)<\infty, \qquad \tau_n\uparrow\tau_{\max}.$$ Since $\tau_{\max}=\infty$ almost surely, $\tau_n\uparrow\infty$ on a common event of probability one.
Propositions \ref{prop:4.2-L2-persistence} and \ref{prop:4.4-time-weighted-H1} provide, respectively, the $L^2_\sigma$ realization up to time zero and the $H^1_\sigma$ realization on every interval bounded away from zero, for this same maximal flow.

Consequently, for every finite $T$ one has $T<\tau_n$ for all sufficiently large $n$, pathwise, and the compatible representatives determine a single adapted process on $[0,\infty)$.
From \cref{subsec:2-hypotheses}, their right and left limits coincide as distributions.
Since intersections carry the sum norm, the resulting paths are c\`adl\`ag in $L^p_\sigma\cap L^2_\sigma$ on $[0,T]$ and in $L^p_\sigma\cap H^1_\sigma$ on every $[\varepsilon,T]$, $0<\varepsilon<T$.

For each finite $T$, choose $n$ such that $T<\tau_n$.
Then \eqref{eq:2-stopped-equation}, applied to $U_n$, coincides with the original equation on $[0,T]$.
Hence the resulting process solves \eqref{eq:1_Main} globally.
Its stochastic integrands are evaluated at the predictable left-limit process, as specified in \cref{subsec:2-stochastic} The stopping convention \eqref{eq:2-predictable-intervals} retains the terminal point of every Poisson integral.

Moreover, Theorem~\ref{thm:2-Paper-I} (ii) and \eqref{eq:2-local-intrinsic-energy} yield, on every finite horizon, $$\sup_{t\le T}\|u(t)\|_p^p +\int_0^T \left\|\nabla\bigl(|u(s)|^{p/2}\bigr)\right\|_2^2\,\dd s +\int_0^T\|u(s)\|_{3p}^p\,\dd s <\infty \qquad\P\text{-a.s.}$$ Thus the native $L^p$ assertions in \eqref{eq:6.2-finite-energy-paths} hold.
Equations \eqref{eq:4.2-maximal-energy} and \eqref{eq:4.4-time-weighted-energy}, together with $\tau_{\max}=\infty$, give \eqref{eq:6.2-L2-energy} and \eqref{eq:6.2-weighted-H1}.
In particular, $$\E\left[ \sup_{\varepsilon\le t\le T}\|u(t)\|_{H^1}^2 +\nu\int_\varepsilon^T\|u(s)\|_{H^2}^2\,\dd s \right] \le \frac{C_T^{(1)}}{\varepsilon} \left(1+\E\|u_0\|_2^2\right).$$ Taking a countable intersection over integer terminal times and reciprocal-integer positive lower endpoints gives a common full-probability event on which all the asserted regularity properties hold.

Finally, let $v$ be another global solution in the uniqueness class of Definition~\ref{def:2-local-solution}, and let $(V_m,\sigma_m)$ be an increasing family of admissible closed localizations exhausting its lifetime.
By Theorem~\ref{thm:2-Paper-I}\textup{(i)}, $$U_n=V_m \qquad\text{on }[0,\tau_n\wedge\sigma_m]$$ almost surely for every $n,m$.
Taking a countable intersection of these full-probability events and then letting $n,m\to\infty$ proves that $u$ and $v$ are indistinguishable on $[0,\infty)$.
\end{proof}

The weighted global estimate does not itself give finite unweighted $p$th moments.
The global estimate obtained above is naturally weighted by the accumulated regularity functional $\Gamma$.
This raises the question whether, under stronger moment assumptions on the initial datum and the jump coefficient, one can recover an unweighted intrinsic $L^p$ moment estimate without requiring any exponential integrability of $\Gamma$.
The next proposition gives such a refinement in the range $3<p\le6$.

We obtain these for $3<p\le6$ from a higher $H^1$ energy estimate, assuming the corresponding $p$th moments of the initial datum and the $H^1$-valued jump coefficient.
\begin{proposition}
[Unweighted intrinsic moments] \label{prop:6.3-unweighted-intrinsic-moments} Let $3<p\le6$ and retain the structural hypotheses of Theorem \ref{thm:6.1-global-continuation}.
Assume additionally that
\begin{equation}
\label{eq:6.3-p-moment}
\begin{aligned}
u_0&\in L^p(\Omega,\mathcal F_0;H^1_\sigma), & \int_Z\|\mathcal G(t,v,z)\|_{H^1}^{p}\,\mu(\dd z) &\le K_p\bigl(1+\|v\|_{H^1}^{p}\bigr), \qquad v\in H^1_\sigma,
\end{aligned}
\end{equation}
uniformly in the suppressed $(\omega,t)$-variables.
Let $u$ be the global solution supplied by Theorem \ref{thm:6.1-global-continuation}.
Then, for every $T<\infty$,
\begin{equation}
\label{eq:6.3-higher-H1}
\begin{aligned}
\E\biggl[ \sup_{0\le t\le T}\|u(t)\|_{H^1}^{p} &+\nu\int_0^T \bigl(1+\|u(s)\|_{H^1}^2\bigr)^{(p-2)/2} \|u(s)\|_{H^2}^2\,\dd s\biggr] \le C_T\bigl(1+\E\|u_0\|_{H^1}^{p}\bigr).
\end{aligned}
\end{equation}
In particular, with
\begin{equation*}
\begin{aligned}
\mathcal Q_u(T) &:=\sup_{0\le t\le T}\|u(t)\|_p^p +\int_0^T\|u(s)\|_{3p}^p\,\dd s, & \mathcal D_p(v) &:=\int_{\R^3}|v|^{p-2}|\nabla v|^2\,\dd x,
\end{aligned}
\end{equation*}
one has
\begin{equation}
\label{eq:6.3-unweighted}
\E\left[ \mathcal Q_u(T)+\int_0^T\mathcal D_p(u(s))\,\dd s \right] \le C_T\bigl(1+\E\|u_0\|_{H^1}^{p}\bigr).
\end{equation}
The constants depend only on $p,T,\nu,c_{N,\nu}$, and $K_p$.
They are independent of all stopping levels and regularization parameters.
No exponential moment of $\Gamma$ is assumed.
\end{proposition}

\begin{proof}
The embedding $H^1(\R^3)\hookrightarrow L^p(\R^3)$ and \eqref{eq:6.3-p-moment} imply the initial-data hypotheses of Theorem \ref{thm:6.1-global-continuation}.
We use its adapted, strongly $H^1_\sigma$-c\`adl\`ag representative.
Set
\begin{equation*}
\label{eq:6.3-Lyapunov-function}
Y(v):=1+\|v\|_{H^1}^2, \qquad \Phi(v):=Y(v)^{p/2}, \qquad Y(t):=Y(u(t)).
\end{equation*}

\emph{Step 1:} Write
\begin{equation*}
A(v):=\nu\Delta v-\mathcal P((v\cdot\nabla)v) -\mathcal P(g_N(|v|^2)v).
\end{equation*}
For $v\in H^2_\sigma$, the $L^2$ cancellation, the self-adjointness of $\mathcal P$, and the coercivity calculation in Theorem \ref{thm:4.3-H1-persistence} give
\begin{equation}
\label{eq:6.3-H1-coercivity}
\begin{aligned}
2(A(v),(I-\Delta)v)_2 &\le -2\nu\|\nabla v\|_2^2-\nu\|\Delta v\|_2^2 +c_{N,\nu}\|\nabla v\|_2^2 \le -\kappa_\nu\|v\|_{H^2}^2 +C_{N,\nu}\|v\|_{H^1}^2,
\end{aligned}
\end{equation}
where $\kappa_\nu>0$.
The Leray projection is removed only against the divergence-free vectors $v$ and $-\Delta v$, no cancellation against $|v|^{p-2}v$ is used.

The function $\Phi$ is convex and twice continuously Fr\'echet differentiable on $H^1_\sigma$.
Its derivatives and Taylor remainder satisfy
\begin{equation}
\label{eq:6.3-Lyapunov}
\begin{aligned}
D\Phi(v)[h] &=pY(v)^{(p-2)/2}(v,h)_{H^1}, \qquad 0\le D^2\Phi(v)[h,h] \le C_pY(v)^{(p-2)/2}\|h\|_{H^1}^2, \\
0\le R_\Phi(v,h) &:=\Phi(v+h)-\Phi(v)-D\Phi(v)[h] \le C_p\left( Y(v)^{(p-2)/2}\|h\|_{H^1}^2+\|h\|_{H^1}^{p} \right).
\end{aligned}
\end{equation}
Consequently, \eqref{eq:2-H0}--\eqref{eq:2-H1} and \eqref{eq:6.3-p-moment} imply
\begin{equation}
\label{eq:6.3-correction-bounds}
\begin{aligned}
\|\Sigma(t,v)\|_{\gamma(\mathcal U;H^1)}^2 +\int_Z\|\mathcal G(t,v,z)\|_{H^1}^2\,\mu(\dd z) &\le C Y(v), \\
\frac12\operatorname{Tr}_{\Sigma(t,v)}D^2\Phi(v) +\int_Z R_\Phi(v,\mathcal G(t,v,z))\,\mu(\dd z) &\le C\Phi(v).
\end{aligned}
\end{equation}
Here $\operatorname{Tr}_B D^2\Phi(v) =\sum_jD^2\Phi(v)[Be_j,Be_j]$ for any orthonormal basis $(e_j)$ of $\mathcal U$.
Fix $T<\infty$.
The regularity already established in Theorem \ref{thm:4.3-H1-persistence} and \eqref{eq:4.3-drift-integrability} give
\begin{equation*}
\Lambda(T):=\int_0^T \left(\|A(u(s))\|_2^2+\|u(s)\|_{H^2}^2\right)\dd s <\infty \quad\P\text{-a.s.}
\end{equation*}
The integrands are assigned the value zero on the $\dd s\otimes\P$-null set where $u(s)\notin H^2$.
The resulting $\Lambda$ is continuous and adapted.
Define
\begin{equation*}
\theta_k:=T\wedge\inf\{t\ge0:Y(t)>k\} \wedge\inf\{t\ge0:\Lambda(t)\ge k\}, \qquad k\ge2.
\end{equation*}
These are stopping times, $\theta_k\uparrow T$, and $\theta_k=T$ eventually almost surely.
On $(0,\theta_k]$, $Y(s-)\le k$, and $\Lambda(\theta_k)\le k$.
All stochastic integrands below are multiplied by the predictable indicator $\mathds 1_{(0,\theta_k]}$.

For completeness, the It\^o formula used below is justified on the existing solution.
Let $J_\varepsilon=e^{\varepsilon\Delta}$ and apply the Hilbert-space It\^o formula to $\Phi(J_\varepsilon u(\cdot\wedge\theta_k))$.
The mollified drift is $H^1$-valued, and its pairing equals
\begin{equation*}
D\Phi(J_\varepsilon u)[J_\varepsilon A(u)] =pY(J_\varepsilon u)^{(p-2)/2} (A(u),(I-\Delta)J_\varepsilon^2u)_2.
\end{equation*}
It converges in $L^1(\Omega\times(0,\theta_k))$ to the corresponding unmollified pairing, by the bounds defining $\theta_k$, strong convergence on $H^2$, and Cauchy--Schwarz.
The Wiener and linear Poisson integrands converge in their square-integrable coefficient spaces.
The Taylor remainders converge in $L^1(\Omega\times(0,\theta_k)\times Z)$ by \eqref{eq:6.3-Lyapunov} and \eqref{eq:6.3-p-moment}.
Their compensated integrals therefore converge in the expected uniform norm.
The same bounds justify convergence of both compensators.
Finally, strong convergence of $J_\varepsilon$ is uniform on the relatively compact range of each stopped $H^1$-c\`adl\`ag path.
Thus
\begin{equation*}
\sup_{t\le T} \|J_\varepsilon u(t\wedge\theta_k) -u(t\wedge\theta_k)\|_{H^1}\longrightarrow0 \quad\P\text{-a.s.}
\end{equation*}
This is convergence in the uniform path topology.
Writing $\Sigma_s=\Sigma(s,u(s-))$ and $G_s(z)=\mathcal G(s,u(s-),z)$, the limiting identity is
\begin{equation}
\label{eq:6.3-localized-Ito}
\begin{aligned}
\Phi(u(t\wedge\theta_k)) =&\Phi(u_0) +p\int_0^{t\wedge\theta_k}Y(s)^{(p-2)/2} (A(u(s)),(I-\Delta)u(s))_2\,\dd s+\frac12\int_0^{t\wedge\theta_k} \operatorname{Tr}_{\Sigma_s}D^2\Phi(u(s-))\,\dd s \\
&+\int_0^{t\wedge\theta_k}\int_Z R_\Phi(u(s-),G_s(z))\,\mu(\dd z)\,\dd s +M^W_k(t)+M^{J,1}_k(t)+M^{J,2}_k(t),
\end{aligned}
\end{equation}
where
\begin{equation}
\label{eq:6.3-moment-martingales}
\begin{aligned}
M^W_k(t) :=\int_0^{t\wedge\theta_k} \bigl(D\Phi(u(s-))\circ\Sigma_s\bigr)\,\dd W_s,& \qquad M^{J,1}_k(t) :=\int_0^{t\wedge\theta_k}\int_Z D\Phi(u(s-))[G_s(z)]\,\widetilde N(\dd s,\dd z), \\
\qquad\qquad\qquad M^{J,2}_k(t) &:=\int_0^{t\wedge\theta_k}\int_Z R_\Phi(u(s-),G_s(z))\,\widetilde N(\dd s,\dd z).
\end{aligned}
\end{equation}
The preceding localization gives integrable absolute drift and correction terms, square-integrable linear martingales, and an integrable supremum of the remainder martingale.
In particular, $\E\sup_{t\le T}\Phi(u(t\wedge\theta_k))<\infty$.
No upper bound on the post-jump value $Y(\theta_k)$ is used.

\emph{Step 2:} Set $S_k(t):=\sup_{r\le t}\Phi(u(r\wedge\theta_k))$.
The Wiener BDG inequality and the Davis inequality for the optional quadratic variation of $M^{J,1}_k$, followed by Young's inequality and the compensator identity, yield
\begin{equation}
\label{eq:6.3-martingale-moment-bounds}
\begin{aligned}
&\E\sup_{r\le t}|M^W_k(r)| +\E\sup_{r\le t}|M^{J,1}_k(r)| \\
&\quad\le\varepsilon\E S_k(t) +C_{\varepsilon,p}\E\int_0^{t\wedge\theta_k} Y(s-)^{(p-2)/2} \left(\|\Sigma_s\|_{\gamma(\mathcal U;H^1)}^2 +\int_Z\|G_s(z)\|_{H^1}^2\,\mu(\dd z)\right)\dd s \\
&\quad\le\varepsilon\E S_k(t) +C_{\varepsilon,p}\int_0^t\E S_k(s)\,\dd s, \\
&\E\sup_{r\le t}|M^{J,2}_k(r)| \le2\E\int_0^{t\wedge\theta_k}\int_Z R_\Phi(u(s-),G_s(z))\,\mu(\dd z)\,\dd s \le C\int_0^t\E S_k(s)\,\dd s.
\end{aligned}
\end{equation}
For the first two terms the factorization $Y^{p-1}=Y^{p/2}Y^{(p-2)/2}$ extracts $S_k(t)$ from the quadratic variation.
For the remainder term, nonnegativity in \eqref{eq:6.3-Lyapunov} gives the displayed $L^1$ estimate directly.
Only second and $p$th moments of the $H^1$-valued jump coefficient are required.

Combining \eqref{eq:6.3-H1-coercivity}, \eqref{eq:6.3-correction-bounds}, and \eqref{eq:6.3-localized-Ito}--\eqref{eq:6.3-martingale-moment-bounds}, and absorbing a fixed sufficiently small $\varepsilon$, gives
\begin{equation}
\label{eq:6.3-unweighted-Gronwall}
\begin{aligned}
\E S_k(t) +c_{p,\nu}\E\int_0^{t\wedge\theta_k} Y(s)^{(p-2)/2}\|u(s)\|_{H^2}^2\,\dd s \le C\E\Phi(u_0) +C\int_0^t\E S_k(s)\,\dd s.
\end{aligned}
\end{equation}
Gronwall's inequality yields a bound independent of $k$.
Monotone convergence as $k\to\infty$ proves \eqref{eq:6.3-higher-H1}.

\emph{Step 3:} Let $\vartheta_p=\frac12-\frac1p$.
Sobolev embedding and interpolation between $H^1$ and $H^2$ give
\begin{equation}
\label{eq:6.3-Sobolev-interpolation}
\begin{aligned}
\|v\|_p&\le C_p\|v\|_{H^1}, \\
\|v\|_{3p}^p &\le C_p\|v\|_{H^{1+\vartheta_p}}^p \le C_p\|v\|_{H^1}^{(p+2)/2} \|v\|_{H^2}^{(p-2)/2}, \qquad v\in H^2.
\end{aligned}
\end{equation}
Moreover, the Gagliardo--Nirenberg inequality $\|v\|_\infty\le C\|v\|_{H^1}^{1/2}\|v\|_{H^2}^{1/2}$ implies
\begin{equation*}
\mathcal D_p(v) \le\|v\|_\infty^{p-2}\|\nabla v\|_2^2 \le C_p\|v\|_{H^1}^{(p+2)/2} \|v\|_{H^2}^{(p-2)/2}.
\end{equation*}
Since $p\le6$ and $\|v\|_{H^1}\le\|v\|_{H^2}$,
\begin{equation*}
\|v\|_{3p}^p+\mathcal D_p(v) \le C_p\|v\|_{H^1}^{p-2}\|v\|_{H^2}^2 \le C_pY(v)^{(p-2)/2}\|v\|_{H^2}^2.
\end{equation*}
Integrating this inequality and applying \eqref{eq:6.3-higher-H1} and \eqref{eq:6.3-Sobolev-interpolation} proves \eqref{eq:6.3-unweighted}.
\end{proof}
To complete global well-posedness, it remains to control the dependence on the initial datum.
We first establish a quantitative difference estimate up to a common $L^p$ exit time and then remove this stopping time to obtain convergence in probability on every finite interval.
For the remainder of the section, set $$ \mathcal A := L^p(\Omega,\mathcal F_0;E) \cap L^2(\Omega,\mathcal F_0; \mathcal H), $$ where $E=L^p_\sigma$ and $\mathcal H=L^2_\sigma$ are as in \eqref{eq:2-solenoidal-spaces}.
For a finite stopping time $\theta$ and a process $w$, write $$ \|w\|_{\mathcal X_\theta}^p := \sup_{0\le s\le\theta}\|w(s)\|_p^p + \int_0^\theta\|w(s)\|_{3p}^p\,\dd s.
$$ Suprema over stopped intervals always include the terminal value.

\begin{theorem}
[Global continuous dependence in the native $L^p$ class] \label{thm:6.4-global-continuous-dependence} Let $p>3$ and retain the structural and stochastic hypotheses of Theorem~\ref{thm:6.2-finite-energy-global}.
Fix the stochastic basis, the coefficient maps, and the driving noises.
\begin{enumerate}
\item[(i)] Let $u$ and $v$ be the global solutions corresponding to $\xi,\zeta\in\mathcal A$.
For $T>0$ and $R\ge1$, set $$ M_u(t):=\sup_{0\le s\le t}\|u(s)\|_p, \qquad M_v(t):=\sup_{0\le s\le t}\|v(s)\|_p, $$ and $$ \theta_R := T\wedge \inf\{t\ge0:M_u(t)\vee M_v(t)\ge R\}, \qquad \inf\varnothing:=\infty .
$$ Then there exists a deterministic constant $C_{R,T}$, depending only on $R,T$ and the fixed structural parameters, such that, for every $A\in\mathcal F_0$,
\begin{equation}
\label{eq:6.4-stopped-stability}
\begin{aligned}
\E\bigg[\mathds1_A\bigg( &\|u-v\|_{\mathcal X_{\theta_R}}^p +\int_0^{\theta_R} \left\|\nabla\bigl(|u(s)-v(s)|^{p/2}\bigr)\right\|_2^2\,\dd s \bigg)\bigg]\le C_{R,T}\, \E\bigl[\mathds1_A\|\xi-\zeta\|_p^p\bigr].
\end{aligned}
\end{equation}

\item[(ii)] Let $u_0,u_0^{(n)}\in\mathcal A$, and let $u,u^{(n)}$ be the corresponding global solutions.
If $$ \|u_0^{(n)}-u_0\|_p\longrightarrow0 \qquad\text{in probability}, $$ then, for every $T<\infty$,
\begin{equation}
\label{eq:6.4-global-continuity}
\sup_{0\le t\le T} \|u^{(n)}(t)-u(t)\|_p^p + \int_0^T \|u^{(n)}(s)-u(s)\|_{3p}^p\,\dd s \longrightarrow0 \qquad\text{in probability}.
\end{equation}
\end{enumerate}

The constant in \eqref{eq:6.4-stopped-stability} is independent of the initial data, the closed localizing sequences, and all scalar and spatial regularization parameters.
In particular, \eqref{eq:6.4-global-continuity} holds whenever $u_0^{(n)}\to u_0$ in $L^p(\Omega;E)$.
No convergence in $L^2(\Omega;\mathcal H)$ is required, although each datum belongs to $\mathcal A$.
\end{theorem}

\begin{proof}
\emph{Step 1:} Set $w=u-v$ and $\mathcal T(a)=g_N(|a|^2)a$.
The running suprema defining $\theta_R$ are adapted and right-continuous, so $\theta_R$ is a stopping time.
On its nonempty integration interval,
\begin{equation*}
\|u(s)\|_p\vee\|v(s)\|_p<R\quad(0\le s<\theta_R),\qquad \|u(s-)\|_p\vee\|v(s-)\|_p\le R\quad(0<s\le\theta_R).
\end{equation*}
The post-jump values at $\theta_R$ need not satisfy this bound.
The difference equation is
\begin{equation*}
\dd w=(\nu\Delta w+\nabla\cdot f_w+h_w)\,\dd s +b_w\,\dd W_s +\int_Z j_w(s,z)\,\widetilde N(\dd s,\dd z), \qquad w(0)=\xi-\zeta,
\end{equation*}
where
\begin{align*}
f_w(s)&=-\mathcal P^{(1)} \bigl(u(s)\otimes u(s)-v(s)\otimes v(s)\bigr),& h_w(s)&=-\mathcal P\bigl(\mathcal T(u(s))-\mathcal T(v(s))\bigr), \\
b_w(s)&=\Sigma(s,u(s-))-\Sigma(s,v(s-)),& j_w(s,z)&=\mathcal G(s,u(s-),z)-\mathcal G(s,v(s-),z).
\end{align*}
All stochastic coefficients are predictable.

Choose $q_f$ and $q_h$ as defined in \eqref{eq:2_q_fq_h_choice} and define
\begin{equation*}
\frac1{\ell_f}=\frac1{q_f}-\frac1p, \qquad \frac1{\ell_h}=\frac1{q_h}-\frac2p.
\end{equation*}
These choices are possible for every $p>3$, and $p<\ell_f,\ell_h<3p$.
The Leray multiplier bounds, H\"older's inequality, and the taming difference estimate give
\begin{equation*}
\|f_w\|_{q_f}\le C_R\|w\|_{\ell_f},\qquad \|h_w\|_{q_h}\le C_R\|w\|_{\ell_h}.
\end{equation*}
Using the weighted Lipschitz hypotheses for the Wiener and jump coefficients, followed by interpolation, we obtain, for almost every time before $\theta_R$,
\begin{align*}
\|f_w\|_{q_f} &\le C_R\|w\|_p^{\beta_f}\|w\|_{3p}^{1-\beta_f},& \|h_w\|_{q_h} &\le C_R\|w\|_p^{\beta_h}\|w\|_{3p}^{1-\beta_h}, \\
\|b_w\|_{\gamma(\mathcal U;E)} &\le C_R\|w\|_p^{1/4}\|w\|_{3p}^{3/4},& \left(\int_Z\|j_w(s,z)\|_p^r\,\mu(\dd z)\right)^{1/r} &\le C_R\|w\|_p^{\beta_G}\|w\|_{3p}^{1-\beta_G}, \qquad r\in\{2,p\},
\end{align*}
where
\begin{equation*}
\beta_f=\frac{3p/q_f-4}{2}\in(0,1),\qquad \beta_h=\frac{3p/q_h-7}{2}\in(0,1),\qquad \beta_G=1-\frac{3\alpha}{2}\in(0,1].
\end{equation*}
For the last two estimates, H\"older's inequality uses the difference norms $L^{2p}$ and $L^{p/(1-\alpha)}$, respectively.
When $\alpha=0$, the last bound is simply $C\|w\|_p$.
The identities $u_-=u$ and $v_-=v$ hold almost everywhere in time; no $L^{3p}$-valued left limits are needed.

Consequently, for every $\varepsilon>0$, Young's inequality gives
\begin{equation*}
\begin{aligned}
&\|f_w\|_{q_f}^p+\|h_w\|_{q_h}^p +\|b_w\|_{\gamma(\mathcal U;E)}^p +\int_Z\|j_w(s,z)\|_p^p\,\mu(\dd z) +\left(\int_Z\|j_w(s,z)\|_p^2\,\mu(\dd z)\right)^{p/2} \le\varepsilon\|w(s)\|_{3p}^p +C_{\varepsilon,R}\|w(s)\|_p^p.
\end{aligned}
\end{equation*}
The quadratic Poisson mode is controlled by
\begin{equation*}
\left(\int_0^{t\wedge\rho}\int_Z\|j_w(s,z)\|_p^2 \,\mu(\dd z)\,\dd s\right)^{p/2} \le T^{p/2-1}\int_0^{t\wedge\rho} \left(\int_Z\|j_w(s,z)\|_p^2\,\mu(\dd z)\right)^{p/2} \,\dd s
\end{equation*}
whenever $t\le T$ and $\rho\le\theta_R$.
To justify expectations, let $(\tau_m^u)$ and $(\tau_m^v)$ be the closed localizing sequences of the two global solutions and set
\begin{equation*}
\rho_m=\theta_R\wedge\tau_m^u\wedge\tau_m^v.
\end{equation*}
On $[0,\rho_m]$, both solutions have finite expected local energy.
The preceding bounds therefore place all coefficient differences in the forcing classes of Theorem \ref{Thm:heat-exis}.
Apply that theorem to the linear equation with initial value $\mathds1_A(\xi-\zeta)$ and forcing coefficients $\mathds1_A\mathds1_{(0,\rho_m]}(f_w,h_w,b_w,j_w)$.
Its solution agrees with $\mathds1_Aw$ through $\rho_m$ and evolves by the heat semigroup afterwards.
Thus the stopped estimate is obtained from a heat extension, with every terminal jump included.
The factor $\mathds1_A$ is admissible because $A\in\mathcal F_0$.

For $0\le t\le T$, put
\begin{equation*}
\Phi_m(t)=\E\left[\mathds1_A \sup_{0\le s\le t\wedge\rho_m}\|w(s)\|_p^p\right], \qquad D_m(t)=\E\left[\mathds1_A \int_0^{t\wedge\rho_m} \left\|\nabla\bigl(|w(s)|^{p/2}\bigr)\right\|_2^2\,\dd s\right].
\end{equation*}
The linear estimate and the preceding coefficient bounds imply
\begin{equation*}
\Phi_m(t)+D_m(t) \le C_T\E[\mathds1_A\|\xi-\zeta\|_p^p] +C_T\varepsilon\E\left[\mathds1_A \int_0^{t\wedge\rho_m}\|w(s)\|_{3p}^p\,\dd s\right] +C_{\varepsilon,R,T}\int_0^t\Phi_m(s)\,\dd s.
\end{equation*}
Here $C_T$ absorbs the factor $1+T^{p/2-1}$ and is independent of $m,A,\xi,\zeta$.
The Sobolev inequality
\begin{equation*}
\|w\|_{3p}^p \le C_p\left\|\nabla\bigl(|w|^{p/2}\bigr)\right\|_2^2
\end{equation*}
allows absorption of the second term after choosing $\varepsilon>0$ sufficiently small.
Gronwall's inequality then gives
\begin{equation*}
\E\left[\mathds1_A\left( \|w\|_{\mathcal X_{\rho_m}}^p +\int_0^{\rho_m} \left\|\nabla\bigl(|w(s)|^{p/2}\bigr)\right\|_2^2\,\dd s \right)\right] \le C_{R,T}\E[\mathds1_A\|\xi-\zeta\|_p^p].
\end{equation*}
Globality implies $\tau_m^u,\tau_m^v\uparrow\infty$ almost surely.
Consequently, $\rho_m=\theta_R$ for all sufficiently large $m$, pathwise.
Fatou's lemma proves \eqref{eq:6.4-stopped-stability}.

\emph{Step 2:} Fix $T>0$ and write $w_n=u^{(n)}-u$.
Let $\theta_R^{(n)}$ be the stopping time in part \textup{(i)} for the pair $(u^{(n)},u)$.
For $\delta>0$, set
\begin{equation*}
A_{n,\delta}=\{\|u_0^{(n)}-u_0\|_p\le\delta\}\in\mathcal F_0.
\end{equation*}
By \eqref{eq:6.4-stopped-stability},
\begin{equation*}
\E\left[\mathds1_{A_{n,\delta}} \|w_n\|_{\mathcal X_{\theta_R^{(n)}}}^p\right] \le C_{R,T}\delta^p.
\end{equation*}
On the event $\{M_u(T)<R/2,\ \theta_R^{(n)}<T\}$, the exit is caused by $u^{(n)}$.
Right-continuity of the running supremum gives $M_{u^{(n)}}(\theta_R^{(n)})\ge R$, and hence
\begin{equation*}
\sup_{0\le s\le\theta_R^{(n)}}\|w_n(s)\|_p\ge R/2.
\end{equation*}
This remains valid when the exit occurs by a jump, because the supremum includes the value at $\theta_R^{(n)}$.
If $\theta_R^{(n)}=T$, the stopped norm is the full norm.
Therefore, for every $\varepsilon>0$, Markov's inequality yields
\begin{equation*}
\P\bigl(\|w_n\|_{\mathcal X_T}>\varepsilon\bigr) \le\P\bigl(M_u(T)\ge R/2\bigr) +\P\bigl(\|u_0^{(n)}-u_0\|_p>\delta\bigr)+C_{R,T}\delta^p \left((2/R)^p+\varepsilon^{-p}\right).
\end{equation*}
First let $n\to\infty$, then $\delta\downarrow0$, with $R$ fixed.
Since $u$ has global $E$-valued c\`adl\`ag paths, $M_u(T)<\infty$ almost surely.
Letting $R\to\infty$ proves \eqref{eq:6.4-global-continuity}.
Convergence in $L^p(\Omega;E)$ implies convergence in probability of the initial data, which proves the final assertion.
\end{proof}

Together, Theorems \ref{thm:6.2-finite-energy-global} and \ref{thm:6.4-global-continuous-dependence} give global existence, pathwise uniqueness, and continuous dependence in probability in the native $L^p$ path and space--time norms for initial data in $\mathcal A$.
In particular, the solution map is continuous from $\mathcal A$, equipped with its intersection norm, to these processes equipped with convergence in probability on every finite time interval.
The uniform path convergence in \eqref{eq:6.4-global-continuity} also implies convergence in probability in the Skorokhod $J_1$ topology.
No unstopped expected $p$th-power stability estimate is asserted.
\begin{remark}
[Almost surely finite initial data] The pathwise global-existence conclusion of Theorem \ref{thm:6.2-finite-energy-global} extends to every strongly $\mathcal F_0$-measurable datum with values in $E\cap\mathcal H$, without moment assumptions on its law.
Here the local-energy uniqueness class is understood after localization on increasing $\mathcal F_0$-measurable events.

Indeed, set $ A_m:=\{\|u_0\|_p+\|u_0\|_2\le m\}, \qquad u_{0,m}:=\mathds1_{A_m}u_0, \qquad m\ge1.
$ Each $u_{0,m}$ belongs to $\mathcal A$.
Let $u_m$ be the global solution supplied by Theorem \ref{thm:6.2-finite-energy-global}.
For $\ell\ge m$, the initial values agree on $A_m$.
Applying \eqref{eq:6.4-stopped-stability} with $A=A_m$ and then exhausting the common exit times and finite horizons gives $u_\ell=u_m$ on $A_m$, indistinguishably.
Since $A_m\uparrow\Omega$ almost surely, the prescription $u=u_m$ on $A_m$ defines an adapted global solution.
Predictability and the equation follow by restriction to these $\mathcal F_0$-measurable events.
On each $A_m$, this solution has the closed local-energy localizations of $u_m$; event-localized pathwise uniqueness therefore gives uniqueness in the stated localized class.
All pathwise regularity conclusions of Theorem \ref{thm:6.2-finite-energy-global} are inherited from the branches.
The estimates in expectation retain their original moment hypotheses.
\end{remark}
\section*{CRediT authorship contribution statement}
Bikram Podder : Conceptualization, Formal analysis, Investigation, Methodology, Writing- original draft.

Surendra Kumar : Supervision, Validation, Writing- review and editing.
\section*{ Declaration of competing interest }
The authors declare that they have no known competing financial interest or personal relationship that could have appeared to influence the work reported in this paper.
\section*{ Data availability}
No data has been used to prepare this manuscript.
\section*{Acknowledgment}
The first author would like to express sincere gratitude to International Center for Theoretical Sciences (ICTS).
This research was supported in part by ICTS for participating in the program - A School on Gradient Flows 2026 (code: ICTS/SGFLOWS2026/08)

\bibliographystyle{elsarticle-num} \bibliography{References}
\end{document}